\documentclass[a4paper]{article}
\usepackage{latexsym,amsfonts, amsmath, amssymb, amsthm}
\usepackage[utf8]{inputenc}
\usepackage[final=true,protrusion=true,expansion=true]{microtype}
\usepackage[noadjust]{cite}
\usepackage{enumitem}
\usepackage{color}
\usepackage{empheq}
\usepackage[dvipsnames,svgnames]{xcolor}
\usepackage{mathtools}
\usepackage{bm}
\usepackage{adjustbox}

\usepackage{booktabs}
\usepackage{multirow}

\usepackage{subcaption}
\usepackage[font=scriptsize,labelfont=bf]{caption}
\usepackage[affil-it]{authblk}

\usepackage{geometry}
\usepackage{hyperref}
\hypersetup{
	final,
    colorlinks=true,
    linkcolor=blue,
    filecolor=magenta,
    urlcolor=cyan,
}
\usepackage[nameinlink]{cleveref}
\crefname{subsection}{subsection}{subsections}

\usepackage{stmaryrd}

\usepackage[normalem]{ulem}

\theoremstyle{plain}
\newtheorem{theorem}{Theorem}[section]
\newtheorem{proposition}[theorem]{Proposition}
\newtheorem{lemma}[theorem]{Lemma}
\newtheorem{corollary}[theorem]{Corollary}

\theoremstyle{plain}
\newtheorem{definition}[theorem]{Definition}
\newtheorem{example}[theorem]{Example}
\newtheorem{remark}[theorem]{Remark}
\newtheorem{assumption}[theorem]{Assumption}

\newcommand{\dummy}{\mathord{\color{black!33}\bullet}}

\providecommand{\argmin}{\operatorname*{arg\,min}}  % argument yielding inf
\providecommand{\CB}{{\cal B}}
\providecommand{\CC}{{\cal C}}

\providecommand{\CE}{{\cal E}}

\providecommand{\CI}{{\cal I}}

\providecommand{\CN}{{\cal N}}
\providecommand{\CO}{{\cal O}}
\providecommand{\CP}{{\cal P}}

\providecommand{\CV}{{\cal V}}
\providecommand{\CW}{{\cal W}}

\providecommand{\bbE}{\mathbb{E}}

\providecommand{\bbN}{\mathbb{N}}

\providecommand{\bbP}{\mathbb{P}}

\providecommand{\bbR}{\mathbb{R}}

\providecommand*{\N}[1]{\left\|{#1}\right\|} % Double bar norm
\providecommand*{\Nnormal}[1]{\|{#1}\|} % Double bar norm
\providecommand*{\Nbig}[1]{\big\|{#1}\big\|} % Double bar norm
\newcommand*{\SN}[1]{\left|{#1}\right|}      % Single bar norm
\usepackage{bbm} % Black board math
\usepackage{ifdraft}
\usepackage{arydshln}
\usepackage{subcaption}

\usepackage{mdframed}
\usepackage{xcolor}
\usepackage{soul}

\usepackage{tikz}
\usepackage{scalerel}
\usepackage{tikz-cd}
\usetikzlibrary{arrows,intersections,calc,patterns,angles,quotes}
\usetikzlibrary{shapes.geometric}
\usetikzlibrary{calc}
\usetikzlibrary{shapes,arrows}

\usepackage[textsize=footnotesize]{todonotes}

\newcommand{\tn}{\textnormal}

\newcommand{\overbar}[1]{\mkern 1.5mu\overline{\mkern-1.5mu#1\mkern-1.5mu}\mkern 1.5mu}

\newcommand{\overbarsuperscriptt}[1]{\mkern 1.5mu\overline{\mkern-1.5muX_{#1}\mkern-4mu}\mkern4mu}
\newcommand{\overbarsuperscriptit}[1]{\mkern -1.5mu{\mkern 1.5mu{\mkern-1.5mu\overline{X}^i_{#1}\mkern-1.5mu}\mkern1.5mu}}

\newcommand{\hatsuperscriptit}[1]{\mkern -1.5mu{\mkern 1.5mu{\mkern-1.5mu\widehat{X}^i_{#1}\mkern-1.5mu}\mkern1.5mu}}

\DeclarePairedDelimiter\abs{\lvert}{\rvert}%
\renewcommand{\abs}[1]{\left|{#1}\right|}
\DeclareMathOperator*{\Law}{Law}

\newcommand{\dt}[0]{\Delta t}

\def\R{\mathbb{R}}
\def\RD{\mathbb{R}^D}
\def\Rdl{\mathbb{R}^{d_{\ell}}}
\def\Rd{\mathbb{R}^d}
\def\Rn{\mathbb{R}^n}
\def\B{B} % may be later replaced by \def\B{B^{\infty,}}
\def\intrd{\mathbf{d}}

\newcommand{\Xbar}[1]{\overbarsuperscriptt #1}
\newcommand{\Xibar}[1]{\overbarsuperscriptit #1}

\newcommand{\Xihat}[1]{\hatsuperscriptit {#1}} % for EM-CBO

\newcommand{\globmin}{x^*}

\newcommand{\indivmeasure}[0]{\varrho} %\widetilde\rho already used

\newcommand{\empmeasure}[1]{\widehat\rho_{#1}^N}

\newcommand{\omegaa}[0]{\omega_{\alpha}}

\newcommand{\conspoint}[1]{x_{\alpha}({#1})}
\newcommand{\conspointcoordinate}[1]{x_{\alpha}^{\ell}({#1})}

\def\Vk{\mathcal{V}_k}
\def\CEunder{\underline{\CE}}
\def\CEupper{\overbar{\CE}}

\def\qlk{q_{\ell_k}}
\def\rlk{r_{\ell_k}}

\def\alphazerol{\alpha_{0,\ell}}
\def\qle{q_{\ell,\varepsilon}}
\def\rle{r_{\ell,\varepsilon}}
\def\ple{p_{\ell,\varepsilon}}

\newcommand{\cna}[0]{C_{\tn{NA}}} 
\newcommand{\cmfa}[0]{C_{\tn{MFA}}}
\newcommand{\acc}[0]{\varepsilon_{\mathrm{total}}}

\makeatother

\title{\usefont{OT1}{bch}{b}{n}
	\LARGE Exploiting Structure with\\Anisotropic Consensus-Based Optimization\\
}
\author[Sabrina Bonandin, Konstantin Riedl and Sara Veneruso]{}
\date{}

\begin{document}

\maketitle

% Authors, emails and affiliations
\vspace{-1.8cm}
%\centerline{Sabrina Bonandin$^{*1}$, Konstantin Riedl$^{\dagger 2}$ and Sara Veneruso$^{\ddagger 1,3}$}
\centerline{Sabrina Bonandin$^{1}$, Konstantin Riedl$^{2}$ and Sara Veneruso$^{1,3}$}

\bigskip

{\footnotesize
% Enter the full affiliation and country name:
% Do not insert commas or periods at the end of lines.
 \centerline{$^1$Institute for Geometry and Practical Mathematics, RWTH Aachen University,}
 \centerline{Templergraben 55, 52062 Aachen, Germany}
} 

\medskip

{\footnotesize
% Enter the full affiliation and country name:
% Do not insert commas or periods at the end of lines.
 \centerline{$^2$Mathematical Institute, University of Oxford,}
 \centerline{Radcliffe Observatory, Andrew Wiles Building, Woodstock Rd, Oxford OX2 6GG, United Kingdom}
} 

\medskip

{\footnotesize
 % Enter the full affiliation and country name:
  \centerline{$^3$Department of Mathematics and Computer Science, University of Ferrara,}
   \centerline{Via Machiavelli 30, 44121 Ferrara, Italy}
}

\medskip
{\footnotesize
 % Enter email addresses:
 % \centerline{Email addresses: $^*$\texttt{bonandin@eddy.rwth-aachen.de}, $^\dagger \texttt{konstantin.riedl@maths.ox.ac.uk}$,}
 % \centerline{$^\ddagger$\texttt{sara.veneruso@unife.it}}
 \centerline{Email addresses: \href{mailto:bonandin@eddy.rwth-aachen.de}{\texttt{bonandin@eddy.rwth-aachen.de}}, \href{mailto:konstantin.riedl@maths.ox.ac.uk}{\texttt{konstantin.riedl@maths.ox.ac.uk}},}
 \centerline{\href{mailto:sara.veneruso@unife.it}{\texttt{sara.veneruso@unife.it}}
}

\bigskip

% \author[1]{Sabrina Bonandin\thanks{Email: \texttt{bonandin@eddy.rwth-aachen.de}}}
% \author[2]{Konstantin Riedl\thanks{Email: \texttt{Konstantin.Riedl@maths.ox.ac.uk}}}
% \author[3]{Sara Veneruso\thanks{Email: \texttt{sara.veneruso@unife.it}}}

% \affil[1]{Institute for Geometry and Practical Mathematics, RWTH Aachen University, Templergraben 55, 52062 Aachen, Germany}
% \affil[2]{Mathematical Institute, University of Oxford, Radcliffe Observatory, Andrew Wiles Building, Woodstock Rd, Oxford OX2 6GG, United Kingdom}
% \affil[3]{Department of Mathematics and Computer Science, University of Ferrara, Via Machiavelli 30, 44121 Ferrara, Italy}
% \renewcommand\Authands{ and }

%%%%%%%%%%%%%%%%%%%%%%%%%%%%%%%%%%%%%%%%%%%%%%%%%%
%%%%%%%%%% Abstract %%%%%%%%%%%%%%%%%%%%%%%%%%%%%%
%%%%%%%%%%%%%%%%%%%%%%%%%%%%%%%%%%%%%%%%%%%%%%%%%%
\begin{abstract}
\noindent
Anisotropic consensus‑based optimization (CBO), a multi-agent
metaheuristic derivative-free optimization method, which reliably finds global minima of nonsmooth and nonconvex objective functions while being amenable to a rigorous theoretical analysis, automatically detects and exploits additively separable structures of high-dimensional objective functions.
This enables the algorithm to provably mitigate the curse of dimensionality in scenarios where the objective function decomposes additively into lower-dimensional components.
In this paper, we show the latter property by proving that the computational complexity of anisotropic CBO depends exponentially only on the intrinsic dimension $\intrd$ of the objective function, rather than the ambient dimension $D\gg \intrd$.
Additionally, we demonstrate that rather than depending on a tractability condition of the full energy landscape, the computational complexity depends only on tractability conditions of the lower-dimensional components, 
which allows for a more refined description of the objective function and algorithmic complexity, as the objective function landscape is captured directly on the level of the individual components.
Our results highlight the effectiveness of anisotropic CBO for additively separable objective functions provided that there is sufficient alignment between the structure of the anisotropic noise in the algorithm and the separability structure of the objective itself.
This motivates the design of an enhanced algorithm that learns during the optimization how to most effectively explore the loss landscape by aligning on the fly the noise with the structure of the objective function, which we leave for future research.
Numerical experiments validate our theoretical results accentuating in particular the influence 
of the intrinsic dimensionality $\intrd$ and the level of separability as well as the level of complexity and non-convexity of the objective function within the separable components 
on the performance and computational complexity of the anisotropic CBO algorithm.
\end{abstract}

%%%%%%%%%% Keywords %%%%%%%%%%%%%%%%%%%%%%%%%%%%%%
{\small\noindent{\textbf{Keywords:} high-dimensional optimization, global optimization, derivative-free optimization, additively separable functions, nonsmoothness, nonconvexity, metaheuristics, consensus-based optimization, mean-field limit, Fokker-Planck equations, anisotropic diffusion}}\\

{\small\noindent{\textbf{AMS subject classifications:} 65K10, 90C26, 90C56, 35Q90, 35Q84}}

%\tableofcontents

\normalsize
%%%%%%%%%%%%%%%%%%%%%%%%%%%%%%%%%%%%%%%%%%%%%%%%%%
%%%%%%%%%% Section %%%%%%%%%%%%%%%%%%%%%%%%%%%%%%%
%%%%%%%%%%%%%%%%%%%%%%%%%%%%%%%%%%%%%%%%%%%%%%%%%%
\section{Introduction}
\label{sec:intro}
High-dimensional nonconvex global optimization problems inherently suffer from the curse of dimensionality \cite{bellman1957dynamic,bottou2018optimization}
since the volume of the search space and hence the computational complexity of the optimization problem grows exponentially with the ambient dimension.
At the same time, the no-free-lunch theorem~\cite{wolpert2005coevolutionary}, a well-known cornerstone principle of optimization theory, states that no single algorithm can be superior across all---in theory---possible objective functions.
In particular, without prior structural assumptions, this implies that every optimization algorithm performs, once their performance is averaged across all problems, identical to random guessing.

This result, however, has to be put into perspective, since not all function classes are equally relevant in practical applications,
therefore demonstrating the necessity of exploiting prior information and structure of the objective function landscape: In order to be able to break this worst-case algorithmic barrier and achieve scalable convergence guarantees, an optimization algorithm must explicitly and effectively exploit the intrinsic structure and underlying geometry of the problem class.
% randomised subspace methods

In this work, we reveal how anisotropic consensus-based optimization (CBO)~\cite{carrillo2019consensus,fornasier2021convergence,riedl2024perspective}, a versatile and powerful multi-agent metaheuristic derivative-free optimization method, which has attracted significant recent research interest due to its empirical successes~\cite{carrillo2019consensus, zhang2026proxicbo,borghi2023constrained,beddrich2024constrained,chen2020consensus} and its provable ability~\cite{bonandin2025strong, fornasier2021anisotropic,carrillo2019consensus,riedl2024perspective} to globally minimize nonconvex nonsmooth objective functions,
exploits intrinsic lower-dimensional structure of objective functions.
To do so, we consider the class of additively separable functions $\CE:\RD \to \R$,
where the $D$-dimensional objective~$\CE$ can be decomposed into a sum of $L$ lower-dimensional objective functions~$\CE_\ell:\Rdl \to \R$, $\ell=1,\dots,L$, with disjoint coordinates and $d_\ell\ll D$.
More precisely, an additively separable objective function~$\CE$ has a structure as defined in Assumption~\ref{ass:add_sep}.

\begin{assumption}[additive separability]
\label{ass:add_sep}
The objective function~$\CE:\RD\rightarrow\R$ can be decomposed into a sum
\begin{equation}
    \label{eq:add_sep}
    \CE(x)
    = \sum_{{\ell}=1}^L \CE_{\ell}(\CB^\top_{\ell}x)
\end{equation}
of $L \in \mathbb{N}$ lower-dimensional components~$\CE_{\ell}:\Rdl \to \R$, with $\sum_{{\ell}=1}^L d_{\ell} = D$ and $1 \le d_{\ell} \le D$ for all $\ell \in \{1,\dots,L\}$,
where $\CB_{\ell}^\top:\RD \to \Rdl$ denote suitable coordinate selector maps with $\CB_{\ell}^\top x = x|_{\CI_{\ell}} \in \Rdl$, and $\CI_1, \ldots, \CI_L$ being a partition of $\{1,\dots,D\}$.
\end{assumption}
\noindent Note that $P_{\ell}:\RD \to \RD$ with $P_{\ell} = \CB_{\ell} \CB_{\ell}^\top$ denotes the orthoprojector onto the $d_\ell$-dimensional subspace of $\RD$ spanned by the coordinates in $\CI_{\ell}$, and, by construction, $\#\CI_{\ell} = d_{\ell}$ for any $\ell$, and $1 \le L \le D$. 

This property of additive separability arises naturally in various application areas in science and engineering, whenever systems exhibit a certain amount of modularity, which is typical in practical settings.
Examples where such structures appear range from statistical regression models, generalized additive models~\cite{hastie1986generalized}, and  analysis of variance decomposition (ANOVA)~\cite{wahba1990spline, vapnik1998statistical} in statistics
% see paper \cite{duvenaud2011additive} on Additive gaussian processes for references
to decoupled physical potentials in molecular dynamics~\cite{gandolfi2023molecular}.
%A similar condition is used in \cite[Equation~2.1]{shirokoff2025convergence}
Furthermore, in several problems arising in machine learning, 
including sparse support vector machines, matrix completion, and  graph cuts, 
the objective function~$\CE$ splits into approximately independent components where a overlap over separate parameter clusters is rare \cite[Eq.~(1)]{recht2011hogwild}.
%approximation results of this class of functions constructive approximation / scattered data a approximation: \cite{rieger2024approximability}
% information-theoretic complexity: \cite{krieg2026approximation}
However, let us emphasize that there are certainly equally many problems leading to objective functions not fulfilling Assumption~\ref{ass:add_sep}.
For concrete illustrative examples of functions with this property, we refer the reader to Example~\ref{ex:examples_separability}.

Assumption~\ref{ass:add_sep} naturally gives rise to a new notion of dimensionality, which we call the \emph{intrinsic dimension} and which corresponds to 
\begin{equation}
    \label{def:intrinsic_dim}
    \intrd \coloneqq \max_{\ell \in \{1,\dots,L\}} d_{\ell},
\end{equation}
in contrast to the dimensionality of the search space $D\gg \intrd$, henceforth referred to as the \emph{ambient dimension}.

Throughout this manuscript,
we consider the task of solving the global unconstrained optimization problem 
\begin{equation}
    \label{eq:min_prob}
    \globmin = \argmin_{x\in \RD} \CE(x)
    \quad \text{ with } \quad
    \CEunder \coloneqq \CE(\globmin),
\end{equation}
for a potentially nonconvex nonsmooth and high-dimensional objective function $\CE$ obeying Assumption~\ref{ass:add_sep} on the structure of the objective function landscape.
The optimizer $\globmin $ is assumed to exists and be unique.
The focus of our analysis is on the class of derivative‑free (zero‑order) optimization methods~\cite{conn09intro}, which have recently been shown to be particularly effective in high‑dimensional and multimodal landscapes, such as the one encountered in problem~\eqref{eq:min_prob}, in particular, since in such settings gradient‑based methods may fail to detect globally optimal solutions due to the presence of numerous local minima, a lack of regularity, or unreliable derivative information \cite{back1997handbook,blum2003metaheuristics}.
The family of zero-order optimization methods includes notable examples such as particle‑swarm optimization \cite{kennedy1995particle}, simulated annealing \cite{aarts1989simulated}, genetic algorithms \cite{reeves2010genetic}, as well as consensus‑based optimization (CBO), which constitutes the central optimization method studied in the present work.
Introduced in \cite{pinnau2017consensus}, CBO has emerged as a mathematically grounded alternative and representative within this class of schemes, 
owing to its amenability to rigorous analysis using tools from statistical physics.
CBO methods employ a finite number of \(N\in\bbN^+\) particles (also referred to as agents) to explore the landscape of the objective function \(\CE\).  
Their dynamics combine two complementary mechanisms: on the one hand, a deterministic exploitation process that drives the particles toward regions of the search space containing high‑quality approximations of the solution, and, on the other, a stochastic exploration feature that injects randomness, enabling the particles to explore admissible regions and to escape from local minima or saddle points.
Over time, the particles progressively collapse around a so-called consensus point, which serves as a reliable approximation of the global minimizer $\globmin$ in \eqref{eq:min_prob}.
The update rule of the CBO algorithm consists of a system of $N$ coupled discrete-time equations (see equation~\eqref{eq:aCBO_micro_EM} below). Nevertheless, the analysis of the method typically involves a continuous-time approximation \eqref{eq:aCBO_micro} where the $N$ coupled discrete-time equations are replaced by a system of $N$ coupled continuous-time stochastic differential equations (SDEs) in the limit of vanishing step size, see, e.g., \cite{pinnau2017consensus,carrillo2018analytical}.
An explicit rate of convergence in the step size together with a holistic convergence analysis in expectation was recently proved in~\cite{bonandin2025strong},
and an approach that does not employ such an approximation can be found in~\cite{borghi2024mean}.
A global convergence analysis of CBO can be carried out either in the finite-particle regime (also named the microscopic level), as done, for instance, in \cite{bellavia2025discrete,ha2020convergenceHD,ha2021convergence}, or at the designated mean‑field level, where the number of agents \(N\) is let tend to infinity, and the dynamics are described statistically in terms of the average behavior of a typical agent, as done in \cite{pinnau2017consensus, carrillo2018analytical, fornasier2024consensus, riedl2024perspective,huang2025faithful}, amongst others.  
The relationship between these two regimes has been subject of intensive recent research, and quantitative mean‑field results bridging the finite‑particle and infinite‑particle settings have been established, e.g., in \cite{fornasier2020consensus_hypersurface_wellposedness, fornasier2024consensus, gerber2025mean, gerber2025uniform, kalise2022consensus}.

The adaptability, simplicity, and ease of use~\cite{bailo2024cbx} of CBO have driven its increasing popularity for tackling diverse optimization challenges. A concise overview of selected variants is provided in Subsection~\ref{sec:intro_literatureCBO}.

Let us now present the numerical formulation of the CBO method.
For a fixed number of iterations $H \in \bbN^+$ and a time step size $\dt>0$,
we denote by $\Xihat{h \dt} \in \RD$ the position of agent $i \in \{1,\ldots,N\}$ at time $h \dt$, for $h \in \{ 0, \ldots, H \}$.
For user-specified parameters $\alpha,\lambda,\sigma>0$, the time-discrete evolution of the $i$th agent is defined by the update rule 
\begin{equation}
    \label{eq:aCBO_micro_EM}
        \Xihat{(h+1)\dt} - \Xihat{h\dt} = -\lambda  \big(\Xihat{h\dt} - \conspoint{\empmeasure{h\dt}}\big) \dt + \sigma \sum_{k=1}^D \big(\Xihat{h\dt} - \conspoint{\empmeasure{h\dt}}\big)_k \big(B^{(i)}_{h\dt}\big)_k e_k^D, 
\end{equation}
where $\{ B^{(i)}_{h\dt} \}^{i \in \{1,\dots,N\}}_{h \in \{ 0, \ldots, H \}}$ are independent and identically distributed (i.i.d.) Gaussian random vectors in $\RD$ with zero mean and covariance matrix $\dt I_D$ (in short, $B^{(i)}_{h\dt} \sim \CN_D(0, \dt I_D)$). % compact notation  that will be used in the rest of the manuscript 
The system is complemented with independent initial conditions $\{ \Xihat{0} \}^{i \in \{1,\dots,N\}}$ distributed according to a common law $\rho_0 \in \CP(\RD)$. 
The quantity $\empmeasure{\dummy} \coloneqq \frac{1}{N} \sum_{i=1}^N \delta_{\Xihat{\dummy}}$ is the 
empirical measure associated with the particle system $\{ \Xihat{\dummy} \}^{i \in \{1,\ldots,N\}}_{\dummy}$, and 
$x_{\alpha}$ denotes the consensus point, which is given for any probability distribution $\indivmeasure$ on $\RD$ by
\begin{equation}
    \label{eq:conspoint}
    \conspoint{\indivmeasure} \coloneqq \int_{\RD} x \frac{\omegaa(x)}{\N{\omegaa}_{L^1(\indivmeasure)}} \,d\indivmeasure(x)
    \quad \text{ with }\quad
    \omegaa(x) \coloneqq \exp(-\alpha\CE(x)).
\end{equation}
It can be shown that, by invoking the classical Laplace principle \cite{dembo2009large,miller2006applied}, $\conspoint{\indivmeasure}$ yields an accurate approximation of \(\globmin\) as
\begin{equation*}
    \lim_{\alpha \to \infty} \left( -\frac{1}{\alpha} \log \left( \int_{\RD} \omegaa(x) d\indivmeasure(x) \right) \right) = \inf_{x \in \text{supp}(\indivmeasure)} \CE(x).
\end{equation*}
Equation~\eqref{eq:aCBO_micro_EM} captures the two aforementioned competing mechanisms of exploitation and exploration.  
The exploitation component is governed by the parameter \(\lambda\) and drives each agent toward the weighted point \(x_{\alpha}\).  
The exploration component is controlled by \(\sigma\). In the CBO literature, such stochastic exploration is typically modeled in one of two ways: as isotropic noise, where all coordinates are perturbed equally (as in the initial work~\cite{pinnau2017consensus}), or as  
anisotropic noise (as proposed in \cite{carrillo2021consensus}) to improve the scalability and performance of CBO in high‑dimensional settings.
The latter is the type employed in Equation~\eqref{eq:aCBO_micro_EM}.

\subsection{Contributions and main result}
\label{sec:intro_contributions}

To explain the recent successes of CBO with anisotropic noise in tackling high‑dimensional optimization problems,
we investigate in this paper how the anisotropic CBO algorithm mitigates the curse of dimensionality by exploiting intrinsic structure of the objective function in form of additive separability.
Our computational complexity analysis identifies two key structural ingredients that anisotropic CBO is able to exploit.

\begin{enumerate}
    \item A structured energy landscape $\CE$, which is additively separable as specified by Assumption \ref{ass:add_sep}, with a lower intrinsic dimensionality $\intrd$ (as defined in \eqref{def:intrinsic_dim}). This allows the algorithm to automatically separate the optimization problem into its lower-dimensional components, see Figure~\ref{fig:heatmaps_cbo_ANISO8D_intro}.
    \item Component-wise specified geometries of the energy landscapes around the respective components of the global minimizer $\globmin$. This allows a more finely tuned selection of hyperparameters, enabled by our analysis which captures differently complex objective function landscapes in different components, see Figure~\ref{fig:heatmaps_cbo_ANISO4D_cD_intro}.
\end{enumerate}

Let us now present an informal, intuitive formulation of our main result, whose fully rigorous version is deferred to Theorem~\ref{thm:conv_dlg1_Vk_micro} and Corollary~\ref{thm:conv_dlg1_Vk_micro_epstot}, which can be found in Subsection~\ref{sec:mainresults_convmicro}.

\begin{theorem}
\label{thm:conv_dlg1_Vk_micro_nonrigorous}
Let $\CE:\RD\rightarrow\R$ be additively separable as in Assumption~\ref{ass:add_sep}, fulfill Assumption~\ref{ass:well-posedness_E}, and let $\CE_\ell$ satisfy Assumption~\ref{ass:invcont_QLPl2} for all $\ell \in \{ 1, \ldots, L \}$.
Let $\rho_0 \in \CP_6(\Rd)$ be such that $\globmin \in \tn{supp}(\rho_0)$. Moreover, suppose that the measures $\rho^\ell_0\coloneqq\CB_{\ell}^\top\#\rho_0$, $\ell \in \{ 1, \ldots, L \}$, are statistically independent.
Let $\acc > 0$ be a prescribed accuracy and fix $\delta_1 \in (0,1-(1/2)^{1/L})$ and $\delta_2 \in (0,1/2)$.
Then, there exist parameter choices $\alpha_0$, $N$, $\dt$ and $H$ such that for all $\alpha > \alpha_0$ it holds that
\begin{equation*}
        \N{\frac{1}{N} \sum_{i=1}^N \widehat{X}^{i}_{H\dt} - \globmin}_{\infty,D} \le \acc
\end{equation*}
with probability at least $(1-\delta_1)^L-\delta_2$, 
% presentation of process after so that I have introduced delta t and H
where $\{\Xihat{h \dt}\}_{h\in \{0,\ldots,H\}}^{i\in\{1,\ldots,N\}}$ denote 
the iterates generated by the CBO algorithm \eqref{eq:aCBO_micro_EM} initialized according to $\Xihat{0}\sim\rho_0$.
This can be accomplished by
     \begin{enumerate}
     %[label=(\roman*)]
        \item[(i)] choosing $\alpha_0 = \max_{\ell \in \{1,\ldots,L\}} \alphazerol >0$ with $\alphazerol=\alphazerol(\varepsilon^{-1}) >0$ as specified in~\eqref{eq:alphazerol_simplified}, where $\varepsilon^{-1}\propto\intrd/(\acc^2\,\delta_1)$;\\ (Hence, $\alphazerol$ depends on the intrinsic dimension $\intrd$ and the constants of Assumption \ref{ass:invcont_QLPl2},  
        but is independent of both $D$ and the parameters appearing in the regularity assumptions imposed on the energy functional $\CE$.)
        \item[(ii)] fixing
        $N$ and $\dt$, so that $N$ grows polynomially as 
        $\cmfa/(\acc^{2}\,\delta_2)$,
        while $\dt$ scales linearly with  $\acc^2\,\delta_2/\cna$, and selecting $H = \left\lceil \frac{T}{\dt}\right\rceil$ for a suitable time horizon $T>0$ dependent on $\rho_0$ and $\intrd/(\acc^2\,\delta_1)$.\\ (Here, $\cna,\cmfa>0$ are positive constants that depend only on the quantities appearing in Assumption~\ref{ass:well-posedness_E}, on $\rho_0$, and exponentially on $\alpha$ and $H\dt$. Moreover, $\cna$ depends linearly on the ambient dimension $D$.)
    \end{enumerate}

\end{theorem}

\begin{remark}[Scaling of $\alpha_0$]
\label{rem:crucial_remark}
In order to fully appreciate how our main result, Theorem~\ref{thm:conv_dlg1_Vk_micro_nonrigorous}, exploits the separability of the objective function from Assumption~\ref{ass:add_sep} to mitigate the curse of dimensionality, let us now provide a more detailed description of the structure of $\alpha_0=\max_{\ell \in \{1,\ldots,L\}} \alphazerol$ and the dependencies of $\alphazerol$. 
Proceeding as in \cite[Remark~4.8]{fornasier2024consensus}, we consider a simplified setting that yields an informative expression for $\alphazerol$.
As noted in the cited work, extracting general quantitative information from the explicit formula of $\alphazerol$ in \eqref{def:alpha0l} is considerably more challenging.
Hence, let us assume that $\CE_\ell$ is locally $\Lambda_\ell$-Lipschitz continuous in a $d_\ell$-dimensional ball $\B^{d_\ell}_{R_\ell}(\CB^\top_{\ell}\globmin)$ for some radius $R_\ell>0$, and that there exists a constant $\eta_\ell>0$ such that 
\begin{equation*}
    \N{v-\CB_\ell^\top \globmin}_{\infty,d_\ell} \le \frac{1}{\eta_\ell} \left( \CE_\ell(v)-\CEunder_\ell \right)^{1/2} \quad \text{for all } v \in \Rdl, \quad \text{with } \; \CEunder_{\ell} \coloneqq \inf_{v \in \Rdl}\CE_{\ell}(v).
\end{equation*}
The above condition corresponds to Assumption~\ref{ass:invcont_QLPl2} satisfied globally with exponent $1/2$, and 
ensures that $\CE_\ell$ grows proportionally to $v \mapsto \N{v-\CB^\top_{\ell}\globmin}_{\infty,d_\ell}^2$ with the parameter $\eta_\ell$ regulating the rate of growth (the smaller $\eta_\ell$, the faster the function grows).
Under these two assumptions on $\CE_\ell$, the form of $\alphazerol$ reduces to 
\begin{equation}
\label{eq:alphazerol_simplified}
    \alphazerol = \frac{-8}{c^2\eta_\ell^2\varepsilon} \log \left( c \rho_0^\ell( \B^{d_\ell}_{\min\{R_\ell,c^2\eta_\ell^2\varepsilon/(8\Lambda_\ell)\}}(\CB^\top_{\ell}\globmin)) \right)
\end{equation}
for some constant $c = c(\lambda,\sigma)$ and with $\varepsilon^{-1}\propto\intrd/(\acc^2\,\delta_1)$ (see~\eqref{eq:def_varepsilon} in Corollary~\ref{thm:conv_dlg1_Vk_micro_epstot} for its rigorous definition).
The expression of $\alphazerol$ in~\eqref{eq:alphazerol_simplified} clarifies how anisotropic CBO  leverages the intrinsic structure of the optimization problem, as summarized right before Theorem~\ref{thm:conv_dlg1_Vk_micro_nonrigorous}.
\begin{enumerate}
    \item First, under the assumption that $\rho_0^\ell$ is equivalent to the Lebesgue measure, it holds $\alphazerol \sim \intrd d_\ell \log(\intrd)$.
    To see this, note that the logarithmic term in \eqref{eq:alphazerol_simplified} behaves as $\log(\rho_0^\ell(B^{d_\ell}_{\varepsilon})) = \log(\varepsilon^{d_\ell})=d_\ell \log(\varepsilon)$, where $\varepsilon^{-1}\propto\intrd/(\acc^2\,\delta_1)$. Taking the maximum over $\ell \in \{1,\ldots,L\}$ shows that $\alpha_0$ behaves as $\CO(\intrd^2 \log(\intrd))$, i.e., grows with the intrinsic dimension $\intrd$, but remains independent of the ambient dimension $D$.
    Hence, Theorem~\ref{thm:conv_dlg1_Vk_micro_nonrigorous} exploits the lower dimensional structure associated with the structured energy landscape $\CE$.
    We refer to Figure~\ref{fig:heatmaps_cbo_ANISO8D_intro} for numerical experiments capturing this behavior.
    \item Second, it holds $\alphazerol \propto d_{\ell} \frac{1}{\eta_\ell^2} \max \left\{ \log\left(\frac{1}{R_\ell}\right),\log\left(\frac{\Lambda_\ell}{\eta_\ell^2}\right)\right\}$. Taking the maximum over $\ell\in\{1,\dots,L\}$ shows that $\alpha_0$ is determined by the geometric properties of the objective function landscapes of the individual components together with their respective dimensionality. To be precise, it is determined by the worst among the components, where $\alphazerol$ is largest. Hence, Theorem~\ref{thm:conv_dlg1_Vk_micro_nonrigorous} exploits component-wise specified geometries of the energy landscape of the components \(\CE_{\ell}\) rather than the full high‑dimensional objective function $\CE$, leveraging the underlying structure on the local level, thereby complementing the first item.
    We refer to Figure~\ref{fig:heatmaps_cbo_ANISO4D_cD_intro} for numerical experiments capturing this behavior.
\end{enumerate}
\end{remark}

Theorem~\ref{thm:conv_dlg1_Vk_micro_nonrigorous} estimates the $\ell^{\infty}$-norm between the final iterate of the implementable CBO scheme~\eqref{eq:aCBO_micro_EM} and the global minimizer $\globmin$ in high probability.
The bound holds for any hyperparameter $\alpha>\alpha_{0}$, where $\alpha_0$ depends in particular on the intrinsic dimensionality $\intrd$ and the constants capturing the tractability properties of the energy landscapes of the individual component functions $\CE_\ell$, $\ell \in \{1,\ldots,L\}$, appearing in the assumptions on $\CE_\ell$, as highlighted in Remark~\ref{rem:crucial_remark}.
Furthermore, the bound allows to derive explicit choices of the required number of particles $N$, the time step size $\Delta t$, and the total number of iterations $H$.

To the best of our knowledge, this work is the first to show that anisotropic CBO provably mitigates the curse of dimensionality by exploiting certain structures of high-dimensional objective functions.
\begin{itemize}
\item 
    \cite[Theorem~2]{fornasier2021convergence} and 
    \cite[Theorem~3.6]{riedl2024perspective} present 
    global convergence results for the mean-field formulation~\eqref{eq:aCBO_mf} of \eqref{eq:aCBO_micro_EM} for general objective functions $\CE$, both for isotropic and anisotropic noise, respectively.
    As we later discuss in Remarks~\ref{rem:implications_ICP} and~\ref{rem:proof_Vk_dlg1_specialcases}, such analysis corresponds to our special scenario $L=1$, or, equivalently, $\intrd=D$ (see also Remark \ref{rem:notablechoicesL}). In this setting, $\alpha_0$ grows with the ambient dimension $D$. Consequently, optimization with CBO is subject to the curse of dimensionality in the absence of structural assumptions on $\CE$. 
    The aforementioned convergence results are valid under a global tractability condition on $\CE$. As we will show later in Lemma~\ref{lem:assEl_onE_ICP_general} of Subsection~\ref{sec:relEandEl_relICP}, the tractability conditions imposed on $\CE_\ell$ imply those on $\CE$ with worst‑case constants. Consequently, even if the objective $\CE$ behaves well in every component except one, this favorable behavior is not captured by the globally coarser conditions on $\CE$ (see also the discussion preceding Lemma~\ref{lem:assEl_onE_ICP_general} and Example~\ref{ex:ICP_worstcase}). Our Theorem~\ref{thm:conv_dlg1_Vk_micro_nonrigorous} therefore provides a richer and more precise description when a priori knowledge about the structure of the energy landscape is available.
    \item In \cite[Theorem 1.1]{bonandin2025strong} and~\cite[Corollary 1.2]{bonandin2025strong}, strong mean‑square convergence of the practical, time‑discrete CBO algorithm to the global minimizer was established for both isotropic and anisotropic noise, with explicit convergence rates in the time step size $\dt$ and the number of particles $N$. Building on this result, our Theorem~\ref{thm:conv_dlg1_Vk_micro} bounds in high probability the $\ell^{\infty}$-error under more specific assumptions on the objective function and for better-suited hyperparameters that exploit the intrinsic low‑dimensional structure from Assumption~\ref{ass:add_sep}.
    Our Corollary~\ref{thm:conv_dlg1_Vk_micro_epstot} provides practical guidelines for choosing the hyperparameters $\alpha$, $N$, $\dt$ and $H$ so that a prescribed total accuracy $\acc$ is achieved when a priori structural information about the cost function is available. 
\end{itemize}

These theoretical insights are supported by numerical evidence.
We depict success heatmaps for different values of the hyperparameter $\alpha$ and the required number of particles $N$ when the anisotropic CBO algorithm \eqref{eq:aCBO_micro_EM} is used to minimize an $8$-dimensional ($D=8$) function whose separability deteriorates progressively from left to right (Figure~\ref{fig:heatmaps_cbo_ANISO8D_intro}) or whose level of non-convexity, and hence complexity, increases progressively in a prescribed component from left to right (Figure~\ref{fig:heatmaps_cbo_ANISO4D_cD_intro}).

In Figure~\ref{fig:heatmaps_cbo_ANISO8D_intro}, we observe a degradation in the success rates as the function becomes less separable (decrease of $L$ from $L=8$ to $L=2$, or, equivalently, increase in $\intrd$ from $\intrd=1$ to $\intrd=7$), indicating that anisotropic CBO, which operates along the canonical directions of the Cartesian coordinates, is most effective when the separability structure of the objective is aligned with them.
    \begin{figure}[htbp]
        \centering
        % Subfigure 1
        \begin{subfigure}[b]{0.3\textwidth}
            \centering
            \includegraphics[width=\textwidth]{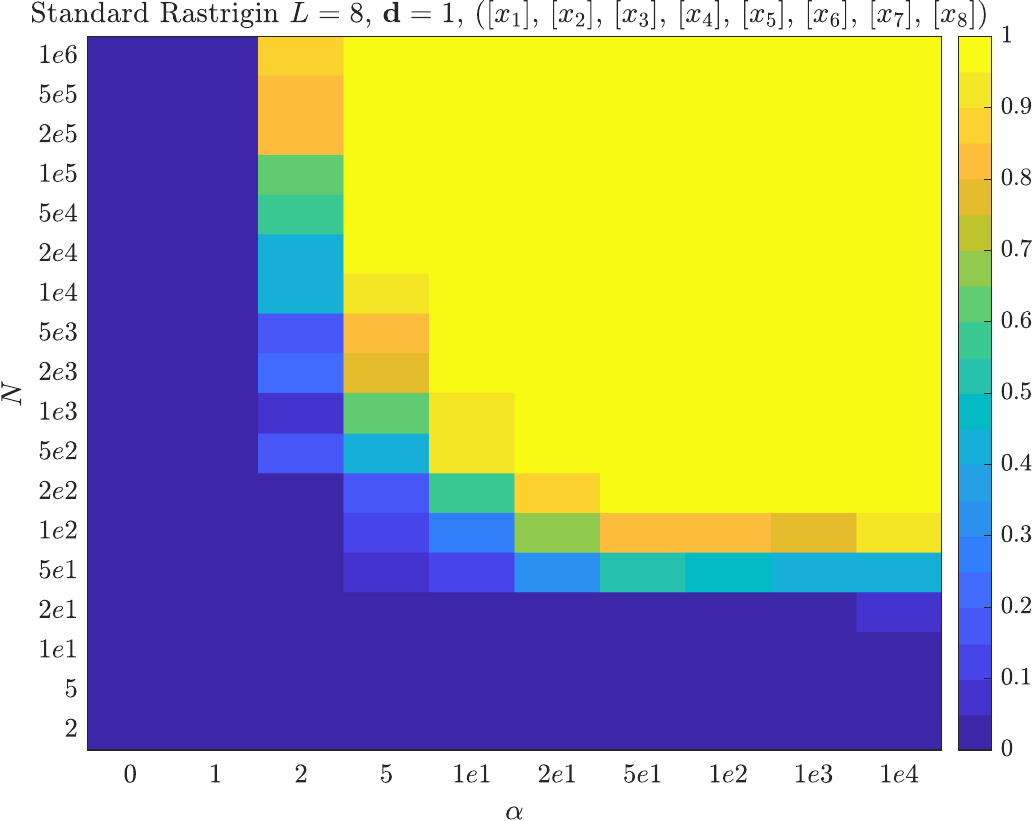}
            \caption{$\mathbf{d}=1$}
            \label{fig:heatmap_rastr8D_anisod1_intro}
        \end{subfigure}
        \hfill
        % Subfigure 2 
        \begin{subfigure}[b]{0.3\textwidth}
            \centering
            \includegraphics[width=\textwidth]{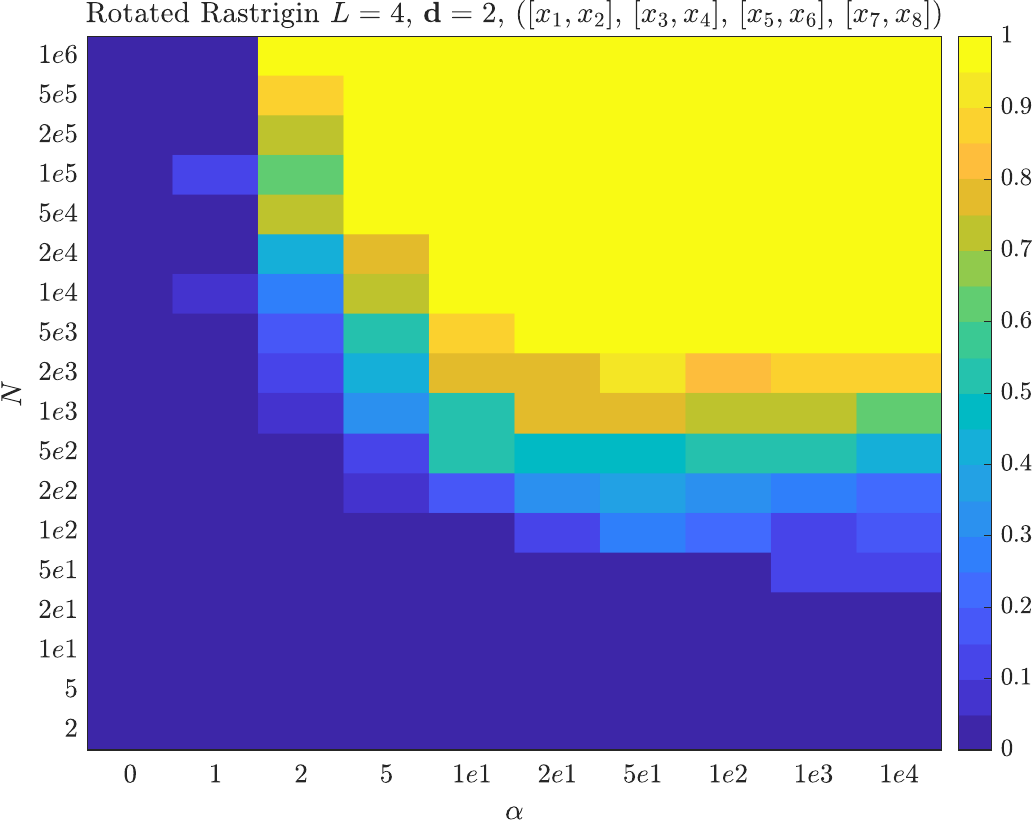}
            \caption{$\mathbf{d}=2$.}
            \label{fig:heatmap_rastr8D_anisod2_intro}
        \end{subfigure}
        \hfill
        % Subfigure 3 
        \begin{subfigure}[b]{0.3\textwidth}
            \centering
            \includegraphics[width=\textwidth]{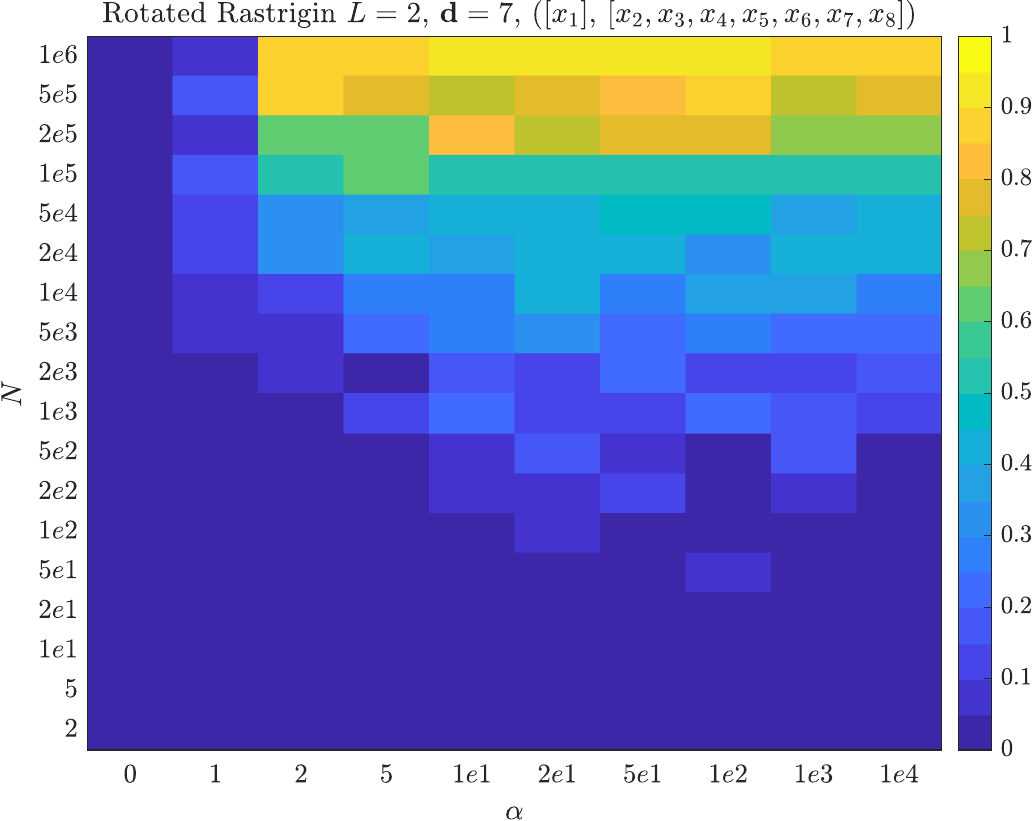}
            \caption{$\mathbf{d}=7$.}
            \label{fig:heatmap_rastr8D_anisod7_intro}
        \end{subfigure}
    
        \caption{The performance of anisotropic CBO progressively deteriorates as the complexity of the objective function (here, the intrinsic dimensionality $\intrd$) increases (from left to right), see Figure \ref{fig:heatmaps_cbo_ANISO8D} for more details.}
        \label{fig:heatmaps_cbo_ANISO8D_intro}
    \end{figure}
    
In Figure~\ref{fig:heatmaps_cbo_ANISO4D_cD_intro}, we observe that an increase in the complexity of the local geometry of the individual components~$\CE_\ell$ (increase of the constants $\{(c_D)_k\}_{k \in \{1,\ldots,D\}}$ that quantify the magnitude of the oscillatory component of the function) leads to a decrease in the success rates of the anisotropic CBO algorithm. 
    \begin{figure}[htbp]
        \centering
        % Subfigure 1
        \begin{subfigure}[b]{0.3\textwidth}
            \centering
            \includegraphics[width=\textwidth]{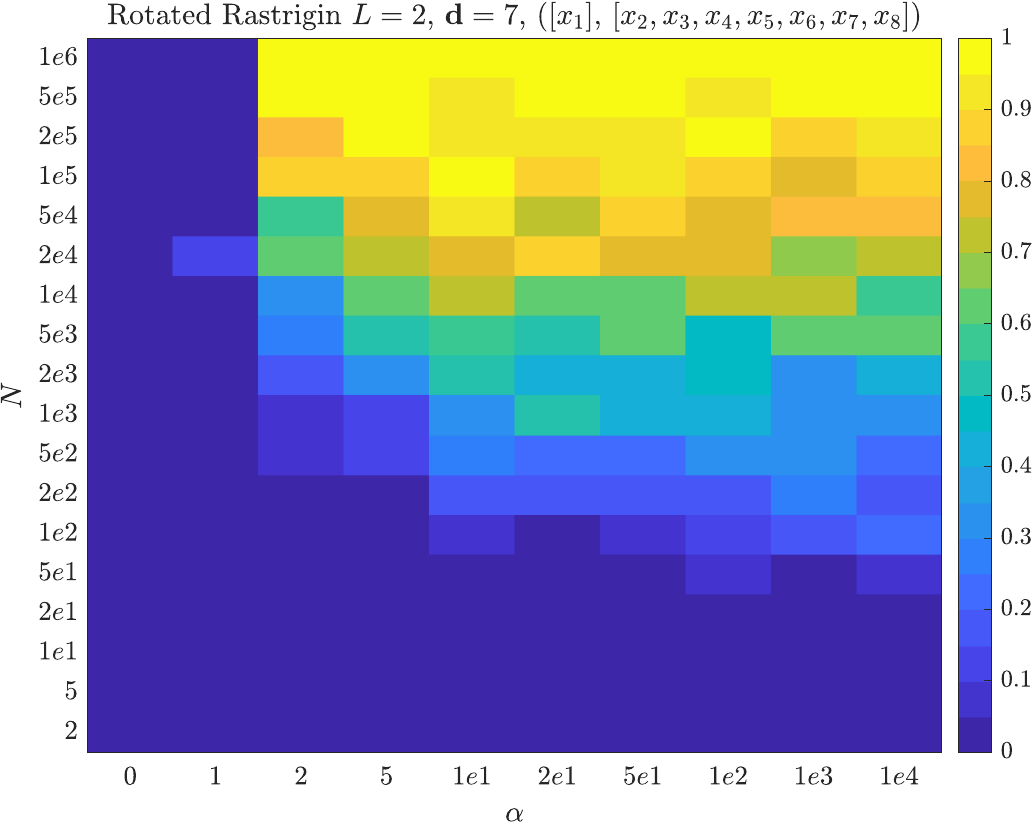}
            \caption{$\mathbf{d}=7$ with $(c_D)_1=2.5$ and $(c_D)_2=\cdots=(c_D)_8=1.5$.}
            \label{}
        \end{subfigure}
        \hfill
        % Subfigure 2 
        \begin{subfigure}[b]{0.3\textwidth}
            \centering
            \includegraphics[width=\textwidth]{Figures/rotatedD8d7.pdf}
            \caption{$\mathbf{d}=7$ with $(c_D)_1=2.5$ and $(c_D)_2=\cdots=(c_D)_8=2.5$.}
            \label{}
        \end{subfigure}
        \hfill
        % Subfigure 3 
        \begin{subfigure}[b]{0.3\textwidth}
            \centering
            \includegraphics[width=\textwidth]{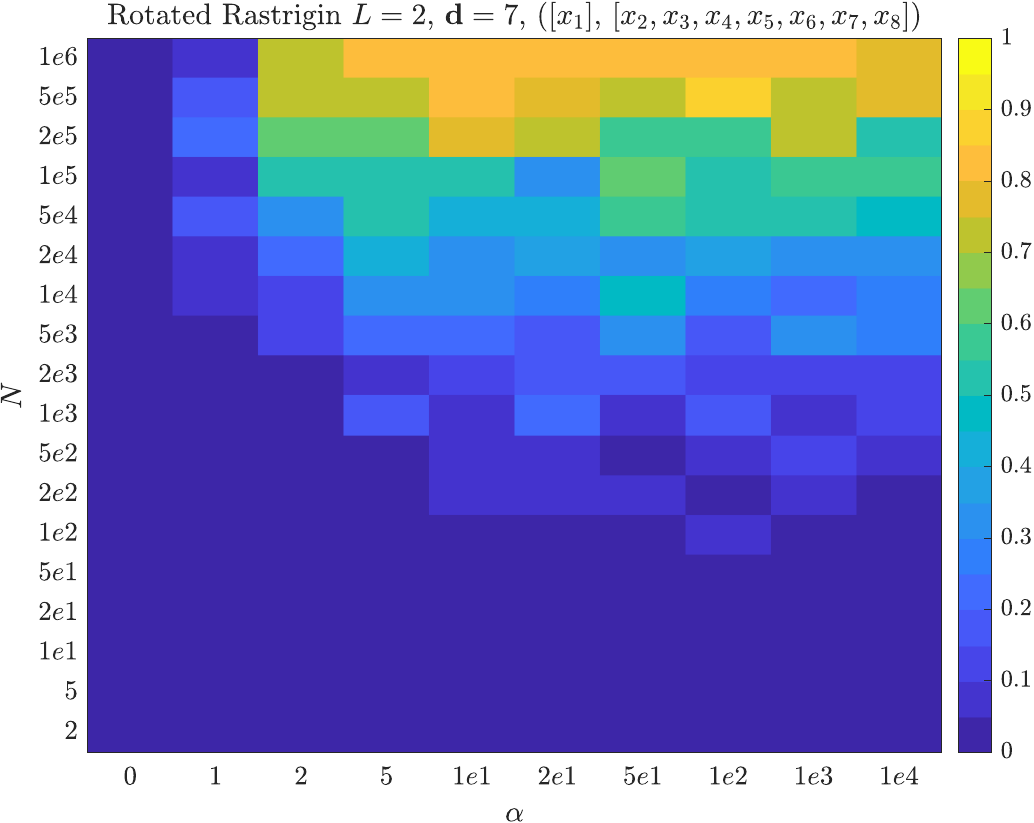}
            \caption{$\mathbf{d}=7$ with $(c_D)_1=10$ and $(c_D)_2=\cdots=(c_D)_8=2.5$.}
            \label{fig:heatmap_rastr8D_anisod7_intro_cD_caseC}
        \end{subfigure}
    
        \caption{The performance of anisotropic CBO progressively deteriorates as the complexity of the objective function (here, the non-convexity within the separable components) increases (from left to right), see Figure \ref{fig:heatmaps_cbo_ANISO8D_cD} for more details.}
\label{fig:heatmaps_cbo_ANISO4D_cD_intro}
    \end{figure}

%%%%%%%%%%%%%%%%%%%%%%%%%%%%%%%%%%%%%%%%%%%%%%%%%%%%%%%%%%
%%%%%%%%%%%%%%%%%%%%%%%%%%%%%%%%%%%%%%%%%%%%%%%%%%%%%%%%%%
% \newpage
\subsection{Organization}

The remainder of this paper is organized as follows.
After a literature review in Section~\ref{sec:intro_literatureCBO},
we discuss in depth Assumption~\ref{ass:add_sep} of additive separability in Section~\ref{sec:relEandEl}.
%introduce the hypotheses used for our theoretical analysis, and examine how they adapt to the additional structure imposed on the objective function.
Section~\ref{sec:mainresults} then presents a  rigorous formulation for our main result, so far presented only in an informal manner in Theorem~\ref{thm:conv_dlg1_Vk_micro_nonrigorous}.
This section concludes with a detailed discussion and a proof.
Several auxiliary results and their proofs are deferred to Section~\ref{sec:proofmainresults}. 
In Section~\ref{sec:numerics}, we provide compelling numerical experiments that verify and support our theoretical findings,
before Section~\ref{sec:conclusion} concludes the manuscript.

%%%%%%%%%%%%%%%%%%%%%%%%%%%%%%%%%%%%%%%%%%%%%%%%%%%%%%%%%%
%%%%%%%%%%%%%%%%%%%%%%%%%%%%%%%%%%%%%%%%%%%%%%%%%%%%%%%%%%
\subsection{Extensions and applications of consensus-based optimization}
\label{sec:intro_literatureCBO}

Before moving onto the technical details of our paper,
let us give an overview of relevant CBO variants and their applications,
illustrating the method's adaptability and breadth.
A comprehensive literature review, however, lies beyond the scope of this manuscript.

The conceptual design and fundamental dynamical principle of CBO are rooted in multi-agent metaheuristic derivative-free optimization algorithms,
in particular closely resembling the renowned particle‑swarm optimization framework \cite{grassi2020particle,huang2022global,herty2025micromacro} while being inspired by mechanisms of evolution strategies \cite{riedl2023all,fornasier2026consensus}.
Surprisingly, it can be observed that, through its consensus‑driven drift term, CBO exhibits a stochastic gradient descent-like behavior despite solely relying on zero-order evaluations of the objective function \cite{riedl2023all}.

Numerous extensions of this method have been developed to tackle a broad spectrum of optimization tasks, such as constrained optimization \cite{borghi2023constrained,herty2025micromacro,fornasier2020consensus_hypersurface_wellposedness,fornasier2020consensus_sphere_convergence,beddrich2024constrained}, multi-objective  \cite{borghi2022adaptive,klamroth2022consensus,borghi2022consensus} and
hierarchical or multi-level formulations \cite{herty2025multiscale,trillos2024cb,chao2026convergence},
sparse optimization and mirror descent \cite{bungert2025mirrorcbo}, 
min-max or saddle point problems \cite{borghi2024particle,huang2024consensus}, optimization problems involving stochastic or noisy objectives \cite{bonandin2025consensus,bonandin2025kinetic,bellavia2025discrete} as well as multiple global minima \cite{bungert2022polarized,fornasier2024polarized}.
Recent work has also focused on formulating the method so that it can be applied to infinite‑dimensional optimization tasks \cite{borghi2023model,borghi2025dynamics,khatab2026consensus,CHSV,huang2026derivative}.
Several adaptations of the original CBO algorithm have been introduced to accommodate more intricate dynamics. Among the enhancements, we mention formulations incorporating memory components \cite{huang2025self,riedl2022leveraging,totzeck2020personal}, momentum terms \cite{chen2020consensus}, second‑order dynamics \cite{byeon2025secondorder,cipriani2021zero,herty2025micromacro,grassi2023mean}, 
gradient information \cite{riedl2022leveraging,zhang2026proxicbo}, % (Riedl and ProxiCBO)
optimal control mechanisms \cite{huang2026fast}, truncated‑noise perturbations \cite{fornasier2023consensus}, jump processes \cite{kalise2022consensus,aceves2026consensus}, and superlinear drift \cite{franceschi2025superlinear}.

Ultimately, the method has been employed in a broad spectrum of domains, such as compressed sensing \cite{riedl2022leveraging} and phase retrieval \cite{fornasier2020consensus_sphere_convergence}, robust subspace detection \cite{fornasier2020consensus_sphere_convergence}, the training of neural networks \cite{carrillo2019consensus,fornasier2021convergence,riedl2022leveraging,borghi2023consensus,de2025mean}, the optimization of qubit configurations \cite{de2025consensus}, robotics \cite{sun2026consensus}, (clustered) federated learning problems \cite{carrillo2024fedcbo,trillos2024attack} and adversarial training \cite{trillos2024attack,roith2025consensus}, as well as multi‑agent game‑theoretic frameworks \cite{chenchene2025consensus}.

Beyond that, technical tools developed in the context of CBO analysis have been used recently to study concentration phenomena of mean-field transformers~\cite{alcalde2026quantifying}.

A code package for the CBO method and its variants is available at \cite{bailo2024cbx}.

%%%%%%%%%%%%%%%%%%%%%%%%%%%%%%%%%%%%%%%%%%%%%%%%%%%%%%%%%%
%%%%%%%%%%%%%%%%%%%%%%%%%%%%%%%%%%%%%%%%%%%%%%%%%%%%%%%%%%
\subsection{Notation}
\label{sec:notation}

We denote by $D$ the ambient dimension of the optimization problem,
whereas the $d_\ell$'s are the dimensions of the domains of the respective component functions~$\CE_\ell$, $\ell \in \{1,\dots,L\}$, as specified in Assumption \ref{ass:add_sep}.
The intrinsic dimension $\intrd$, defined as the largest of the $d_\ell$'s, is given in \eqref{def:intrinsic_dim}. 
When formulating a general statement, which will be applied later in the setting of some problem-specific dimension, we write $d$.

% We set $\setnumber{J}\coloneqq\{1,\ldots,N\}$, for any positive natural number $J >1$.
We let $\bbN^+ \coloneqq \{1,2,\ldots\}$. % and $\bbR^+$ be the set of positive real numbers.
The $k$-th vector of the standard basis in $\bbR^d$ is abbreviated by $e_k^d$. 
The Euclidean and infinity norms of a vector $u \in \Rd$ are denoted by $\N{u}_{2,d}$ and $\N{u}_{\infty,d}$ respectively, whereas the absolute value of a real number $w \in \R$ is denoted by $\SN{w}$.
Throughout the manuscript, only $\ell^{\infty}$-balls will be considered and denoted by $\B^{d}_{r}(u) \coloneqq \{y \in \bbR^d: \N{y-u}_{\infty,d} \leq r\}$.
For the space of continuous functions~$f:X\rightarrow Y$ we write $\CC(X,Y)$, with $X \subset \Rn$ and $Y$ a suitable topological space. 
In the case of real-valued functions we omit $Y$. $\nabla$ indicates the gradient operator of a function on $\Rd$.

Given a random variable $X$, $\text{Law}(X)$ and $\bbE(X)$  denote its law and expectation, respectively.
The set of probability measures over a metric space $E$ is denoted by $\CP(E)$, and the set of probability measures with finite moments up to order $1\leq p<\infty$ are collected in $\CP_p(E) \subset \CP(E)$.
When discussing a particular fixed distribution, we write~$\indivmeasure$.
The \mbox{Wasserstein-$p$} distance~$W_p$ between two probability measures~$\indivmeasure_1,\indivmeasure_2\in\CP_p(\bbR^d)$ is defined by
\begin{align*}
W_p(\indivmeasure_1,\indivmeasure_2) = \left(\inf_{\pi\in\Pi(\indivmeasure_1,\indivmeasure_2)}\int\N{x_1-x_2}_2^pd\pi(x_1,x_2)\right)^{1/p},
\end{align*}
where $\Pi(\indivmeasure_1,\indivmeasure_2)$ denotes the set of all couplings of $\indivmeasure_1$ and $\indivmeasure_2$, i.e., the collection of all probability measures over $\mathbb{R}^d\times\mathbb{R}^d$ with marginals $\indivmeasure_1$ and $\indivmeasure_2$ on the first and second component, respectively (see, e.g., \cite{savare2008gradientflows}).
This manuscript mainly focuses on laws of stochastic processes, $\rho=(\rho_t)_{t\in[0,T]}\in\CC([0,T],\CP(\bbR^d))$. The term $\rho_t\in\CP(\bbR^d)$ refers to a specific instance of this law at time~$t$.

Finally, $\CN_d(\mu,\Sigma)$ denotes the $d$-dimensional multivariate normal distribution with mean vector $\mu \in \Rd$ and covariance matrix $\Sigma \in \R^{d \times d}$.

%%%%%%%%%%%%%%%%%%%%%%%%%%%%%%%%%%%%%%%%%%%%%%%%%%
%%%%%%%%%% Section %%%%%%%%%%%%%%%%%%%%%%%%%%%%%%%
%%%%%%%%%%%%%%%%%%%%%%%%%%%%%%%%%%%%%%%%%%%%%%%%%%
\section{Assumptions on the objective \texorpdfstring{$\CE$}{} and its components \texorpdfstring{$\CE_\ell$}{}}
\label{sec:relEandEl}
In this section,
we discuss the assumptions on the additively separable objective function
\begin{equation*}
    \CE(x)
    = \sum_{{\ell}=1}^L \CE_{\ell}(\CB^\top_{\ell}x)
\end{equation*}
that will be used throughout this manuscript.
In Section~\ref{sec:add_sep_asm}, we elaborate in more depth on the assumption of additive separability, i.e., Assumption~\ref{ass:add_sep}.
Thereafter, in Sections~\ref{sec:relEandEl_relwp} and \ref{sec:relEandEl_relICP}, respectively, we introduce the regularity and tractability assumptions required for our theoretical analysis.
We in particular investigate the relation between such assumptions on the objective $\CE$ itself, and its components~$\CE_\ell$.

\subsection{Additive separability assumption}
\label{sec:add_sep_asm}

Assumption~\ref{ass:add_sep} is a structural assumption on the objective function~$\CE$. Under this assumption, the $D$-dimensional optimization problem~\eqref{eq:min_prob} decomposes into $L$ lower-dimensional $d_{\ell}$-dimensional optimization problems, as made rigorous in the following lemma.

\begin{lemma}
    \label{lem:rel_Eunder_Elunder}
    Let $ \CEunder_{\ell} \coloneqq \inf_{v \in \Rdl}\CE_{\ell}(v)$ for any $\ell \in \{1,\dots,L\}$. 
    If $\CE$ is additively separable as in Assumption~\ref{ass:add_sep},
    it holds 
    \begin{equation}
        \label{eq:rel_globmin_Elunder}
        \CB^\top_\ell \globmin = \argmin_{v\in \Rdl} \CE_\ell(v)
        \quad \text{ with } \quad
        \CEunder_\ell = \CE_\ell (\CB^\top_\ell \globmin),
    \end{equation}
    and, in particular,
    $\CEunder = \sum_{\ell=1}^L\CEunder_{\ell}$.
\end{lemma}
\begin{proof}
    Since $\globmin$ is a global minimizer of $\CE$ according to \eqref{eq:min_prob}, we have $\CE(\globmin)=\inf_{x\in\RD}\CE(x)$.
    With Assumption~\ref{ass:add_sep}, it holds for all $x \in \RD$ that 
    $\CE(\globmin)=\sum_{\ell=1}^L \CE_\ell(\CB^\top_\ell \globmin) \le \CE(x)$.
    We first show that, for every $\ell\in\{1,\dots,L\}$, $\CB_\ell^\top \globmin$ minimizes $\CE_\ell$.
    Fix $\bar\ell\in\{1,\dots,L\}$.
    Suppose, by contradiction, that there exists $\bar v\in\R^{d_{\bar\ell}}$ such that $\CE_{\bar\ell}(\bar v)<\CE_{\bar\ell}(\CB_{\bar\ell}^\top \globmin)$.
    Because the coordinate selector maps $\CB_\ell^\top$ act on disjoint coordinate sets $\CI_{\ell}$,
    we can define a vector $\bar x\in\RD$ such that $\CB_{\bar\ell}^\top \bar x=\bar v$ but $\CB_\ell^\top \bar x=\CB_\ell^\top \globmin$ for all $\ell\neq \bar\ell$.
    Hence,
    \begin{equation}
        \CE(\bar x)
        = \sum_{\ell\neq \bar\ell}\CE_\ell(\CB_\ell^\top \bar x) + \CE_{\bar\ell}(\CB_{\bar\ell}^\top \bar x)
        = \sum_{\ell\neq \bar\ell}\CE_\ell(\CB_\ell^\top \globmin) + \CE_{\bar\ell}(\bar v)
        < \sum_{\ell\neq \bar\ell}\CE_\ell(\CB_\ell^\top \globmin) + \CE_{\bar\ell}(\CB_{\bar\ell}^\top \globmin)
        = \CE(\globmin),
    \end{equation}
    which contradicts the optimality of $\globmin$.
    Therefore, $\CB_\ell^\top \globmin$ minimizes $\CE_\ell$ for every $\ell\in\{1,\dots,L\}$, and thus $\CEunder_\ell=\inf_{v\in\R^{d_\ell}}\CE_\ell(v)=\CE_\ell(\CB_\ell^\top \globmin)$.
    Summing over $\ell\in\{1,\dots,L\}$ and using that $\CE$ is additively separable gives $\CEunder = \CE(\globmin) = \sum_{\ell=1}^L \CE_\ell(\CB_\ell^\top \globmin) = \sum_{\ell=1}^L \CEunder_\ell$,
    which concludes the proof.
\end{proof}

\begin{remark}
\label{rem:why_intrd}
To fully capture the structure of a function $\CE$ that is additively separable as in Assumption~\ref{ass:add_sep},
we require information about the dimensionalities of all components, i.e., $\{d_1,\dots,d_L\}$.
The tuple $(D,L)$ is not able to provide this information to sufficient extent.
Just imagine an example, where $D$ is large and $L=2$. Both, $\{d_1=1,d_2=D-1\}$ and $\{d_1=D/2,d_2=D/2\}$ would be characterized by the same tuple, while their functions having substantially different complexities.
The intrinsic dimension $\intrd$ introduced in \eqref{def:intrinsic_dim}, on the other hand,
naturally captures the complexity of such functions.
\end{remark}

\begin{remark}
    \label{rem:notablechoicesL}
    There are two trivial cases in which the additive separability property in \eqref{eq:add_sep} simplifies.
    They correspond to two distinct choices of $L$ and $\intrd$.
    (i) If $L=1$, and hence $\intrd = D$, then there is only one component, i.e., $d_{1} = \intrd = D$ with $\CB_{1}^\top:\RD \to \RD$ being the identity map.
    This setting corresponds to the treatment of general objective functions $\CE$ without any additive separability assumption.
    (ii) On the other side, if $d_\ell=1$ for all $\ell \in \{1,\dots,L\}$, hence $\intrd = 1$ and $L=D$, then each coordinate direction is its own component, i.e., $\CB_{\ell}^\top:\RD \to \R$ coincides with $(e^D_{\ell})^\top$ where $e^D_{\ell} \in \RD$ denotes the $\ell$-th vector from the standard basis of $\RD$. In this case, $\CE(x) = \sum_{\ell=1}^D \CE_\ell(x_\ell)$ is the sum of $L=D$ one-dimensional functions $\CE_\ell:\R\rightarrow\R$.
\end{remark}

\begin{example}
\label{ex:examples_separability}
    Let us provide a few examples to make the additive separability property in Assumption~\ref{ass:add_sep} tangible.
    (i) The objective $\CE(x)=\big(\!\sum_{k=1}^D x_k\big)^2$ is not additively separable due to the presence of cross-terms. In general, functions of the form $\CE(x)=\N{Rx}_{2,D}^2$ are not additively separable. In fact, if $R\in\R^{D\times D}$ is a random Gaussian matrix, $\CE$ is not additively separable with probability $1$.
    However, for certain block-structures of $R$, certain levels of additive separability are achieved.
    (ii) In particular, the $D$-dimensional parabola $\CE(x)=\N{x}_{2,D}^2$ is additively separable with $L=D$ and hence $d_\ell=1$ for all $\ell \in \{1,\dots,L\}$. The same holds for the classical $D$-dimensional Rastrigin function~$\CE(x) = \sum_{k=1}^{D} \left( x_k^2 - c_D \cos(2\pi x_k) + c_D \right)$, where the constant $c_D$ quantifies the amplitude of the oscillatory component of the function.
\end{example}

%%%%%%%%%%%%%%%%%%%%%%%%%%%%%%%%%%%%%%%%%%%%%%
\subsection{Regularity assumptions}
\label{sec:relEandEl_relwp}

Let us now state regularity conditions on the objective function $\CE$ which are standard requirements for proving the well‑posedness of the interacting particle system, the mean‑field CBO system~\cite{carrillo2018analytical,fornasier2024consensus} as well as quantitative mean‑field estimates~\cite{carrillo2018analytical,carrillo2021consensus,fornasier2024consensus,gerber2025mean}.
They involve boundedness, Lipschitz continuity, and growth conditions on $\CE$.

\begin{assumption}
    \label{ass:well-posedness_E}
    The objective function 
    $\CE \in \CC(\RD)$ 
    \begin{enumerate}[label=W\arabic*,labelsep=10pt,leftmargin=35pt]
        \item\label{ass:well-posedness_E_lowerbound} 
        is lower bounded, i.e., 
        $\CEunder \coloneqq \inf_{x \in \RD}\CE(x) >- \infty$,
    \item\label{ass:well-posedness_E_lipschitz} 
    satisfies for some constants $C_u, c_u > 0$ the conditions
	\begin{align*}
		\abs{{\CE(x)-\CE(x')}} 
        &\leq C_u \left(1+\N{x}_{2,D} + \N{x'}_{2,D}\right) \N{x-x'}_{2,D}
        \quad \text{for all } x,x' \in \RD,\\
		\CE(x) - \CEunder
        &\leq    c_{u} \left(1+\N{x}_{2,D}^2\right)
        \quad \text{for all } x \in \RD,
	\end{align*}
    \item\label{ass:well-posedness_E_growth} is either upper bounded, i.e., $\CEupper \coloneqq \sup_{x \in \RD}\CE(x) < \infty,$ or satisfies for some constants $  c_{l},   C_{l} > 0$ the assumption
	\begin{align*} \label{eq:quadratic_growth_condition_car}
		\CE(x) - \CEunder \geq   c_{l} \N{x}_{2,D}^2 \quad \text{for all } \N{x}_{2,D} \geq C_{l}.
	\end{align*}
    \end{enumerate}
\end{assumption}

Analogously, we can state regularity conditions on the objective functions $\CE_\ell$ that involve boundedness, Lipschitz continuity, and growth conditions on $\CE_\ell$.

\begin{assumption}
    \label{ass:well-posedness_El}
    The objective function 
    $\CE_{\ell} \in \CC(\Rdl)$ 
    \begin{enumerate}[label=W\arabic*$_{\ell}$,labelsep=10pt,leftmargin=35pt]
        \item\label{ass:well-posedness_El_lowerbound} 
        is lower bounded, i.e., 
            $\CEunder_{\ell} \coloneqq \inf_{v \in \Rdl}\CE_{\ell}(v) >- \infty,$
    \item\label{ass:well-posedness_El_lipschitz} 
    satisfies for some constants $C_{u,\ell}, c_{u,\ell}  > 0$ the conditions
	\begin{align*}
		\abs{\CE_{\ell}(v)-\CE_{\ell}(v')}
        &\leq C_{u,\ell} \left(1+\N{v}_{2,d_\ell} + \N{v'}_{2,d_\ell}\right)\N{v-v'}_{2,d_\ell}
        \quad \text{for all } v,v' \in \Rdl,\\
		\CE_{\ell}(v) - \CEunder_{\ell}
        &\leq c_{u,\ell} \left(1+\N{v}_{2,d_\ell}^2\right) 
        \quad \text{for all } v \in \Rdl,
	\end{align*}
    \item\label{ass:well-posedness_El_growth} is either upper bounded, i.e., $\CEupper_{\ell} \coloneqq \sup_{v \in \Rdl}\CE_{\ell}(v) < \infty,$ or satisfies for some $c_{l,\ell}, C_{l,\ell} > 0$ the assumption
	\begin{align*} %\label{eq:quadratic_growth_condition_car}
		\CE_\ell(v) - \CEunder_\ell \geq c_{l,\ell} \N{v}_{2,d_\ell}^2 \quad \text{for all } \N{v}_{2,d_\ell} \geq C_{l,\ell}.
	\end{align*}
    \end{enumerate}
\end{assumption}

Assumption~\ref{ass:well-posedness_El} characterizes the regularity of the objective function~$\CE$ at the level of its components~$\CE_\ell$,
whereas Assumption~\ref{ass:well-posedness_E} imposes a global \emph{coarser} condition directly on $\CE$.
In Lemma~\ref{lem:assE_onEl_wellp} below, we show that, for an additively separable objective as in Assumption~\ref{ass:add_sep}, the global assumption implies the component‑wise one with the same constants.
Lemma~\ref{lem:assEl_onE_wellp} provides the converse statement, indicating that the global assumption holds with the worst-case constants appearing in the component‑wise assumptions.
Consequently, describing regularity more \emph{finely} (through the constants~$C_{u,\ell},c_{u,\ell},c_{l,\ell}, C_{l,\ell}$) for each component offers greater flexibility.
This lemma furthermore ensures, that for an additively separable objective function it is sufficient to directly impose the regularity conditions on the lower‑dimensional objectives~$\CE_\ell$.
This observation will be exploited later in Remark \ref{rem:implications_wp}, and Sections \ref{sec:mainresults_sep} and \ref{sec:mainresults_convmicro} (more precisely, in the proofs of Theorems \ref{thm:separation_anisotropicCBO_dynamics} and \ref{thm:conv_dlg1_Vk_micro} respectively).

\begin{lemma}
    \label{lem:assE_onEl_wellp}
    Let $\CE:\RD\rightarrow\R$ be additively separable as in Assumption~\ref{ass:add_sep}.
    If $\CE$ satisfies Assumption \ref{ass:well-posedness_E}, then $\CE_\ell$ fulfills Assumption~\ref{ass:well-posedness_El} for all $\ell \in \{1,\dots,L\}$ with constants
    \begin{equation}
        C_{u,\ell} \coloneqq C_u,
        \qquad
        c_{u,\ell}\coloneqq c_u,
        \qquad
        c_{l,\ell}\coloneqq c_l,
        \qquad
        C_{l,\ell}\coloneqq C_l.
    \end{equation}
    Moreover, $\CEunder = \sum_{\ell=1}^L\CEunder_{\ell}$. If $\CE$ is upper bounded, then $\CE_\ell$ is upper bounded for all $\ell \in \{1,\dots,L\}$ and $\CEupper = \sum_{\ell=1}^L\CEupper_{\ell}$.
\end{lemma}

\begin{proof}
    Fix $\bar\ell\in\{1,\dots,L\}$.
    Since the coordinate selector maps $\CB_\ell^\top$ act on disjoint coordinate sets $\CI_{\ell}$,
    we can define for any vector $\bar v\in\R^{d_{\bar\ell}}$ a vector $\bar x\in\RD$ such that $\CB_{\bar\ell}^\top \bar x= \bar v$ but $\CB_\ell^\top \bar x=\CB_\ell^\top \widetilde{x}$ for some $\widetilde{x}\in\RD$ for all $\ell\neq \bar\ell$.

    \textit{Regarding \ref{ass:well-posedness_E_lowerbound} $\implies$ \ref{ass:well-posedness_El_lowerbound}.}
    This follows directly from Lemma~\ref{lem:rel_Eunder_Elunder}.
    
    \textit{Regarding \ref{ass:well-posedness_E_lipschitz} $\implies$ \ref{ass:well-posedness_El_lipschitz}.}
    Let $\bar v, \bar v'\in\R^{d_{\bar\ell}}$ and consider $\bar x, \bar x'\in\RD$ as defined above with $\widetilde x=0$, i.e., $\CB_{\bar\ell}^\top \bar x= \bar v$ and $\CB_{\bar\ell}^\top \bar x'= \bar v'$ but $\CB_\ell^\top \bar x = \CB_\ell^\top \bar x'=\CB_\ell^\top 0$ for all $\ell\neq \bar\ell$.
    For the first part of \ref{ass:well-posedness_El_lipschitz}, using the first part of \ref{ass:well-posedness_E_lipschitz} in the second step, we have
    \begin{equation}
    \begin{split}
        \abs{\CE_{\bar\ell}(\bar v)-\CE_{\bar\ell}(\bar v')}
        &= \abs{\CE(\bar x)-\CE(\bar x')} \\
        &\leq C_{u} \left(1+\N{\bar x}_{2,D} + \Nnormal{\bar x'}_{2,D}\right)\Nnormal{\bar x-\bar x'}_{2,D}
        = C_{u} \left(1+\N{\bar v}_{2,d_{\bar\ell}} + \Nnormal{\bar v'}_{2,d_{\bar\ell}}\right)\Nnormal{\bar v-\bar v'}_{2,d_{\bar\ell}},
    \end{split}
    \end{equation}
    where the last equality holds since $\N{\bar x}_{2,D} = \N{\bar v}_{2,d_{\bar\ell}}$, $\N{\bar x'}_{2,D} = \N{\bar v'}_{2,d_{\bar\ell}}$, and $\N{\bar x - \bar x'}_{2,D} = \N{\bar v - \bar v'}_{2,d_{\bar\ell}}$ by choice of $\bar x$ and $\bar x'$.
    Analogously, for the second part of \ref{ass:well-posedness_El_lipschitz}, by using the second part of \ref{ass:well-posedness_E_lipschitz}, we have 
    \begin{equation}
    \begin{split}
        \CE_{\bar\ell}(\bar v)-\CEunder_{\bar\ell}
        \leq \CE(\bar x)-\CEunder
        \leq c_{u} \left(1+\N{\bar x}_{2,D}^2\right)
        = c_{u} \left(1+\N{\bar v}_{2,d_{\bar\ell}}^2\right).
    \end{split}
    \end{equation}

    \textit{Regarding \ref{ass:well-posedness_E_growth} $\implies$ \ref{ass:well-posedness_El_growth}.}
    Clearly, if $\CE$ is upper bounded, $\CE_\ell$ must be upper bounded for any $\ell \in \{1,\dots,L\}$.
    Denoting by $\CEupper_{\ell}$ these upper bounds, it holds $\sum_{\ell=1}^L\CEupper_{\ell} = \CEupper$ with an argument analogous to the one of Lemma~\ref{lem:rel_Eunder_Elunder}.
    For the second part of \ref{ass:well-posedness_El_growth}, 
    let $\bar v\in\R^{d_{\bar\ell}}$ and consider $\bar x\in\RD$ as defined above with $\widetilde x=\globmin$, i.e., $\CB_{\bar\ell}^\top \bar x= \bar v$ but $\CB_\ell^\top \bar x = \CB_\ell^\top \globmin$ for all $\ell\neq \bar\ell$.
    If $\N{\bar v}_{2,d_{\bar\ell}} \ge C_l$, then
    $\N{\bar x}_{2,D}^2 = \sum_\ell \N{\CB_\ell^\top \bar x}^2_{2,d_\ell} = \N{\bar v}_{2,d_{\bar\ell}}^2 + \sum_{\ell\neq \bar\ell}\N{\CB_\ell^\top \globmin}_{2,d_\ell}^2 \ge \N{\bar v}_{2,d_{\bar\ell}}^2 \ge C_l^2$,
    and hence $\N{\bar x}_{2,D}\ge C_l$.
    Thus, with Assumption~\ref{ass:add_sep} and Lemma~\ref{lem:rel_Eunder_Elunder},
    using the second part of \ref{ass:well-posedness_E_growth} in the second step implies
    \begin{equation}
    \begin{split}
        \CE_{\bar\ell}(\bar v) - \CEunder_{\bar\ell}
        = \CE(\bar x) - \CEunder
        \geq c_{l} \N{\bar x}_{2,D}^2
        = c_{l} \N{\bar v}_{2,d_{\bar\ell}}^2 + c_{l}  \sum_{\ell\not=\bar\ell}\N{\CB_\ell^\top \globmin}_{2,d_{\ell}}^2 \geq c_{l} \N{\bar v}_{2,d_{\bar\ell}}^2.
    \end{split}
    \end{equation}
    showing that $\CE_\ell$ fulfills the second part of \ref{ass:well-posedness_El_growth} with $c_{l,\ell}=c_l$ and $C_{l,\ell}=C_l$.
\end{proof}

\begin{lemma}
    \label{lem:assEl_onE_wellp}
    Let $\CE:\RD\rightarrow\R$ be additively separable as in Assumption~\ref{ass:add_sep}.
    If $\CE_{\ell}$ satisfies Assumption~\ref{ass:well-posedness_El} for all $\ell \in \{1,\dots,L\}$, then $\CE$ fulfills Assumption~\ref{ass:well-posedness_E} with constants
    \begin{equation}
        C_u \coloneqq \sqrt{L}\max_{\ell} C_{u,\ell}, \qquad
        c_u \coloneqq L \max_{\ell} c_{u,\ell}, \qquad
        c_l \coloneqq \tfrac12 \min_{\ell} c_{l,\ell}, \qquad
        C_l \coloneqq \sqrt{2\sum_{\ell=1}^L C_{l,\ell}^2},
    \end{equation}
    provided that all $\CE_\ell$ satisfy the same alternative in \ref{ass:well-posedness_El_growth}. Moreover, $\CEunder = \sum_{\ell=1}^L\CEunder_{\ell}$ and, if all $\CE_\ell$ are upper bounded, $\CEupper = \sum_{\ell=1}^L\CEupper_{\ell}$.
\end{lemma}

\begin{proof}
    Since the coordinate selector maps $\CB_\ell^\top$ act on disjoint coordinate sets $\CI_{\ell}$,
    we have for any vector $x\in\RD$ a collection of vectors $x_\ell\in\R^{d_{\ell}}$, $\ell \in \{1,\dots,L\}$, such that $\CB_{\ell}^\top x= x_\ell$ for all $\ell \in \{1,\dots,L\}$.
    
    \textit{Regarding \ref{ass:well-posedness_E_lowerbound} $\impliedby$ \ref{ass:well-posedness_El_lowerbound}.}
    Since each $\CE_\ell$ is lower bounded, $\CE(x)=\sum_{\ell=1}^L \CE_\ell(x_\ell)\ge \sum_{\ell=1}^L \CEunder_\ell$ for all $x\in\RD$, hence $\CE$ is lower bounded. The identity $\CEunder = \sum_{\ell=1}^L\CEunder_\ell$ follows as in Lemma~\ref{lem:rel_Eunder_Elunder}.

    \textit{Regarding \ref{ass:well-posedness_E_lipschitz} $\impliedby$ \ref{ass:well-posedness_El_lipschitz}.}
    Let $x,x'\in\RD$ and denote $x_\ell \coloneqq \CB_\ell^\top x$ and $x'_\ell \coloneqq \CB_\ell^\top x'$ for $\ell \in \{1,\dots,L\}$.
    With triangle inequality in the first step and using \ref{ass:well-posedness_El_lipschitz}, Cauchy-Schwarz inequality together with the fact that $\sum_{\ell=1}^L \N{x_\ell}_{2,d_\ell}^2 = \N{x}_{2,D}^2$ in the second, we have
    \begin{equation}
    \begin{split}
        \abs{\CE(x)-\CE(x')}
        \le \sum_{\ell=1}^L \abs{\CE_\ell(x_\ell)-\CE_\ell(x'_\ell)}
        &\le \sum_{\ell=1}^L C_{u,\ell} \left(1+\N{x_\ell}_{2,d_\ell} + \N{x'_\ell}_{2,d_\ell}\right)\N{x_\ell-x'_\ell}_{2,d_\ell}  \\
        &\le \sqrt{L}\max_\ell C_{u,\ell}\left(1+\N{x}_{2,D}+\N{x'}_{2,D}\right)\N{x-x'}_{2,D}.
    \end{split}
    \end{equation}
    Analogously, for the second part of \ref{ass:well-posedness_E_lipschitz}, by using the second part of \ref{ass:well-posedness_El_lipschitz}, we have 
    \begin{equation}
    \begin{split}
        \CE(x)-\CEunder
        = \sum_{\ell=1}^L \bigl(\CE_\ell(x_\ell)-\CEunder_\ell\bigr)
        &\le \sum_{\ell=1}^L c_{u,\ell} \left(1+\N{x_\ell}_{2,d_\ell}^2\right) \\
        &\le \max_\ell c_{u,\ell}\sum_{\ell=1}^L \left(1+\N{x_\ell}_{2,d_\ell}^2\right)
        \le L\max_\ell c_{u,\ell}\left(1+\N{x}_{2,D}^2\right).
    \end{split}
    \end{equation}

    \textit{Regarding \ref{ass:well-posedness_E_growth} $\impliedby$ \ref{ass:well-posedness_El_growth}.}
    If all $\CE_\ell$ are upper bounded, then $\CE(x)\le \sum_{\ell=1}^L \CEupper_\ell$ for all $x\in\RD$, hence $\CE$ is upper bounded and $\CEupper = \sum_{\ell=1}^L\CEupper_\ell$.
    For the second part of \ref{ass:well-posedness_E_growth},
    if all $\CE_\ell$ satisfy \ref{ass:well-posedness_El_growth},
    then for any $x\in\RD$,
    \begin{equation}
    \begin{split}
        \CE(x)-\CEunder
        = \sum_{\ell=1}^L \bigl(\CE_\ell(x_\ell)-\CEunder_\ell\bigr)
        &\ge \sum_{\ell:\,\N{x_\ell}_{2,d_\ell}\ge C_{l,\ell}} c_{l,\ell}\N{x_\ell}_{2,d_\ell}^2
        \ge \min_{\ell}c_{l,\ell}\sum_{\ell:\,\N{x_\ell}_{2,d_\ell}\ge C_{l,\ell}} \N{x_\ell}_{2,d_\ell}^2 \\
        &= \min_{\ell}c_{l,\ell}\left(\N{x}_{2,D}^2-\sum_{\ell:\,\N{x_\ell}_{2,d_\ell}<C_{l,\ell}}\N{x_\ell}_{2,d_\ell}^2\right).
    \end{split}
    \end{equation}
    With $c_l=\tfrac12\min_\ell c_{l,\ell}$ and using that
    $\sum_{\ell:\,\N{x_\ell}_{2,d_\ell}<C_{l,\ell}}\N{x_\ell}_{2,d_\ell}^2
    \le \sum_{\ell=1}^L C_{l,\ell}^2$ in the first step,
    it follows that
    \begin{equation}
        \CE(x)-\CEunder
        \ge 2c_l\left(\N{x}_{2,D}^2-\sum_{\ell=1}^L C_{l,\ell}^2\right)
        \ge c_l\N{x}_{2,D}^2
    \end{equation}
    for all $x\in\RD$ with $\N{x}_{2,D}\ge C_l=\sqrt{2\sum_{\ell=1}^L C_{l,\ell}^2}$ showing that $\CE$ fulfills the second part of \ref{ass:well-posedness_E_growth}.
\end{proof}

%%%%%%%%%%%%%%%%%%%%%%%%%%%%%%%%%%%%%%%%%%%%%%%
\subsection{Tractability assumptions}
\label{sec:relEandEl_relICP}

In the same spirit as in the previous section,
we examine in what follows tractability conditions of the landscape of the objective function $\CE$, which are required to prove the convergence of CBO to a global minimizer \cite{fornasier2024consensus,fornasier2021convergence,riedl2024perspective,bonandin2025strong}.

\begin{assumption}
    \label{ass:invcont_QLP_glob}
    The objective function 
    $\CE \in \CC(\RD)$ satisfies, for some constants $\eta, R_0, \CE_{\infty}>0$, and $\nu \in (0,\infty)$, the conditions
        \begin{subequations}
        \begin{align}
            \N{x-\globmin}_{\infty,D} &\le \frac{1}{\eta} \left( \CE(x)-\CEunder\right)^{\nu} \quad \text{for all }x \in \B^{D}_{R_0}(\globmin), \label{ass:invcont_QLP_glob_1}\\
            \CE_{\infty} &< \CE(x)-\CEunder \quad \text{for all }x \in \left( \B^{D}_{R_0}(\globmin) \right)^c. \label{ass:invcont_QLP_glob_2}
        \end{align}
        \end{subequations}
\end{assumption}

\begin{assumption}%[Inverse continuity assumption for $\CE_\ell$]
    \label{ass:invcont_QLPl2}
    The objective function 
    $\CE_{\ell} \in \CC(\Rdl)$ satisfies, for some constants $\eta_\ell, R_{0,\ell}, \CE_{\infty,\ell} >0$, and $\nu_\ell \in (0,\infty)$, the conditions
        \begin{subequations}
        \begin{align}
            \N{v-\CB_\ell^\top \globmin}_{\infty,d_\ell} &\le \frac{1}{\eta_\ell} \left( \CE_\ell(v)-\CEunder_\ell \right)^{\nu_\ell} \quad \text{for all $v \in \B^{d_\ell}_{R_0,\ell}(\CB_\ell^\top\globmin)$,} \label{ass:invcont_QLPl2_1}\\
            \CE_{\infty,\ell} &< \CE_\ell(v)-\CEunder_\ell \quad \text{for all $v \in \left( \B^{d_\ell}_{R_0,\ell}(\CB_\ell^\top\globmin) \right)^c$}. \label{ass:invcont_QLPl2_2}
        \end{align}
        \end{subequations}
\end{assumption}

Both Assumptions~\ref{ass:invcont_QLP_glob} and \ref{ass:invcont_QLPl2} should be interpreted as tractability conditions of the landscape of the objective function $\CE$ around the global minimizer $\globmin$ (see \eqref{ass:invcont_QLP_glob_1} and \eqref{ass:invcont_QLPl2_1}) and in the farfield (see \eqref{ass:invcont_QLP_glob_2} and \eqref{ass:invcont_QLPl2_2}).
The first parts (\eqref{ass:invcont_QLP_glob_1} and \eqref{ass:invcont_QLPl2_1}) describe  the local coercivity of $\CE$ and $\CE_\ell$, respectively, which implies that there is a unique minimizer $\globmin$ on $\B_{R_0}^D(\globmin)$ and that $\CE$ grows like  $x\mapsto \N{x-\globmin}_{\infty,D}^{1/\nu}$, and analogously for $\CE_\ell$.
This condition is known as the inverse continuity condition \cite{fornasier2020consensus_sphere_convergence,fornasier2021convergence,fornasier2024consensus,riedl2024perspective}, as a quadratic growth condition in the case $\nu= 1/2$ from \cite{anitescu2000degenerate, necoara2019linear}, or as the H\"olderian error bound condition in the case $\nu \in (0,1]$ \cite{bolte2017error}. 
\cite[Theorem~4]{necoara2019linear} and \cite[Theorem~2]{karimi2016linear} provide related assumptions that imply these conditions on $\RD$ and $\Rdl$ respectively.
The second parts (\eqref{ass:invcont_QLP_glob_2} and \eqref{ass:invcont_QLPl2_2}) describe the behavior of $\CE$ and $\CE_\ell$, respectively, in the farfield and prevent $\CE(x) \approx \CEunder$ for some $x \in \RD$ far away from $\globmin$, and analogously for $\CE_\ell$.
It is introduced to cover objectives that tend to a constant just above $\CE_{\infty}$ as $\N{x}_{\infty,D}\rightarrow\infty$, since such functions do not satisfy the first condition globally.
Further note that \eqref{ass:invcont_QLP_glob_1} and \eqref{ass:invcont_QLP_glob_2} as well as \eqref{ass:invcont_QLPl2_1} and \eqref{ass:invcont_QLPl2_2} imply the uniqueness of the respective global minimizer on the full space.

Assumption~\ref{ass:invcont_QLPl2} provides a \emph{more refined description} of the objective function, as it captures the tractability condition of the objective function landscape of $\CE$ at the level of individual components~$\CE_\ell$.
In contrast, Assumption~\ref{ass:invcont_QLP_glob} imposes global \emph{coarser} conditions on $\CE$.
This is captured by Lemma~\ref{lem:assE_onEl_ICP} below, which shows that Assumption~\ref{ass:invcont_QLP_glob} implies Assumption~\ref{ass:invcont_QLPl2} for all components $\ell \in \{1,\dots,L\}$ with the same constants in the case that $\CE$ is additively separable as in Assumption~\ref{ass:add_sep}.
Hence, characterizing the tractability (via the constants \(\eta_{\ell}, R_{0,\ell}, \CE_{\infty,\ell}, \nu_\ell\)) at the level of individual components allows for more flexibility, and a \emph{finer} description and characterization of the objective function~$\CE$.
On the other hand, it is sufficient to impose Assumption~\ref{ass:invcont_QLPl2} on the individual components in order to obtain the corresponding condition for $\CE$, see Lemma~\ref{lem:assEl_onE_ICP_general}.
In other words, for an objective function with this additively separable structure, we can assume tractability conditions directly on the lower‑dimensional components,
see also Remark~\ref{rem:crucial_remark}.
The relevance of the aforementioned lemmas is 
further recalled in Remark~\ref{rem:implications_ICP} of Section~\ref{sec:mainresults_convmf}.

The following implication shows that the tractability conditions required on the level of the component functions $\CE_\ell$ are weaker than those imposed on the full objective $\CE$.

\begin{lemma}
    \label{lem:assE_onEl_ICP}
    Let $\CE:\RD\rightarrow\R$ be additively separable as in Assumption~\ref{ass:add_sep}.
    If $\CE$ satisfies Assumption \ref{ass:invcont_QLP_glob}, then $\CE_\ell$ fulfills Assumption~\ref{ass:invcont_QLPl2} for all $\ell \in \{1,\dots,L\}$ with constants
    \begin{equation}
        \eta_{\ell} \coloneqq \eta,
        \qquad
        R_{0,\ell}\coloneqq R_{0},
        \qquad
        \CE_{\infty,\ell}\coloneqq \CE_{\infty},
        \qquad
        \nu_{\ell}\coloneqq \nu.
    \end{equation}
\end{lemma}

\begin{proof}
    Fix $\bar\ell\in\{1,\dots,L\}$.
    Since the coordinate selector maps $\CB_\ell^\top$ act on disjoint coordinate sets $\CI_{\ell}$,
    we define for any vector $\bar v\in\R^{d_{\bar\ell}}$ a vector $\bar x\in\RD$ such that $\CB_{\bar\ell}^\top \bar x= \bar v$ but $\CB_\ell^\top \bar x=\CB_\ell^\top \globmin$ for all $\ell\neq \bar\ell$.
    By construction,
    \begin{equation}
        \label{eq:lem:assE_onEl_ICP:aux:1}
        \N{\bar{x}-\globmin}_{\infty,D} = \N{\bar{v}-\CB_{\bar{\ell}}^\top \globmin}_{\infty,d_{\bar{\ell}}}.
    \end{equation}
    
    \textit{Regarding \eqref{ass:invcont_QLP_glob_1} $\implies$ \eqref{ass:invcont_QLPl2_1}.}
    For the first part \eqref{ass:invcont_QLPl2_1} of Assumption~\ref{ass:invcont_QLPl2}, 
    let $\bar{v} \in \B^{d_{\bar{\ell}}}_{R_0,{\bar{\ell}}}(\CB_{\bar{\ell}}^\top\globmin)$.
    Due to \eqref{eq:lem:assE_onEl_ICP:aux:1}, it then holds
    $\bar{x} \in \B^{D}_{R_0}(\globmin)$.
    Hence, using \eqref{ass:invcont_QLP_glob_1} of Assumption~\ref{ass:invcont_QLP_glob} in the first inequality, we have
    \begin{equation}
    \begin{split}
        \N{\bar{v}-\CB_{\bar{\ell}}^\top \globmin}_{\infty,d_{\bar\ell}}
        = \N{\bar{x}-\globmin}_{\infty,D}
        &\le \frac{1}{\eta} \left( \CE(\bar{x})-\CEunder \right)^{\nu}
        = \frac{1}{\eta} \left( \sum_{\ell=1}^L \left( \CE_{\ell}(\CB_{\ell}^\top \bar{x}) - \CEunder_\ell \right) \right)^{\nu} \\
        &= \frac{1}{\eta} \left( \left(\CE_{\bar{\ell}}(\CB_{\bar{\ell}}^\top \bar{x}) - \CEunder_{\bar{\ell}}\right) + \sum_{\ell \neq \bar\ell} \left(\CE_\ell(\CB_\ell^\top \bar{x}) - \CEunder_\ell \right)\right)^{\nu} \\
        &= \frac{1}{\eta} \left( \left(\CE_{\bar{\ell}}(\bar{v}) - \CEunder_{\bar{\ell}}\right) + \sum_{\ell \neq \bar\ell} \left(\CE_\ell(\CB_\ell^\top \globmin) - \CEunder_\ell \right)\right)^{\nu}
        = \frac{1}{\eta} \left( \CE_{\bar\ell}(\bar v) - \CEunder_{\bar\ell} \right)^{\nu},
    \end{split}
    \end{equation}
    where the last equality in the first line is due to $\CE$ being additively separable,
    while the remaining equalities use the definition of $\bar x$ together with the fact that $\CEunder_\ell = \CE_\ell (\CB^\top_\ell \globmin)$ as of Lemma~\ref{lem:rel_Eunder_Elunder}.

    \textit{Regarding \eqref{ass:invcont_QLP_glob_2} $\implies$ \eqref{ass:invcont_QLPl2_2}.}
    For the second part \eqref{ass:invcont_QLPl2_2} of Assumption~\ref{ass:invcont_QLPl2}, 
    let $\bar{v} \in \big(\B^{d_{\bar{\ell}}}_{R_0,{\bar{\ell}}}(\CB_{\bar{\ell}}^\top\globmin)\big)^c$.
    Due to \eqref{eq:lem:assE_onEl_ICP:aux:1}, it then holds
    $\bar{x} \in \big(\B^{D}_{R_0}(\globmin)\big)^c$.
    Hence, using \eqref{ass:invcont_QLP_glob_2} of Assumption~\ref{ass:invcont_QLP_glob} in the first inequality and exploiting that $\CE$ is additively separable in the second step, we conclude that
    \begin{equation}
    \begin{split}
        \CE_{\infty}
        &< \CE(\bar{x})-\CEunder
        = \sum_{\ell=1}^L \left( \CE_\ell(\CB_\ell^\top \bar{x}) - \CEunder_\ell \right)
        = \left(\CE_{\bar\ell}(\CB_{\bar\ell}^\top \bar{x}) - \CEunder_{\bar\ell}\right) + \sum_{\ell\not=\bar\ell} \left( \CE_\ell(\CB_\ell^\top \bar{x}) - \CEunder_\ell \right) \\
        &= \left(\CE_{\bar\ell}(\bar v) - \CEunder_{\bar\ell}\right) + \sum_{\ell\not=\bar\ell} \left( \CE_\ell(\CB_\ell^\top \globmin) - \CEunder_\ell \right)
        = \CE_{\bar\ell}(\bar v) - \CEunder_{\bar\ell},
    \end{split}
    \end{equation}
    where we used in the last equality the fact that $\CEunder_\ell = \CE_\ell (\CB^\top_\ell \globmin)$ as of Lemma~\ref{lem:rel_Eunder_Elunder}. This shows  \eqref{ass:invcont_QLPl2_2}.
\end{proof}

The next statement demonstrates that the converse of Lemma~\ref{lem:assE_onEl_ICP} holds for worst‑case constants, i.e., the worst values obtained from the original conditions.

To appreciate the relevance of this observation, consider the following scenario (see Example~\ref{ex:ICP_worstcase} for a concrete illustration).
Suppose that the objective function $\CE$ is additively separable as in Assumption~\ref{ass:add_sep} and that the constants appearing in Assumption \ref{ass:invcont_QLPl2} are well‑behaved for all indices $\ell$ except for one. In this  case, the global tractability condition on $\CE$, Assumption \ref{ass:invcont_QLP_glob}, does not reflect this situation, since they are determined by the worst‑case constants (as shown in the following lemma).
In contrast, examining the tractability conditions on the level of the components $\CE_\ell$ does capture the favorable behavior of the majority of the indices. Consequently, whenever an additively separable structure of $\CE$ is known a priori, the tractability conditions on the individual components $\CE_\ell$ provide a more informative description than the corresponding global tractability condition on $\CE$.

\begin{lemma}
    \label{lem:assEl_onE_ICP_general}
    Let $\CE:\RD\rightarrow\R$ be additively separable as in Assumption~\ref{ass:add_sep}.
    If $\CE_{\ell}$ satisfies Assumption~\ref{ass:invcont_QLPl2} for all $\ell \in \{1,\dots,L\}$, then $\CE$ fulfills Assumption~\ref{ass:invcont_QLP_glob} with constants
    \begin{equation}
        \nu \coloneqq \min_{\ell}\nu_\ell, 
        \qquad
        R_0 \coloneqq \min_{\ell} R_{0,\ell},
        \qquad
        \CE_{\infty} \coloneqq \min_{\ell} \widetilde \CE_{\infty,\ell},
        \qquad
        \eta \coloneqq \left(\min_{\ell} \eta_{\ell} \right)\frac{(\overbar{\CE}_{R_0}-\CEunder)^{\nu}}{\max_\ell (\overbar{\CE}_{R_0}-\CEunder)^{\nu_\ell}},
    \end{equation}
    where $\overbar{\CE}_{R_0} \coloneqq \max_{x \in \B_{R_0}^D\left(\globmin\right)} \CE(x)$ and $\widetilde{\CE}_{\infty,\ell} \coloneqq \min\bigl\{\CE_{\infty,\ell}, (\eta_\ell R_0)^{1/\nu_\ell}\bigr\}$.
\end{lemma}

\begin{proof}
    Since the coordinate selector maps $\CB_\ell^\top$ act on disjoint coordinate sets $\CI_{\ell}$,
    we have for any vector $x\in\RD$ a collection of vectors $x_\ell\in\R^{d_{\ell}}$, $\ell \in \{1,\dots,L\}$, such that $\CB_{\ell}^\top x= x_\ell$ for all $\ell \in \{1,\dots,L\}$.
    Moreover, by Lemma~\ref{lem:rel_Eunder_Elunder}, we have $\CEunder = \sum_{\ell=1}^L \CEunder_\ell$.

    \textit{Regarding \eqref{ass:invcont_QLP_glob_1} $\impliedby$ \eqref{ass:invcont_QLPl2_1}.}
    Let $x \in \B^D_{R_0}(\globmin)$. Then, for every $\ell \in \{1,\dots,L\}$,
    \begin{equation}
        \N{x_\ell-\CB_\ell^\top\globmin}_{\infty,d_\ell}
        =\N{\CB_\ell^\top x-\CB_\ell^\top\globmin}_{\infty,d_\ell}
        \le \N{x-\globmin}_{\infty,D}
        \le R_0 \le R_{0,\ell},
    \end{equation}
    hence $x_\ell \in \B^{d_\ell}_{R_{0,\ell}}(\CB_\ell^\top\globmin)$.
    Using now in the first line that \eqref{ass:invcont_QLPl2_1} holds for all $\ell \in \{1,\dots,L\}$ and in the second inequality the definition $\widetilde\eta \coloneqq \min_\ell\eta_\ell$,
    we obtain
    \begin{equation}
        \label{eq:proof:lem:assEl_onE_ICP_general:10}
    \begin{split}
        \N{x-\globmin}_{\infty,D}
        = \max_{\ell} \N{\CB_\ell^\top x-\CB_\ell^\top\globmin}_{\infty,d_\ell}
        &\le \max_{\ell} \frac{1}{\eta_\ell}
        \left(\CE_\ell(\CB_\ell^\top x-\CEunder_\ell\right)^{\nu_\ell}
        \le \frac{1}{\widetilde\eta} \max_{\ell} 
        \left(\CE_\ell(\CB_\ell^\top x)-\CEunder_\ell\right)^{\nu_\ell} \\
        &\le \frac{1}{\widetilde\eta} \max_{\ell} 
        \left(\sum_{\ell=1}^L\CE_\ell(\CB_\ell^\top x)-\CEunder_\ell\right)^{\nu_\ell}
        = \frac{1}{\widetilde\eta} \max_{\ell} 
        \left(\CE(x)-\CEunder\right)^{\nu_\ell}
    \end{split}
    \end{equation}
    after having used the additive separability as of Assumption~\ref{ass:add_sep} together with Lemma~\ref{lem:rel_Eunder_Elunder} in the last step.
    Since $\CE$ is continuous on ${\B^{D}_{R_0}}(\globmin)$, we denote $\overbar{\CE}_{R_0} \coloneqq \max_{x \in \B_{R_0}^D\left(\globmin\right)} \CE(x)<\infty$.
    
    \eqref{eq:proof:lem:assEl_onE_ICP_general:10} to obtain 
    \begin{equation}
    \begin{split}
        \N{x- \globmin}_{\infty,D}
        =
        \frac{1}{\widetilde\eta} \max_\ell\left( \CE(x)-\CEunder \right)^{\nu_\ell}
        &=
        \frac{1}{\widetilde\eta} \max_\ell\left(\big(\overbar{\CE}_{R_0}-\CEunder\big) \frac{\CE(x)-\CEunder}{\overbar{\CE}_{R_0}-\CEunder} \right)^{\nu_\ell} \\
        &\leq 
        \frac{1}{\widetilde\eta} \max_\ell \big(\overbar{\CE}_{R_0}-\CEunder\big)^{\nu_\ell} \max_\ell\left(\frac{\CE(x)-\CEunder}{\overbar{\CE}_{R_0}-\CEunder} \right)^{\nu_\ell} \\
        &\leq \frac{1}{\widetilde\eta / \big(\!\max_\ell (\overbar{\CE}_{R_0}-\CEunder)^{\nu_\ell}\big)} \left(\frac{\CE(x)-\CEunder}{\overbar{\CE}_{R_0}-\CEunder} \right)^{\nu}
        = \frac{1}{\eta} \left(\CE(x)-\CEunder\right)^\nu,
    \end{split}
    \end{equation}
    where we used in the next-to-last step that $(\CE(x)-\CEunder)/(\overbar{\CE}_{R_0}-\CEunder)<1$.

    \textit{Regarding \eqref{ass:invcont_QLP_glob_2} $\impliedby$ \eqref{ass:invcont_QLPl2_2}.}
    We first observe that for all $\ell \in \{1,\dots,L\}$, thanks to the continuity of $\CE_\ell$ on $\B^{d_\ell}_{R_0,\ell}(\CB_\ell^\top\globmin)$, we can obtain conditions analogous to \eqref{ass:invcont_QLPl2_2} but on the larger (since $R_0\leq R_{0,\ell}$) sets $\big( \B^{d_\ell}_{R_0}(\CB_\ell^\top\globmin) \big)^c$ for some new, potentially smaller constants $\widetilde{\CE}_{\infty,\ell} >0$.
    That is, it holds
    \begin{equation}
        \label{ass:invcont_QLPl2_new}
        \widetilde{\CE}_{\infty,\ell} < \CE_\ell(v)-\CEunder_\ell \quad \text{for all $v \in \left( \B^{d_\ell}_{R_0}(\CB_\ell^\top\globmin) \right)^c$}.
    \end{equation}
    To be more precise, decompose
    $\big(\B^{d_\ell}_{R_0}(\CB_\ell^\top\globmin)\big)^c = \big(\B^{d_\ell}_{R_{0,\ell}}(\CB_\ell^\top\globmin)\big)^c \cup \big(\overline{\B^{d_\ell}_{R_{0,\ell}}(\CB_\ell^\top\globmin)} \setminus \B^{d_\ell}_{R_0}(\CB_\ell^\top\globmin) \big)$.
    On the first set, \eqref{ass:invcont_QLPl2_2} already yields $\CE_{\infty,\ell} < \CE_\ell(v)-\CEunder_\ell$.
    On the second set, which is compact, continuity of $\CE_\ell$ implies that the minimum $m_\ell \coloneqq \min_{v \in \R^{d_\ell} : R_0 \le \|v-\CB_\ell^\top\globmin\|_{\infty,d_\ell} \le R_{0,\ell}} \bigl(\CE_\ell(v)-\CEunder_\ell\bigr)$ is attained and strictly positive.
    Since $R_0 \le R_{0,\ell}$, we can use \eqref{ass:invcont_QLPl2_1} on the annulus $\left\{v \in \R^{d_\ell} : R_0 \le \|v-\CB_\ell^\top\globmin\|_{\infty,d_\ell} \le R_{0,\ell}\right\}$.
    Indeed, for any such $v$, \eqref{ass:invcont_QLPl2_1} implies
    \begin{equation}
        \N{v-\CB_\ell^\top\globmin}_{\infty,d_\ell} \le \frac{1}{\eta_\ell}\bigl(\CE_\ell(v)-\CEunder_\ell\bigr)^{\nu_\ell},
    \end{equation}
    hence $\CE_\ell(v)-\CEunder_\ell \ge \bigl(\eta_\ell \N{v-\CB_\ell^\top\globmin}_{\infty,d_\ell}\bigr)^{1/\nu_\ell} \ge (\eta_\ell R_0)^{1/\nu_\ell}$.
    Therefore, $m_\ell \ge (\eta_\ell R_0)^{1/\nu_\ell}$.
    Hence, defining $\widetilde{\CE}_{\infty,\ell} \coloneqq \min\bigl\{\CE_{\infty,\ell}, (\eta_\ell R_0)^{1/\nu_\ell}\bigr\}$,
    we obtain \eqref{ass:invcont_QLPl2_new}.

    Now let $x \in \bigl(\B^D_{R_0}(\globmin)\bigr)^c$.
    Then there exists $\bar\ell\in\{1,\dots,L\}$ such that $\N{\CB_{\bar{\ell}}^\top x - \CB_{\bar{\ell}}^\top \globmin}_{\infty, d_{\bar{\ell}}} > R_0$.
    Hence, $x_{\bar\ell} = \CB_{\bar{\ell}}^\top x \in \big( \B^{d_{\bar{\ell}}}_{R_0}(\CB_{\bar{\ell}}^\top\globmin) \big)^c$.    
    Using the additive separability as of Assumption~\ref{ass:add_sep} together with Lemma~\ref{lem:rel_Eunder_Elunder} in the first step, and the new condition \eqref{ass:invcont_QLPl2_new} to obtain the first inequality in the third line, we have
    \begin{equation}
    \begin{split}
        \CE(x) - \CEunder
        &= \sum_{\ell=1}^L \left( \CE_\ell(\CB_\ell^\top x)-\CEunder_\ell \right) 
        = \left( \CE_{\bar\ell}(\CB_{\bar\ell}^\top x)-\CEunder_{\bar\ell} \right) + \sum_{\ell\not=\bar\ell} \left( \CE_\ell(\CB_\ell^\top x)-\CEunder_\ell \right) \\
        &\ge \CE_{\bar{\ell}}(\CB_{\bar{\ell}}^\top x)-\CEunder_{\bar{\ell}}  \\
        &> \widetilde{\CE}_{\infty,\bar{\ell}} 
        \ge \min_\ell \widetilde{\CE}_{\infty,\ell},
    \end{split} 
    \end{equation}
    which concludes the proof.
\end{proof}

\begin{example}
    \label{ex:ICP_worstcase}
    In this example, we consider an objective function \(\mathcal{E}\) that is additively separable as in Assumption \ref{ass:add_sep} with $L=2$,
    and such that the constants appearing in Assumption \ref{ass:invcont_QLPl2} are well‑behaved in the first coordinates (corresponding to component $\ell=1$) but not the last coordinate (corresponding to $\ell=2$).
    
    Let $D \gg 1$, and set $d_1=D-1$, $d_2 = 1$.
    We define 
    \begin{equation}
        \CE(x) = \CE_1(x_1,\ldots,x_{D-1}) + \CE_2(x_D),
    \end{equation}
    for any $x \in \RD$
    with
    \begin{equation}
    \label{eq:ICP_worstcase}
    \begin{split}
        \CE_1(v) &= 2 \N{v}^2_{2,D-1} \quad \text{for all }v \in \R^{D-1},\\
        \CE_2(v) &= \abs{v}^p \quad \text{or} \quad \CE_2(v) = v^2-c_D \cos(2\pi v) + c_D  \quad \text{for all }v \in \R,
    \end{split}
    \end{equation}
    where $p \gg 1$ in the first case, or $c_D \in \R$ in the second.
    (Note, that the second choice of $\CE_2$ in \eqref{eq:ICP_worstcase} corresponds to the classical one‑dimensional Rastrigin function with amplitude coefficient $c_D$).
    It is easy to verify, that $\CE$ admits a unique minimizer $\globmin = 0 \in \RD$,
    and that $\CEunder_1=\CEunder_2=0$.
    
    For all $v$ in a neighborhood of $\CB_1^\top \globmin = 0 \in \R^{D-1}$ and for the choices $\nu_1=1/2$ and $\eta_1=1$, it holds that
        \begin{equation}
        \label{eq:}
            \N{v}_{\infty,D-1} \le \N{v}_{2,D-1} \le \sqrt{2} \N{v}_{2,D-1} = 
            \frac{1}{\eta_1} (\CE_1(v)-\CEunder_1)^{\nu_1},
        \end{equation}
    which corresponds to \eqref{ass:invcont_QLPl2_1} holding globally for $\ell=1$.
    Hence, $\CE_1$ is well-conditioned, as the exponent $\nu_1$ corresponds to the well-behaved quadratic growth condition \cite{anitescu2000degenerate, necoara2019linear}.
        
    On the other hand, for all $v$ in a neighborhood of $\CB_2^\top \globmin = 0 \in \R$, and for the first choice of $\CE_2$ in \eqref{eq:ICP_worstcase},  condition \eqref{ass:invcont_QLPl2_1} translates to 
        \begin{equation}
    \label{eq:ICP_worstcase_proof1}
            \abs{v} \le \frac{1}{\eta_2} (\CE_2(v)-\CEunder_2)^{\nu_2} = \frac{1}{\eta_2} \abs{v}^{p\nu_2},
        \end{equation}
        which forces $\nu_2 = 1/p \ll 1$ and $\eta_2 \approx 1$. 
        Consequently, \(\CE_{2}\) grows according to \(v\mapsto|v|^{1/p}\) with the exponent \(1/p\) being very small. This extremely slow growth makes the function almost flat in a neighborhood of its minimizer, so the behavior of the objective in this component is poorly conditioned.
    An analogous condition to that in \eqref{eq:ICP_worstcase_proof1} can be derived for the second choice of $\CE_2$ in \eqref{eq:ICP_worstcase}.
    Certainly, this $\nu_2$ is worse than in the quadratic growth case, hence leading to the same conclusion of a poorly conditioned behavior of  $\CE_2$.
    
    Since Lemma \ref{lem:assEl_onE_ICP_general} shows that $\nu = \min \{\nu_1,\nu_2\} = \nu_2 = 1/p$ needs to be taken,
    the global tractability condition in Assumption~\ref{ass:invcont_QLP_glob} would employ the constant \(\nu = 1/p\), thereby overlooking the favorable behavior of the function in the first \(D-1\) coordinates and failing to exploit the favorable structure of $\CE$ in the first component.
\end{example}

%%%%%%%%%%%%%%%%%%%%%%%%%%%%%%%%%%%%%%%%%%%%%%%%%%
%%%%%%%%%% Section %%%%%%%%%%%%%%%%%%%%%%%%%%%%%%%
%%%%%%%%%%%%%%%%%%%%%%%%%%%%%%%%%%%%%%%%%%%%%%%%%%
\section{Discussion of the main result}
\label{sec:mainresults}
In Theorem~\ref{thm:conv_dlg1_Vk_micro_nonrigorous} in the introduction we presented the main result of this paper as an informal statement.
The purpose of this section is to provide with Theorem~\ref{thm:conv_dlg1_Vk_micro} and Corollary~\ref{thm:conv_dlg1_Vk_micro_epstot} a precise formulation for it, together with a rigorous proof in Section~\ref{sec:mainresults_convmicro}.
To this end,
let us first introduce auxiliary continuous-time processes related to \eqref{eq:aCBO_micro_EM}, discuss their well-posedness as well as several intermediate results.

\paragraph{Continuous-time formulation of the CBO dynamics and its mean-field perspective.}
The interacting discrete-time particle system~\eqref{eq:aCBO_micro_EM} can be regarded as an Euler-Maruyama time discretization \cite{higham2001algorithmic,graham2013stochastic} of the continuous-time system of stochastic differential equations
\begin{equation}
	\label{eq:aCBO_micro}
		dX^i_t = -\lambda\left(X^i_t - \conspoint{\empmeasure{t}}\right) dt + \sigma \sum_{k=1}^D \left(X^i_t - \conspoint{\empmeasure{t}}\right)_k d(B^{(i)}_t)_k e_k^D,
        \qquad i \in \{ 1, \ldots, N \},
\end{equation}
where $\{B^{(i)}_t\}_{i \in \{ 1, \ldots, N \}}$ denote $D$-dimensional independent Brownian processes, and where the consensus point $x_{\alpha}$ is defined as in \eqref{eq:conspoint}.
To gain a theoretical understanding of CBO,
we analyze the macroscopic behavior of the agent density through a mean-field limit associated with the particle dynamics instead of investigating the microscopic system~\eqref{eq:aCBO_micro}.
More precisely, since the propagation of chaos assumption \cite{mckean1967propagation,sznitman1991propagation} holds for CBO~\cite{huang2021MFLCBO,fornasier2024consensus,gerber2025mean},
we consider the McKean mono-particle process
\begin{equation}
\label{eq:aCBO_McKean}
		d\Xbar{t} = -\lambda\left(\Xbar{t}-\conspoint{\rho_t}\right)dt
	+ \sigma \sum_{k=1}^D \left(\Xbar{t}-\conspoint{\rho_t}\right)_k d(B_t)_ke_k^D,
\end{equation}
where the law $\rho_t = \Law(\Xbar{t})$ satisfies
the nonlinear nonlocal Fokker-Planck equation
\begin{equation}
    \label{eq:aCBO_mf}
    \partial_t \rho_t = \lambda \tn{div}\left((x-\conspoint{\rho_t})\rho_t\right) +\frac{\sigma^2}{2} \sum_{k=1}^D \partial_{kk}\left((x-\conspoint{\rho_t}\right)_k^2 \rho_t)
\end{equation}
in a weak sense, see Definition \ref{def:fokker_planck_weak_sense}.

\paragraph{Separation of the anisotropic CBO dynamics.}
If the objective function $\CE$ satisfies Assumption~\ref{ass:add_sep}, i.e., if it is additively separable,
we show in Theorem~\ref{thm:separation_anisotropicCBO_dynamics} that the $D$-dimensional anisotropic CBO dynamics~\eqref{eq:aCBO_McKean} separates into $L$ independent and separate $d_\ell$-dimensional dynamics of the form
\begin{subequations}
	\label{eq:aCBO_McKean_sepl}
	\begin{align}
	d\Xbar{t}^\ell &=
	-\lambda \left(\Xbar{t}^\ell - \conspointcoordinate{\rho_t^\ell}\right)dt
	+\sigma \sum_{k=1}^{d_\ell}\left(\Xbar{t}^\ell - \conspointcoordinate{\rho_t^\ell}\right)_kd(B^{\ell}_t)_k e_k^{d_\ell}\\
	\conspointcoordinate{\rho_t^\ell} &\coloneqq \int_{\Rdl} x \frac{\omegaa^\ell(x)}{\N{\omegaa^\ell}_{L_1(\rho_t^\ell)}}d\rho_t^\ell(x), \quad\text{with}\quad \omegaa^\ell(x)\coloneqq\exp(-\alpha\CE_\ell(x)), \label{def:xalphadl}
	\end{align}
\end{subequations}
with initial data~$\Xbar{0}^\ell \sim \rho^\ell_0 \coloneqq \CB_{\ell}^\top\#\rho_0 \in \CP(\Rdl)$ (here, $\#$ denotes the pushforward measure of $\rho_0$ under $\CB_{\ell}^\top$), Brownian motions $B^\ell_t \coloneqq \CB_{\ell}^\top B_t \in \Rdl$, and where $\rho^\ell_t = \Law(\Xbar{t}^\ell)$  for $\ell \in \{ 1, \ldots, L \}$.
This separation of the dynamics arises from the anisotropic noise structure and does not occur in the isotropic setting.

By exploiting the separability of the objective function, the system reduces from a high-dimensional dynamics to a set of $L$ independent lower-dimensional dynamics. 

\paragraph{Weak solutions and well-posedness.}
Before making this separation of the anisotropic CBO dynamics~\eqref{eq:aCBO_McKean} into its component dynamics~\eqref{eq:aCBO_McKean_sepl} rigorous in Theorem~\ref{thm:separation_anisotropicCBO_dynamics},
let us discuss the well-posedness of these equations in the subsequent remark.

\begin{definition}[Weak solution]
    \label{def:fokker_planck_weak_sense}
    Let $d \ge 1$ be a given dimension and $\conspoint{\dummy}$ be defined as in \eqref{eq:conspoint} for a general cost $\CE:\Rd\rightarrow\bbR$.
    Let $\mu_0 \in \CP(\Rd)$ and $T > 0$.
	We say that a measure $\mu\in\CC([0,T],\CP(\Rd))$ satisfies the Fokker-Planck equation
    \begin{equation}
    \label{eq:aCBO_mf_d}
        \partial_t \mu_t = \lambda \tn{div}\left((x-\conspoint{\mu_t})\mu_t\right) +\frac{\sigma^2}{2} \sum_{k=1}^d \partial_{kk}\left((x-\conspoint{\mu_t})_k^2 \mu_t\right)
    \end{equation}
    with initial condition $\mu_0$ in the weak sense in the time interval $[0,T]$, if we have for all $\phi \in \CC_*^2(\Rd)\coloneqq\{\phi\in\CC^2(\Rd) : \abs{\partial_k\phi(x)}\leq c(1+|x_k|) \text{ and } \N{\partial_{kk}^2 \phi}_{\infty,d}<\infty \text{ for all } k \in \{1,\dots,d\} \text{ and some } c>0\}$
    and all $t \in (0,T)$
	\begin{equation}                      
        \label{eq:aCBO_mf_weak}
		\frac{d}{dt}\int \phi(x) \,d\mu_t(x)
        =
		- \lambda\int \sum_{k=1}^d (x - \conspoint{\mu_t})_{k} \partial_{k} \phi(x) \, d\mu_t(x)
		+ \frac{\sigma^2}{2} \int \sum_{{k}=1}^d \left(x-\conspoint{\mu_t}\right)_{k}^2 \partial^2_{{kk}} \phi(x) \,d\mu_t(x)
	\end{equation}
	and $\lim_{t\rightarrow 0}\mu_t = \mu_0$ pointwise.
\end{definition}

\begin{remark}[Well-posedness of \eqref{eq:aCBO_McKean} and \eqref{eq:aCBO_McKean_sepl}]
\label{rem:implications_wp}
If $\CE$ satisfies Assumption \ref{ass:well-posedness_E}, then according to 
\cite[Theorem~3.1]{carrillo2018analytical} or \cite[Theorem~1]{fornasier2024consensus} there exists a unique nonlinear process $\overbar V \in \CC([0,T],\RD)$ satisfying \eqref{eq:aCBO_McKean}. The associated law $\rho = \Law(\overbar V)$ has regularity $\rho \in \CC([0,T], \CP_4(\RD))$ and is a weak solution to the Fokker-Planck equation \eqref{eq:aCBO_mf}.

With the same arguments, if $\CE_{\ell}$ satisfies Assumption \ref{ass:well-posedness_El}, then there exists a nonlinear process $\overbar{V}^{\ell} \in \CC([0,T],\Rdl)$ satisfying \eqref{eq:aCBO_McKean_sepl} and $\rho^{\ell}$ such that $\rho^{\ell} = \Law(\overbar{V}^\ell)$ belongs to $\CC([0,T], \CP_4(\Rdl))$.

As we will see in Theorem~\ref{thm:separation_anisotropicCBO_dynamics}, the processes $\overbar{V}^{\ell} \in \CC([0,T],\Rdl)$ and the measures $\rho^{\ell} = \Law(\overbar{V}^\ell)$ can be used to construct $\overbar V \in \CC([0,T],\RD)$ and $\rho = \Law(\overbar V)$, respectively.
Hence, for additively separable functions $\CE$, it suffices to require regularity conditions on the lower-dimensional functions $\CE_\ell$ to guarantee existence and uniqueness of a solution of both the mean-field dynamics \eqref{eq:aCBO_McKean} and \eqref{eq:aCBO_mf}.
At the same time,
Lemma \ref{lem:assEl_onE_wellp} shows that Assumption \ref{ass:well-posedness_El} implies Assumption \ref{ass:well-posedness_E}. Consequently, in view of the above discussion, we conclude that Assumption \ref{ass:well-posedness_El} is sufficient to guarantee the well‑posedness of both \eqref{eq:aCBO_McKean} and \eqref{eq:aCBO_McKean_sepl} and the respective Fokker-Planck equations.
\end{remark}

\paragraph{Proof strategy.}
Let us now outline the proof strategy leading to the rigorous formulation of the main result of this paper.
Theorem \ref{thm:separation_anisotropicCBO_dynamics} (presented in Section \ref{sec:mainresults_sep} and proven in Section~\ref{sec:mainresults_sep_proof}) rigorously demonstrates the separation of the anisotropic CBO dynamics for additively separable objective functions.
This separation provides insight into how the anisotropic CBO method exploits the additively separable structure of the objective function~$\CE$. 
With Theorem~\ref{thm:conv_dlg1_Vk_micro} and Corollary~\ref{thm:conv_dlg1_Vk_micro_epstot} (discussed and proven in Section \ref{sec:mainresults_convmicro})  we then present the rigorous formulation of Theorem~\ref{thm:conv_dlg1_Vk_micro_nonrigorous}, constituting the main result of this paper.
The proof of Theorem~\ref{thm:conv_dlg1_Vk_micro} relies on two central ingredients.
\begin{itemize}
    \item Theorem~\ref{thm:conv_dlg1_Vk} (stated in Section \ref{sec:mainresults_convmf} and proven in Section~\ref{sec:mainresults_convmf_proof}) which shows convergence of the anisotropic CBO dynamics~\eqref{eq:aCBO_McKean} and \eqref{eq:aCBO_mf} in the mean-field sense~\cite[Theorem~2]{fornasier2021convergence} under weaker assumptions and with more stringent hyperparameter choices, achieved by exploiting additively separable structure of $\CE$ (provided that such is available).
    
    In particular, if $\CE$ is additively separable, the hyperparameter~$\alpha$ can be chosen independently of the ambient dimension $D$, depending instead only on the intrinsic dimension $\intrd$.
    Analogously, it does not depend on properties of the $\CE$ as summarized in Assumption \ref{ass:invcont_QLP_glob} but only on the corresponding Assumptions~\ref{ass:invcont_QLPl2} on the lower-dimensional components $\CE_\ell$. 
    Besides allowing for more stringent requirements on the hyperparameters, this allows the convergence result to hold under less restrictive conditions than in the general case (see Subsection~\ref{sec:relEandEl_relICP}).
    \item The strong global convergence result for the CBO algorithm established in \cite{bonandin2025strong}.
\end{itemize}

%%%%%%%%%%%%%%%%%%%%%%%%%%%%%%%%%%%%%%%%%%%%%%%%%%%%%%%%
%%%%%%%%%%%%%%%%%%%%%%%%%%%%%%%%%%%%%%%%%%%%%%%%%%%%%%%%
% \newpage
\subsection{Separability of anisotropic CBO}
\label{sec:mainresults_sep}

This section is concerned with the separation of the full dynamics \eqref{eq:aCBO_McKean} into the $L$ component dynamics \eqref{eq:aCBO_McKean_sepl},
and analogously for \eqref{eq:aCBO_mf}.
The proof is deferred to Section~\ref{sec:mainresults_sep_proof}.

\begin{theorem}   
\label{thm:separation_anisotropicCBO_dynamics}
    Let $\CE:\RD\rightarrow\R$ be additively separable as in Assumption \ref{ass:add_sep}, and let $\CE_\ell$ satisfy Assumption~\ref{ass:well-posedness_El} for all $\ell \in \{ 1, \ldots, L \}$. 
    Let $T>0$ and $\rho_0\in\CP_4(\RD)$.
	Moreover, suppose that the measures $\rho^\ell_0\coloneqq\CB_{\ell}^\top\#\rho_0$, $\ell \in \{ 1, \ldots, L \}$, are statistically independent.
	Then, denoting by $\overbar{X} =(\Xbar{t})_{t\in[0,T]}$ the solution to the full dynamics~\eqref{eq:aCBO_McKean} and by $\overbar{X}^\ell=(\Xbar{t}^\ell)_{t\in[0,T]}$, $\ell \in \{ 1, \ldots, L \}$, the solutions to the component dynamics~\eqref{eq:aCBO_McKean_sepl},
    we have
	\begin{equation}
        \bbE \left(
		\sup_{t \in [0,T]} \N{\Xbar{t}-\sum_{\ell=1}^L \CB_\ell\Xbar{t}^\ell}_{2,D}^2 \right)
        = 0
	\end{equation}
    and therefore that 
    \begin{equation}
        \rho_{[0,T]}
        = \rho^1_{[0,T]} \otimes \cdots \otimes \rho^L_{[0,T]},
    \end{equation}
    where $\rho_{[0,T]}$ and $\rho^\ell_{[0,T]}$ for $\ell \in \{ 1, \ldots, L \}$ are the path-wise mean-field laws over $[0,T]$, that is $\rho_{[0,T]} \coloneqq \Law(\overbar{X}) \in \CP(\CC([0,T],\RD))$ and $\rho^\ell_{[0,T]} \coloneqq \Law(\overbar{X}^\ell) \in \CP(\CC([0,T],\Rdl))$.
\end{theorem}

Theorem~\ref{thm:separation_anisotropicCBO_dynamics} provides a separation result at the level of path-wise mean-field laws over $[0,T]$ for the mean-field CBO dynamics~\eqref{eq:aCBO_McKean}, which implies a point-wise-in-time result, that we use in Proposition~\ref{prop:QLPl_Vk} and in the proof of Theorem~\ref{thm:conv_dlg1_Vk}.
Let us mention, however, that the separation property of Theorem~\ref{thm:separation_anisotropicCBO_dynamics} is a mean-field phenomenon and, in general, does not hold for the finite-particle anisotropic CBO dynamics~\eqref{eq:aCBO_micro_EM}.

\begin{remark}[Computational complexity in the case of additive separability]
    The key to the results of this manuscript is Theorem~\ref{thm:separation_anisotropicCBO_dynamics} stating that for additively separable objective functions $\CE$ of the form~\eqref{eq:add_sep},
    the full $D$-dimensional anisotropic CBO dynamics separates into $L$ lower-dimensional $d_\ell$-dimensional dynamics.
    This is a fundamental simplification of the optimization problem since it is substantially more efficient to solve $L$  $d_\ell$-dimensional problems than a single $D$-dimensional problem.

    To simplify the following discussion, suppose we are tackling a discrete optimization problem, with $P\gg1$ options per dimension.
    If no prior information about the objective function is available, it is necessary to evaluate it at all $P^D$ admissible points  in order to find the minimizer.
    To contrast this with an extreme example, consider now a fully separable function $\CE$ of the form $\CE(x) = \sum_{\ell=1}^D \CE_\ell(x_\ell)$,
    where $d_\ell =1$ for all $\ell \in \{1,\ldots,L\}$.
    In this case, the objective function needs to be evaluated at only $D \cdot P$ points, which is much smaller than $P^D$.
    This argument extends to the continuous case and to the scenario in which not all $d_\ell$ are equal to one.
    In the latter, the leading order is~$\mathcal{O}(P^{\intrd})$, where $\intrd$ is the  intrinsic dimension given as in \eqref{def:intrinsic_dim}.
\end{remark}

%%%%%%%%%%%%%%%%%%%%%%%%%%%%%%%%%%%%%%%%%%%%%%%%%%%%%%%%
%%%%%%%%%%%%%%%%%%%%%%%%%%%%%%%%%%%%%%%%%%%%%%%%%%%%%%%%
\subsection{Global convergence in mean-field law}
\label{sec:mainresults_convmf}

In this section, we first show convergence in mean-field law of the anisotropic CBO dynamics \eqref{eq:aCBO_McKean}
by analyzing the energy functional $\CV_{\infty}$ defined for $\indivmeasure \in \CP_2(\RD)$ according to
\begin{equation}
\label{def:Vinfty}
    \CV_{\infty}(\indivmeasure) \coloneqq \max_{k \in \{ 1, \ldots, D \}} \CV_k(\indivmeasure),
\end{equation}
where 
\begin{equation}
    \label{def:Vk}
    \CV_k(\indivmeasure) \coloneqq \frac{1}{2} \int_{\RD} (x-\globmin )^2_k \,d \indivmeasure(x)
\end{equation}
for $k \in \{ 1, \ldots, D \}$.

\begin{theorem}[Global convergence in mean-field law]
\label{thm:conv_dlg1_Vk}
Let $\CE:\RD\rightarrow\R$ be additively separable as in Assumption \ref{ass:add_sep}, and let $\CE_\ell$ satisfy Assumption~\ref{ass:invcont_QLPl2} for all $\ell \in \{ 1, \ldots, L \}$. 
Let $\rho_0\in\CP_4(\RD)$ be such that $\globmin \in \tn{supp}(\rho_0)$. 
Define $\CV_\infty$ as in \eqref{def:Vinfty}. 
Fix any $\varepsilon\in \left(0, \CV_{\infty}(\rho_0)\right)$ and $\vartheta\in(0,1)$, choose parameters $\lambda, \sigma>0$ such that $2\lambda>\sigma^2$, and define
    \begin{equation}
    \label{def:Tstar}
        T^*\coloneqq\frac{1}{(1-\vartheta)(2\lambda-\sigma^2)}\log\left(\frac{\CV_\infty(\rho_0)}{\varepsilon}\right).
    \end{equation}
Let $\rho \in \CC([0,T^*],\CP_4(\RD))$ be a weak solution to the Fokker-Planck equation associated to the full dynamics \eqref{eq:aCBO_McKean} on the time interval $[0,T^*]$ with initial condition $\rho_0$ and let $\rho^{\ell} \in \CC([0,T^*],\CP_4(\Rdl))$, for $\ell \in \{ 1, \ldots, L \}$, be the solution to \eqref{eq:aCBO_McKean_sepl} with initial condition $\rho^\ell_0=\CB_{\ell}^\top\#\rho_0$, and such that $\rho_t = \rho^1_t \otimes \cdots \otimes \rho^L_t$ for any $t \in [0,T^{*}]$.
Then, there exists $\alpha_0>0$ of the form 
$\alpha_0 = \max_{\ell \in \{ 1, \ldots, L \}} \alphazerol$, 
with $\alphazerol >0$ as in \eqref{def:alpha0l}, such that for all $\alpha>\alpha_0$, we have
    \begin{equation}
    \label{eq:Vinfty_accuracy}
        \CV_{\infty}(\rho_T)=\varepsilon \quad \text{ with } T\in \left[\frac{1-\vartheta}{1+\vartheta/2}T^*,T^*\right].
    \end{equation}
Here, $\alphazerol$ depends, 
among problem-dependent quantities, on $d_\ell$ and $(\eta_\ell, R_{0,\ell},\CE_{\infty,\ell}, \nu_\ell)$ of Assumption~\ref{ass:invcont_QLPl2}, see Remark~\ref{rem:crucial_remark} for insights regarding the qualitative dependencies of $\alphazerol$ on quantities that are associated with $\CE_\ell$.
In particular, however, $\alphazerol$ is independent of both $D$ and parameters from Assumption~\ref{ass:invcont_QLP_glob}.
\end{theorem}

The proof of this result is postponed to Section~\ref{sec:mainresults_convmf_proof}.

\begin{remark}
\label{rem:notnecass_wellpsep}
    The statement of Theorem~\ref{thm:conv_dlg1_Vk} assumes to be given measures $\rho \in \CC([0,T^*],\CP_4(\RD))$ and $\rho^{\ell} \in \CC([0,T^*],\CP_4(\Rdl))$, $\ell \in \{ 1, \ldots, L \}$, such that $\rho_t = \rho^1_t \otimes \cdots \otimes \rho^L_t$ for any $t \in [0,T^{*}]$.

    The required existence and regularity is ensured by classical well-posedness results, which hold as commented on in Remark~\ref{rem:implications_wp} for instance under Assumption~\ref{ass:well-posedness_El} in combination with Lemma~\ref{lem:assEl_onE_wellp} and if $\rho_0 \in \CP_4(\RD)$.
    The separation for any given time $t \in [0,T^*]$ is verified in Theorem \ref{thm:separation_anisotropicCBO_dynamics} under the additional statistical independence assumption on the initial conditions $\rho^\ell_0=\CB_{\ell}^\top\#\rho_0$.
\end{remark}

\begin{remark}
\label{rem:implications_ICP}
Theorem \ref{thm:conv_dlg1_Vk} is valid under Assumption \ref{ass:invcont_QLPl2} and proves convergence in mean-field law of anisotropic CBO for a choice of CBO parameters depending on $\intrd$ (or, more precisely, on $d_\ell$ for all $\ell \in \{1,\ldots,L\}$) rather than the ambient dimension~$D$. 
In view of Lemma \ref{lem:assE_onEl_ICP}, which shows that the tractability conditions in Assumption \ref{ass:invcont_QLP_glob} satisfied by $\CE$ imply the corresponding conditions in Assumption \ref{ass:invcont_QLPl2} for each $\CE_\ell$ with identical constants,
Theorem \ref{thm:conv_dlg1_Vk} captures the following crucial fact when compared to \cite[Theorem~2]{fornasier2021convergence} and \cite[Theorem~3.6]{riedl2024perspective} (see also the discussion after Remark~\ref{rem:crucial_remark} in Section~\ref{sec:intro}):
if additional structure of $\CE$ is known,
convergence can be obtained under weaker assumptions on the objective function and with better parameter choices (better in the sense that these choices do not depend on the ambient dimension and depend only on properties of the separable components).
Conversely, with Lemma \ref{lem:assEl_onE_ICP_general} showing that tractability of the component functions $\CE_\ell$ implies tractability of $\CE$ with worst‑case constants, \cite[Theorem~3.6]{riedl2024perspective} suggests for additively separable functions $\CE$ unnecessarily pessimistic hyperparameter choices due to not taking into account the structure of the objective.
Theorem \ref{thm:conv_dlg1_Vk}, in contrast, captures this.
In this respect, Theorem \ref{thm:conv_dlg1_Vk} provides a deeper insight for the function class of additively separable objectives.
\end{remark}

%%%%%%%%%%%%%%%%%%%%%%%%%%%%%%%%%%%%%%%%%%%%%%%%%%%%%%%%
%%%%%%%%%%%%%%%%%%%%%%%%%%%%%%%%%%%%%%%%%%%%%%%%%%%%%%%%
\subsection{Strong global convergence of the numerical scheme}
\label{sec:mainresults_convmicro}

We are now prepared to present the rigorous formulation of Theorem~\ref{thm:conv_dlg1_Vk_micro_nonrigorous} together with its proof. We recall that $\N{\dummy}_{\infty,D}$ denotes the standard $\ell^{\infty}$-norm, see the Section~\ref{sec:notation} for notations.

\begin{theorem}[Global convergence in probability]
\label{thm:conv_dlg1_Vk_micro}
    Let $\CE:\RD\rightarrow\R$ be additively separable as in Assumption~\ref{ass:add_sep} and fulfill Assumption~\ref{ass:well-posedness_E}, and let $\CE_\ell$ satisfy Assumption~\ref{ass:invcont_QLPl2} for all $\ell \in \{ 1, \ldots, L \}$.
    Let $\rho_0 \in \CP_6(\Rd)$ be such that $\globmin \in \tn{supp}(\rho_0)$ and initialize $\Xihat{0}\sim\rho_0$. Moreover, suppose that the measures $\rho^\ell_0=\CB_{\ell}^\top\#\rho_0$, $\ell \in \{ 1, \ldots, L \}$, are statistically independent.
    Define $\CV_{\infty}$ as in \eqref{def:Vinfty}.
    Fix a time step size $0 < \dt \le 1$ and a total number of iterations $H \in \bbN^+$, and let
    $\{\Xihat{h \dt}\}_{h\in \{0,\ldots,H\}}^{i\in\{1,\ldots,N\}}$ 
    denote the iterates generated by the anisotropic CBO algorithm~\eqref{eq:aCBO_micro_EM}.
    Let $\vartheta \in (0,1)$, and choose hyperparameters $\lambda, \sigma > 0$ satisfying $2\lambda > \sigma^2$. Fix $\acc > 0$. 
    Then, there exists a constant
    of the form $\alpha_0 = \max_{\ell \in \{1,\ldots,L\}} \alphazerol>0$  (as specified in Theorem \ref{thm:conv_dlg1_Vk})
    depending only on the intrinsic dimension $\intrd$ and the constants of Assumption~\ref{ass:invcont_QLPl2},
    such that for all $\alpha > \alpha_0$
    and for any $i \in \{1,\ldots,N\}$
    it holds
    \begin{equation}
    \label{eq:conv_dlg1_Vk_micro1}
        \N{\frac{1}{N} \sum_{i=1}^N \widehat{X}^{i}_{H\dt} - \globmin}_{\infty,D} \le \acc
    \end{equation}
    with probability larger than
    \begin{equation}
    \label{eq:conv_dlg1_Vk_micro2}
    \left( 1- \frac{8\intrd}{\acc^2}  \CV_\infty(\rho_0) \exp \left(-\left(1-\vartheta \right)\left(2\lambda-\sigma^2\right)H \dt \right) \right)^L - \frac{4}{\acc^2} \left( \cna  \dt + \cmfa N^{-1}\right).
    \end{equation}
    Here, the constants $\cna, \cmfa>0$ depend only on the constants appearing in Assumption~\ref{ass:well-posedness_E}, on the initial condition $\rho_0$, and (exponentially) on the parameters $\alpha$, $\lambda$ and $\sigma$ as well as on the time horizon $H\dt$; moreover, $\cna$  depends linearly on the ambient dimension $D$. Both constants are independent of both the number of particles $N$ and the time step $\dt$. 
\end{theorem}

\begin{proof}% [Proof of Theorem~\ref{thm:conv_dlg1_Vk_micro}]
Following the proof strategy of \cite[Theorem 1.1]{bonandin2025strong}, 
 we fix $0<\dt \le 1$, $H \in \bbN^+$ and $\vartheta \in (0,1)$, and
we define the time horizon
\begin{equation*}
        T^*
        = \frac{1}{(1-\vartheta)\big(2\lambda -  \sigma^2\big)} \log\left(\left(\frac{\CV_{\infty}(\rho_0)}{\widetilde\varepsilon}\right)^{1/\xi}\right),
    \end{equation*}
where $\xi \coloneqq (1-\vartheta)/(1+\vartheta/2) \in (0,1)$ and $\widetilde\varepsilon$ chosen such that $H\Delta t = \xi T^*$ holds. The latter can be realized through the selection 
\begin{equation*}
    \widetilde\varepsilon = \CV_{\infty}(\rho_0)\exp\left(-(1-\vartheta)\big(2\lambda - \sigma^2\big)H\Delta t\right).
\end{equation*}
Observing that the accuracy 
\begin{equation}
\label{eq:def_eps}
        \varepsilon = \CV_{\infty}(\rho_0) (\widetilde\varepsilon/\CV_{\infty}(\rho_0))^{1/\xi}
\end{equation}
satisfies $\varepsilon\in(0,\CV_{\infty}(\rho_0))$ and the time horizon $T^*$ is as in \eqref{def:Tstar}, we can apply Theorem~\ref{thm:conv_dlg1_Vk} to bound the error in the mean-field limit.
More precisely, using Equation~\eqref{eq:upperbound_functional}
in its proof we have that there exists a constant $\alpha_0$ of the form specified in~\eqref{def:alpha} and~\eqref{def:alpha0l} such that for all $\alpha > \alpha_0$ it holds 
\begin{equation}
    \label{eq:mainproof_step1}
    \CV_{\infty}(\rho_{H \Delta t}) \leq \CV_{\infty}(\rho_0)\exp\left(-(1-\vartheta)\big(2\lambda - \sigma^2\big)H\Delta t\right)
       =\widetilde\varepsilon.
\end{equation}

It remains to estimate the probability that \eqref{eq:conv_dlg1_Vk_micro1} holds, i.e., that the empirical average of the particle system $\{\Xihat{{H\dt}}\}^{i \in \{1,\ldots,N\}}$ lies within an $\ell^\infty$-distance $\acc>0$ from the global minimizer $\globmin$.
Therefore, 
let $\{\Xibar{{H\dt}}\}^{i \in \{1,\ldots,N\}}$ denote independent copies of the solution to the mean-field dynamics \eqref{eq:aCBO_McKean}, initialized at $\Xihat{0}$ and driven by the same Brownian motions as those appearing in \eqref{eq:aCBO_micro_EM}.
With a triangle inequality, we can  decompose the total error from \eqref{eq:conv_dlg1_Vk_micro1} at the final time $H\dt$  as
\begin{equation}
    \N{\frac{1}{N} \sum_{i=1}^N \Xihat{{H\dt}} - x^*}_{\infty,D}
    \le \N{\frac{1}{N} \sum_{i=1}^N \left(\Xihat{{H\dt}} - \Xibar{{H\dt}}\right)}_{\infty,D}
    + \N{\frac{1}{N} \sum_{i=1}^N \Xibar{{H\dt}} - x^*}_{\infty,D}.
\end{equation}
Hence, the failure probability associated to \eqref{eq:conv_dlg1_Vk_micro1} can be bounded as
\begin{equation}
\begin{split}
\label{eq:prob_split_main}
    \bbP\left(\N{\frac{1}{N} \sum_{i=1}^N \Xihat{{H\dt}} - x^*}_{\infty,D} > \acc\right)
    &\leq\bbP\left(\N{\frac{1}{N} \sum_{i=1}^N \left(\Xihat{{H\dt}} - \Xibar{{H\dt}}\right)}_{\infty,D} > \frac{\acc}{2} \right) \\
    &\quad\,+ \bbP\left(\N{\frac{1}{N} \sum_{i=1}^N \Xibar{{H\dt}} - x^*}_{\infty,D} > \frac{\acc}{2}\right).
\end{split}
\end{equation}
We now estimate the two terms on the right-hand side of \eqref{eq:prob_split_main} individually.

For the second term in \eqref{eq:prob_split_main}, which concerns only the continuous mean-field trajectories, since the $\ell^{\infty}$-norm on $\RD$ can be written as the maximum over the $\ell^{\infty}$-norms on $\Rdl$ of the lower-dimensional components, we have for the complement that
\begin{equation}
\label{eq:intersection}
\begin{split}
    &\bbP\left(\N{\frac{1}{N} \sum_{i=1}^N \Xibar{{H\dt}} - \globmin}_{\infty,D} \le  \frac{\acc}{2}\right) \\
    &\quad\,= \bbP\left(\max_{\ell \in \{ 1, \ldots, L \}} \N{\CB_\ell^\top \left(\frac{1}{N} \sum_{i=1}^N \Xibar{{H\dt}} - \globmin\right)}_{\infty, d_\ell} \le \frac{\acc}{2}\right)\\
    &\quad\,= \bbP\left(\N{\frac{1}{N} \sum_{i=1}^N \CB_\ell^\top(\Xibar{{H\dt}} - \globmin)}_{\infty, d_\ell} \le \frac{\acc}{2} \text{ for all $\ell \in \{1,\ldots,L\}$}\right)
\end{split}
\end{equation}
As discussed in the proof of Theorem~\ref{thm:separation_anisotropicCBO_dynamics}, $\{\CB_\ell^\top \, \Xibar{0}\}_{\ell \in \{1,\ldots,L\}}$ are independent by assumption, and the independence is preserved by the dynamics since the projected processes $\{\CB_\ell^\top \,\Xibar{{t}}\}_{\ell \in \{1,\ldots,L\}}$ evolve according to decoupled equations driven by the independent Brownian motions $\{\CB_\ell^\top B_t\}_{\ell \in \{1,\ldots,L\}}$ for all $t$. 
Leveraging this, we may estimate
\begin{equation}
    \label{eq:intersection_aux_}
\begin{split}
    &\bbP\left(\N{\frac{1}{N} \sum_{i=1}^N \CB_\ell^\top(\Xibar{{H\dt}} - \globmin)}_{\infty, d_\ell} \le \frac{\acc}{2} \text{ for all $\ell \in \{1,\ldots,L\}$}\right)\\
    &\quad\, =\prod_{\ell=1}^L \bbP\left(\N{\frac{1}{N} \sum_{i=1}^N \CB_\ell^\top(\Xibar{{H\dt}} - \globmin)}_{\infty, d_\ell} \leq \frac{\acc}{2}\right) \\
    &\quad\,= \prod_{\ell=1}^L \left( 1- \bbP\left(\N{\frac{1}{N} \sum_{i=1}^N \CB_\ell^\top(\Xibar{{H\dt}} - \globmin)}_{\infty, d_\ell} > \frac{\acc}{2}\right) \right)\\
    &\quad\, \geq \prod_{\ell=1}^L \left( 1-\bbP\left(\N{\frac{1}{N} \sum_{i=1}^N \CB_\ell^\top(\Xibar{{H\dt}} - \globmin)}_{2, d_\ell} > \frac{\acc}{2}\right)\right)\\
    &\quad\, \geq \prod_{\ell=1}^L \left( 1- \frac{4}{\acc^2} \, \bbE \N{\frac{1}{N} \sum_{i=1}^N \CB_\ell^\top(\Xibar{{H\dt}} - \globmin)}_{2, d_\ell}^2 \right)\\
    &\quad\, \geq \prod_{\ell=1}^L \left( 1-\frac{4}{\acc^2} \frac{1}{N} \sum_{i=1}^N  \, \bbE \N{\CB_\ell^\top(\Xibar{{H\dt}} - \globmin)}_{2, d_\ell}^2 \right),
\end{split}
\end{equation}
where the first inequality follows from the relation between the $\ell^\infty$- and $\ell^2$-norms, the second from Markov's inequality, and the last one follows from Jensen's inequality.
Recalling the definition of the functional $\CV_k$ in~\eqref{def:Vk}, it holds $\mathbb{E} \Nbig{\CB_\ell^\top(\Xibar{{H\dt}}-\globmin)}_{2,d_\ell}^2 = 2 \sum_{k\in\CI_\ell} \CV_k(\rho_{H\dt})$.
Using this together with \eqref{eq:intersection_aux_}, we can conclude \eqref{eq:intersection} by
\begin{equation}
\label{eq:secondterm_bound}
\begin{split}
    \bbP\left(\N{\frac{1}{N} \sum_{i=1}^N \Xibar{{H\dt}} - \globmin}_{\infty,D} \le  \frac{\acc}{2}\right)
    &\geq \prod_{\ell=1}^L \left( 1- \frac{4}{\acc^2} \frac{1}{N} \sum_{i=1}^N  \, \bbE \N{\CB_\ell^\top(\Xibar{{H\dt}} - \globmin)}_{2, d_\ell}^2 \right)  \\
    &= \prod_{\ell=1}^L \left( 1-\frac{8}{\acc^2} \sum_{k\in \CI_\ell} \CV_k(\rho_{H\dt}) \right) \\
    &\geq \prod_{\ell=1}^L \left( 1-\frac{8}{\acc^2} d_\ell \CV_\infty(\rho_{H\dt}) \right) \geq \left( 1-\frac{8}{\acc^2} \intrd \CV_\infty(\rho_{H\dt}) \right)^L\\
    & \geq  \left( 1-  \frac{8}{\acc^2}\intrd \,\CV_\infty(\rho_0) \exp \left(-\left(1-\vartheta \right)\left(2\lambda-\sigma^2\right)H \dt \right) \right)^L
\end{split}
\end{equation}
where we exploited the definition of $\CV_\infty$ from \eqref{def:Vinfty} and used that $d_\ell \leq \intrd$.
In the last step, we further applied \eqref{eq:mainproof_step1}.

For the first term in \eqref{eq:prob_split_main}, which concerns the time discretization and finite-particle approximation,
we bound, switching from the $\ell^\infty$-norm to the $\ell^2$-norm via the trivial bound $\N{\dummy}_{\infty,D} \le \N{\dummy}_{2,D}$, with Markov's and Jensen's inequality
\begin{equation}
\label{eq:markov_poc_bound}
\begin{split}
    \bbP\left(\N{\frac{1}{N} \sum_{i=1}^N \left(\Xihat{{H\dt}} - \Xibar{{H\dt}}\right)}_{\infty,D} > \frac{\acc}{2}\right) 
    &\le \bbP\left(\N{\frac{1}{N} \sum_{i=1}^N \left(\Xihat{{H\dt}} - \Xibar{{H\dt}}\right)}_{2,D} > \frac{\acc}{2}\right) \\
    &\le \frac{4}{\acc^2} \bbE \N{\frac{1}{N} \sum_{i=1}^N \left(\Xihat{{H\dt}} - \Xibar{{H\dt}}\right)}_{2,D}^2 \\
    &\le \frac{4}{\acc^2}  \frac{1}{N} \sum_{i=1}^N \bbE \N{\Xihat{{H\dt}} - \Xibar{{H\dt}}}_{2,D}^2 \\
    &\le \frac{4}{\acc^2}  \left( \cna \dt + \cmfa N^{-1}\right),
\end{split}
\end{equation}
where the last inequality follows from the strong mean-square convergence estimate in \cite[Theorem~2.4]{bonandin2025strong}, 
with constants $\cna$ and $\cmfa$ depending on the constants appearing in Assumption~\ref{ass:well-posedness_E} on $\CE$ as well as $(\alpha,\lambda,\sigma,H\dt,\rho_0)$. Specifically, $\cna$ and $\cmfa$ depend exponentially on the parameters $\alpha,\lambda$ and $\sigma$ and on the time horizon $H \dt$, and $\cna$ further depends linearly on the ambient dimension $D$.
Combining the estimates \eqref{eq:secondterm_bound} and \eqref{eq:markov_poc_bound} and inserting the bounds into \eqref{eq:prob_split_main} shows

\begin{equation*}
\begin{split}
    &\bbP\left(\N{\frac{1}{N} \sum_{i=1}^N \Xihat{{H\dt}} - x^*}_{\infty,D} \le \acc\right)  = 1 - \bbP\left(\N{\frac{1}{N} \sum_{i=1}^N \Xihat{{H\dt}} - x^*}_{\infty,D} > \acc\right)\\
    &\geq 1 - \left(\frac{4}{\acc^2}  \left( \cna \dt + \cmfa N^{-1}\right) + 1 - \left( 1-  \frac{8}{\acc^2}\intrd \,\CV_\infty(\rho_0) \exp \left(-\left(1-\vartheta \right)\left(2\lambda-\sigma^2\right)H \dt \right) \right)^L\right)\\
    &\geq \left( 1-  \frac{8}{\acc^2}\intrd \,\CV_\infty(\rho_0) \exp \left(-\left(1-\vartheta \right)\left(2\lambda-\sigma^2\right)H \dt \right) \right)^L - \frac{4}{\acc^2}  \left( \cna \dt + \cmfa N^{-1}\right),
    \end{split}
\end{equation*}
which concludes the proof.
\end{proof}

The following corollary can be derived immediately from Theorem~\ref{thm:conv_dlg1_Vk_micro}
by bounding \eqref{eq:conv_dlg1_Vk_micro2} from below with $(1-\delta_1)^L-\delta_2$ through suitable choices of $N$, $\Delta t$, and $H$.

\begin{corollary}[Computational complexity]
\label{thm:conv_dlg1_Vk_micro_epstot}
     
     Fix $\delta_1 \in (0,1-(1/2)^{1/L})$ and $\delta_2 \in (0,1/2)$.
     Under the assumptions of Theorem~\ref{thm:conv_dlg1_Vk_micro}, estimate~\eqref{eq:conv_dlg1_Vk_micro1}
    can be achieved with probability at least $(1-\delta_1)^L-\delta_2$ by
    \begin{enumerate}[label=(\roman*)]
        \item choosing $\alpha_0 = \max_{\ell \in \{1,\ldots,L\}} \alphazerol$, for $\alphazerol>0$ 
        as specified in \eqref{def:alpha0l}, with $\varepsilon$ as given in 
        \begin{equation}
        \label{eq:def_varepsilon}
            \varepsilon = \CV_{\infty}(\rho_0) \left( \frac{\acc^2 \, \delta_1}{{8
            \intrd \CV_{\infty}(\rho_0)}}\right)^{\frac{1+\vartheta/2}{1-\vartheta}},
        \end{equation}
        \item fixing a time horizon
        \begin{equation}
        \label{eq:defT}
            T\geq \frac{1}{(1-\vartheta)(2\lambda - \sigma^2)} \log \left(\frac{8 \intrd \CV_{\infty}(\rho_0)}{\acc^2 \, \delta_1} \right) +1,
        \end{equation}
        and selecting
        \begin{equation}
        \label{eq:choicesNdtK}
            N \ge \frac{8\cmfa}{\acc^2 \, \delta_2},
            \quad \dt \le \frac{\acc^2 \, \delta_2}{8\cna}, \quad \text{and} \quad H = \left\lceil \frac{T}{\dt}\right\rceil.
        \end{equation}
    \end{enumerate}
\end{corollary}

\begin{remark}[Computational complexity, discussion]
    Corollary~\ref{thm:conv_dlg1_Vk_micro_epstot} addresses the computational complexity of the anisotropic CBO algorithm~\eqref{eq:aCBO_micro_EM}.
In particular, $\CO(HN)$ evaluations of the objective function $\CE$ are required, with the number of time steps $H$ and the number of particles $N$ as specified in \eqref{eq:choicesNdtK}. 
A closer look at the definition of $H$ and the time step size $\Delta t$ in \eqref{eq:choicesNdtK} reveals that the ambient dimension $D$ enters the bound for the number of time steps $H$ only linearly and not exponentially. This comes through the constant $\cna$ appearing in the nominator of the estimate on $\Delta t$, which depends linearly on $D$, as we mention in Theorem~\ref{thm:conv_dlg1_Vk_micro}.
Furthermore, the number of required particles $N$ depends exponentially on the intrinsic dimensionality $\intrd$ but not on the ambient dimension $D$ (see also Remark~\ref{rem:crucial_remark}).
To see this, recall from Theorem~\ref{thm:conv_dlg1_Vk_micro} that the dependency of the constant $\cmfa$ on $\alphazerol$, and thus $\alpha_0$, is exponential~\cite[Theorem~3.8]{fornasier2024consensus} (or even super-exponential), with $\alphazerol$ scaling according to the intrinsic dimension $\intrd$ (while remaining independent of $D$).
To obtain the latter, we substitute the choice~\eqref{eq:def_varepsilon} of $\varepsilon$ into the expression~\eqref{def:alpha0l} of $\alphazerol$. 
\end{remark}

The structure of the first term $(1-\delta_1)^L$ in the success probability in Corollary~\ref{thm:conv_dlg1_Vk_micro_epstot} is a natural consequence of how the anisotropic dynamics interacts with the choice of the $\ell^{\infty}$-norm. Because the objective function is additively separable, the continuous mean-field trajectories completely decouple into $L$ statistically independent components. 
Crucially, since the convergence being evaluated in the $\ell^{\infty}$-norm requires every single independent coordinate block to successfully concentrate around its respective component of the global minimizer simultaneously for the system to succeed globally,
the product form $(1-\delta_1)^L$ is simply the mathematical manifestation of this joint event.
This structure underlines the efficiency of the anisotropic framework in high-dimensional settings in case of a separable structure. By breaking the $D$-dimensional optimization problem into $L$ independent, low-dimensional sub-problems, the algorithm trades an exponential dependence of the particle number on the ambient dimension $D$ for a mild dependence of the success probability in the number of blocks $L$.

%%%%%%%%%%%%%%%%%%%%%%%%%%%%%%%%%%%%%%%%%%%%%%%%%%
%%%%%%%%%% Section %%%%%%%%%%%%%%%%%%%%%%%%%%%%%%%
%%%%%%%%%%%%%%%%%%%%%%%%%%%%%%%%%%%%%%%%%%%%%%%%%%
\section{Proof details for the main results of Sections~\ref{sec:mainresults_sep} and \ref{sec:mainresults_convmf}}
\label{sec:proofmainresults}

In this section, we present detailed proofs of the results stated in Sections \ref{sec:mainresults_sep} and \ref{sec:mainresults_convmf}. 
Specifically, Section \ref{sec:mainresults_sep_proof} contains the proof of the separability result in Theorem \ref{thm:separation_anisotropicCBO_dynamics}, which is based on Lemma \ref{lem:sepmeasure_implies_sepcons}.
In Section \ref{sec:mainresults_convmf_proof}, we prove Theorem \ref{thm:conv_dlg1_Vk}, which is about the convergence of the anisotropic CBO algorithm for a cost functional $\CE$ that is additively separable as in Assumption \ref{ass:add_sep}.

%%%%%%%%%%%%%%%%%%%%%%%%%%%%%%%%%%%%%%%%%%%%%%%%%%%%%%%%
%%%%%%%%%%%%%%%%%%%%%%%%%%%%%%%%%%%%%%%%%%%%%%%%%%%%%%%%

\subsection{Proof of Theorem~\ref{thm:separation_anisotropicCBO_dynamics}}
\label{sec:mainresults_sep_proof}

Let us start by proving an auxiliary lemma that is crucial to the proof of Theorem~\ref{thm:separation_anisotropicCBO_dynamics}.

\begin{lemma}
\label{lem:sepmeasure_implies_sepcons}
    Let $\CE:\RD\rightarrow\R$ be additively separable as in Assumption \ref{ass:add_sep}.
    Let $\indivmeasure^{\ell} \in \mathcal{P}(\Rdl)$ for all $\ell \in \{ 1, \ldots, L \}$ and let $\indivmeasure \in \mathcal{P}(\RD)$ be such that $\indivmeasure = \indivmeasure^1 \otimes \cdots \otimes \indivmeasure^L$.
    Then, 
    \begin{equation}
     \label{eq:consensus_points_relation_true}
        \conspoint{\indivmeasure}
    		= \sum_{\ell=1}^L \CB_\ell\conspointcoordinate{\indivmeasure^\ell}.
     \end{equation}
\end{lemma}

\begin{proof}%[Proof of Lemma \ref{lem:sepmeasure_implies_sepcons}]
    Recalling the definition of the orthoprojector $P_\ell = \CB_\ell \CB_\ell^\top$, we observe that we can write
	\begin{equation*} \label{eq:consensus_points_relation}
	\begin{split}
		\conspoint{\indivmeasure}
		&= \sum_{\ell=1}^L P_\ell \conspoint{\indivmeasure}
		 = \sum_{\ell=1}^L \CB_\ell \CB_\ell^\top \conspoint{\indivmeasure} =\sum_{\ell=1}^L \CB_\ell\conspointcoordinate{\indivmeasure^\ell},
	\end{split}
	\end{equation*}
	where the last equality follows from the fact that 
	\begin{equation}
	\begin{split}
		\CB_\ell^\top \conspoint{\indivmeasure}
		&= \int \left(\CB_\ell^\top v \frac{\exp(-\alpha\CE(v))}{\int \exp(-\alpha\CE(\tilde v))\,d\indivmeasure(\tilde{v})}\right) d\indivmeasure(v) \\
		&=  \frac{\int \CB_\ell^\top v\prod_{\tilde\ell=1}^L\exp\big(-\alpha\CE_{\tilde\ell}(\CB_{\tilde\ell}^\top v)\big)\,d\indivmeasure(v)}{\int \prod_{\tilde\ell=1}^L\exp\big(-\alpha\CE_{\tilde\ell}(\CB_{\tilde\ell}^\top\tilde v)\big)\, d\indivmeasure(\tilde v)} \\
		&= \frac{\int \!\cdots\! \int z_\ell\prod_{\tilde\ell=1}^L\exp(-\alpha\CE_{\tilde\ell}(z_{\tilde\ell}))\,d\indivmeasure^1(z_1)\cdots d\indivmeasure^L(z_L)}{\int \!\cdots\! \int \prod_{\tilde\ell=1}^L\exp(-\alpha\CE_{\tilde\ell}(\tilde z_{\tilde\ell}))\, d\indivmeasure^1(\tilde z_1)\cdots d\indivmeasure^L(\tilde z_L)} \\
		&= \int z_\ell \frac{\exp(-\alpha\CE_{\ell}(z_\ell))}{\int\exp(-\alpha\CE_{\ell}(\tilde z_\ell))\, d\indivmeasure^{\ell}(\tilde z_\ell)} d\indivmeasure^{\ell}(z_\ell)=  \conspointcoordinate{\indivmeasure^\ell}.
	\end{split}
	\end{equation}
	Here, the second equality uses the additive separability of $\CE$ as in Assumption \ref{ass:add_sep}, and the third step follows from Fubini's theorem (see e.g.\@ \cite{billingsley1999probability}) together with the fact that $\indivmeasure = \indivmeasure^1 \otimes \cdots \otimes \indivmeasure^L$ by assumption.
\end{proof}

\begin{proof}[Proof of Theorem \ref{thm:separation_anisotropicCBO_dynamics}]
    First note that $\rho^\ell_0=\CB_{\ell}^\top\#\rho_0\in\CP_4(\Rdl)$ for $\ell \in \{ 1, \ldots, L \}$ since $\rho_0\in\CP_4(\RD)$. Consequently, in view of the discussion in Remark \ref{rem:implications_wp} and Lemma \ref{lem:assEl_onE_wellp},
    both the full dynamics~\eqref{eq:aCBO_McKean} and the $L$ component dynamics~\eqref{eq:aCBO_McKean_sepl} are well-posed, showing that $\rho\in\CC([0,T],\CP_4(\RD))$ and $\rho^\ell\in\CC([0,T],\CP_4(\Rdl))$ for all $\ell \in \{ 1, \ldots, L \}$.

    Let us now introduce the short-hand notation $Z_t \coloneqq \sum_{\ell=1}^L \CB_\ell\Xbar{t}^\ell \in \RD$ and  $Z =(Z_t)_{t\in[0,T]}$.
    Comparing the full dynamics $\overbar{X}_t$ with the dynamics $Z_t$ constructed from the component dynamics~$\overbar{X}^{\ell}_t$, we obtain
	\begin{equation} \label{eq:proof:separation_anisotropicCBO_dynamics_1}
	\begin{split}
		\Xbar{t}-Z_t
		&= \left(\Xbar{0}-Z_0\right)
		- \lambda\int_0^t \left(\big(\Xbar{\tau} - \conspoint{\rho_\tau}\big) - \sum_{\ell=1}^L \CB_\ell \big(\Xbar{\tau}^\ell - \conspointcoordinate{\rho_\tau^\ell}\big)\right)d\tau \\
		&\quad\, + \sigma\int_0^t \sum_{k=1}^D\big(\Xbar{\tau} - \conspoint{\rho_\tau}\big)_k\,d(B_\tau)_k e_k^{D} - \int_0^t \sum_{\ell=1}^L \CB_\ell \sum_{k=1}^{d_\ell}\big(\Xbar{\tau}^\ell - \conspointcoordinate{\rho_\tau^\ell}\big)_k\,d(\CB_\ell^\top B_\tau)_k e_k^{d_\ell}\\
		&= \left(\Xbar{0}-Z_0\right)
		- \lambda\int_0^t \left(\big(\Xbar{\tau} - \conspoint{\rho_\tau}\big) - \sum_{\ell=1}^L \CB_\ell \big(\Xbar{\tau}^\ell - \conspointcoordinate{\rho_\tau^\ell}\big)\right)d\tau \\
		&\quad\, + \sigma\int_0^t \sum_{k=1}^D\big(\Xbar{\tau} - \conspoint{\rho_\tau}\big)_k\,d(B_\tau)_k e_k^{D} - \int_0^t \sum_{\ell=1}^L \sum_{k=1}^{D}\left(\CB_{\ell} \big(\Xbar{\tau}^\ell - \conspointcoordinate{\rho_\tau^\ell}\big)\!\right)_kd(B_\tau)_k e_k^{D} \\
		&= \left(\Xbar{0}-Z_0\right)
		- \lambda\int_0^t \left(\left(\Xbar{\tau} - Z_{\tau}\right) -  \left(\conspoint{\rho_\tau} - \sum_{\ell=1}^L \CB_{\ell}\conspointcoordinate{\rho_\tau^\ell}\right)\right)d\tau \\
		&\quad\, + \sigma\int_0^t \sum_{k=1}^D\left(\left(\Xbar{\tau}-Z_{\tau}\right)_{k} - \left(\conspoint{\rho_\tau}-\sum_{\ell=1}^L \CB_{\ell}\conspointcoordinate{\rho_\tau^\ell}\right)_{\!\!k\,} \right)d(B_\tau)_k e_k^{D}.
	\end{split}
	\end{equation}
    Denoting by $\rho_t^{\otimes} \coloneqq \rho_t^1\otimes \cdots\otimes \rho_t^L$ the product measure,
    we notice that $\textrm{Law}(Z_t)= \rho_t^\otimes$ as a consequence of the statistical independence of $\rho^1_t,\ldots,\rho^L_t$ induced by the statistical independence of $\rho_0^1,\ldots,\rho_0^L$ and the statistical independence of $\CB_{1}^\top B_t,\ldots,\CB_L^\top B_t$ by the disjointness of the sets $\CI_1,\ldots,\CI_L$.
	Then, by means of Lemma \ref{lem:sepmeasure_implies_sepcons},
    we can write $\conspoint{\rho_t^\otimes} = \sum_{\ell=1}^L \CB_\ell\conspointcoordinate{\rho_t^\ell}$
    and conclude \eqref{eq:proof:separation_anisotropicCBO_dynamics_1} as 
	\begin{equation}
        \label{eq:proof:separation_anisotropicCBO_dynamics_2}
	\begin{split}
		\Xbar{t}-Z_{t}
		&= \left(\Xbar{0}-Z_{0}\right)
		- \lambda\int_0^t \left(\left(\Xbar{\tau} - Z_{\tau} \right)-  \left(\conspoint{\rho_\tau} - \conspoint{\rho_t^\otimes}\right)\right)d\tau \\
		&\quad\, + \sigma\int_0^t \sum_{k=1}^D\left(\left(\Xbar{\tau}-Z_{\tau}\right)_{k} - \left(\conspoint{\rho_\tau}-\conspoint{\rho_t^\otimes}\right)_{k}\right)d(B_\tau)_k e_k^{D}.
	\end{split}
	\end{equation}
    Taking the squared norm, the supremum over $t \in [0,T]$ and consequently applying the expectation on both sides of~\eqref{eq:proof:separation_anisotropicCBO_dynamics_2},
    we arrive with Young's inequality at
	\begin{equation}
        \label{eq:proof:separation_anisotropicCBO_dynamics_3}
	\begin{split}
	    &\bbE\left( \sup_{t \in [0,T]} \N{\Xbar{t}-Z_{t}}_{2,D}^2 \right)\\
		&\qquad\,\leq 2\bbE\N{\Xbar{0}-Z_{0}}_{2,D}^2
		+4\lambda^2 \bbE \left( \sup_{t \in [0,T]} \N{\int_0^t \left(\left(\Xbar{\tau} - Z_{\tau} \right)-  \left(\conspoint{\rho_\tau} - \conspoint{\rho_\tau^\otimes}\right)\right)d\tau}_{2,D}^2 \right) \\
		&\qquad\, \quad\, + 4\sigma^2 \bbE \left( \sup_{t \in [0,T]} \N{\int_0^t \sum_{k=1}^D\left(\left(\Xbar{\tau}-Z_{\tau}\right)_{k} - \left(\conspoint{\rho_\tau}-\conspoint{\rho_\tau^\otimes}\right)_{k}\right)d(B_\tau)_k e_k^{D}}_{2,D}^2 
        \right)\\
		&\qquad\,\leq 2\bbE\N{\Xbar{0}-Z_{0}}_2^2
        + 4\left(\lambda^2T+c_{\text{BDG},2} \sigma^2\right) \bbE\int_0^T \N{\left(\Xbar{t} - Z_{t} \right)-  \left(\conspoint{\rho_t} - \conspoint{\rho_t^\otimes}\right)}_{2,D}^2dt \\
		&\qquad\,\leq 2\bbE\N{\Xbar{0}-Z_{0}}_{2,D}^2
		+ 8\left(\lambda^2T+c_{\text{BDG},2}\sigma^2\right) \int_0^T \bbE\N{\Xbar{t} - Z_{t}}_{2,D}^2 + \N{\conspoint{\rho_t}-\conspoint{\rho_t^\otimes}}_{2,D}^2dt. 
	\end{split}
	\end{equation}
    where the second step uses Jensen's inequality for the second summand, and the Burkholder--Davis--Gundy inequality~\cite[Theorem~7.3]{mao2007stochastic} with constant $c_{\text{BDG},2}=4$ for the third.
	Using now that $\textrm{Law}(Z_t) = \rho_t^\otimes$, the stability estimate~\cite[Lemma~3.2]{carrillo2018analytical} is applicable as a consequence of the assumptions on the individual components $\CE_\ell$, and the $\CP_4$-regularity of the laws $\rho_t$ and $\rho^\ell_t$ for all $\ell \in \{ 1, \ldots, L \}$.
    Hence, we have that
	\begin{equation}
		\N{\conspoint{\rho_t} - \conspoint{\rho_t^\otimes}}_{2,D}^2 \leq c W_2^2(\rho_t, \rho_t^\otimes) \leq c\bbE\N{\Xbar{t}-Z_t}_{2,D}^2.
	\end{equation}
	Inserting this into \eqref{eq:proof:separation_anisotropicCBO_dynamics_3}, we eventually arrive at the bound
	\begin{equation}
	\begin{split}
        &\bbE \left( \sup_{t \in [0,T]} \N{\Xbar{t}-Z_{t}}_{2,D}^2 \right)
        \leq 2\bbE\N{\Xbar{0}-Z_{0}}_{2,D}^2
		+ 8\left(\lambda^2T+c_{\text{BDG},2}\sigma^2\right)(1+c) \int_0^T \bbE \left( \sup_{t \in [0,T]}\N{\Xbar{t} - Z_{t}}_{2,D}^2dt \right),
	\end{split}
	\end{equation}
	which allows for an application of Gr\"onwall's inequality yielding
	\begin{equation}
    \label{eq:2normzero2}
		\bbE \left( \sup_{t \in [0,T]}\N{\Xbar{t}-Z_t}_{2,D}^2 \right)
		= 2\bbE\N{\Xbar{0}-Z_0}_{2,D}^2 \exp\left(8\left(\lambda^2T+c_{\text{BDG},2}\sigma^2\right)(1+c)t\right)
		\equiv0
	\end{equation}
	since $\bbE\N{\Xbar{0}-Z_0}_{2,D}^2 = 0$ as ensured by the assumption on the initial measures.
    This proves the first claim.

    The equality between the path-wise laws $\rho_{[0,T]}$ and $\rho^{\otimes}_{[0,T]} = \rho^1_{[0,T]} \otimes \cdots \otimes \rho^L_{[0,T]}$ is a direct consequence of the observation that with \eqref{eq:2normzero2} it holds
    \begin{equation}
        \label{eq:2normzero2_measures}
        \CW_2^2(\rho_{[0,T]}, \rho^{\otimes}_{[0,T]}) \le \bbE \left( \sup_{t \in [0,T]}\N{\Xbar{t}-Z_t}_{2,D}^2 \right)  \equiv 0,
    \end{equation}
    which concludes the proof.
    Here $\CW_2$ denotes the path-wise Wasserstein-$2$ distance, associated with the continuous path space endowed with the uniform topology, see \cite{chaintron2022propagation} for more details.
\end{proof}

\begin{remark}
We note that the converse of Lemma \ref{lem:sepmeasure_implies_sepcons} does not hold for general measures $\indivmeasure, \indivmeasure^\ell$ (namely, it is not true, in general, that the separation of the consensus points implies the factorization of the measure $\indivmeasure$ as the product of the $\indivmeasure^\ell$'s).
However, it is valid under the assumptions of Theorem \ref{thm:separation_anisotropicCBO_dynamics}, provided that we recognize the measures $\indivmeasure, \indivmeasure^\ell$ as distributions of the stochastic processes  $\overbar{X}$ and $Z$ respectively.
The proof follows directly from the proof of the separation result presented above, by using Equation \eqref{eq:proof:separation_anisotropicCBO_dynamics_2} along with the observation of inequality \eqref{eq:2normzero2_measures}.
\end{remark}

%%%%%%%%%%%%%%%%%%%%%%%%%%%%%%%%%%%%%%%%%%%%%%%%%%%%%%%%
%%%%%%%%%%%%%%%%%%%%%%%%%%%%%%%%%%%%%%%%%%%%%%%%%%%%%%%%
\subsection{Proof details for Theorem \ref{thm:conv_dlg1_Vk}}
\label{sec:mainresults_convmf_proof}

Our proof of the convergence result in Theorem~\ref{thm:conv_dlg1_Vk} follows the framework developed in \cite{fornasier2024consensus,fornasier2021convergence,riedl2024perspective}.
For general objective functions $\CE$, we first study in Subsection \ref{subsec:Vk_evolution} the evolution of the coordinate-wise functionals $\CV_k$ defined in \eqref{def:Vk} and infer therefrom in Subsection \ref{subsec:Vinf_evolution} the evolution of the functional $\CV_\infty$ defined in \eqref{def:Vinfty}.
Then, in Subsection \ref{subsec:QLPl}, we establish a quantitative Laplace principle under the assumption that $\CE$ satisfies the additive separability property in Assumption \ref{ass:add_sep}.
Eventually,
we prove Theorem \ref{thm:conv_dlg1_Vk} in Subsection~\ref{subsec:proof_Vk_dlg1}.

%%%%%%%%%%%%%%%%%%%%%%%%%%%%%%%%%%%%%%%%%%%%%%%%%%%%%%%%%%%%%%%%

\subsubsection{Evolution of the functionals \texorpdfstring{$\Vk$,}{Vk} \texorpdfstring{$k \in \{1,\dots,D\}$}{}}
\label{subsec:Vk_evolution}

We first derive evolution inequalities for the energy (Lyapunov) functionals $\CV_k$ defined in \eqref{def:Vk} for $k \in \{1,\dots,D\}$.

Throughout this subsection, no additive separability assumption as in Assumption \ref{ass:add_sep} is required on $\CE$.

\begin{lemma} \label{lem:evolutionVk_upper}
	Let $\CE:\RD\rightarrow\R$, and fix $\alpha,\lambda,\sigma > 0$.
	Moreover, let $T>0$ and let $\rho \in \CC([0,T], \CP_4(\RD))$ be a weak solution to the Fokker-Planck equation~\eqref{eq:aCBO_mf}.
	Then, for all $k \in \{1,\dots,D\}$, the functionals $\CV_k(\rho_t)$ satisfy
	\begin{equation}
    \label{eq:evolutionVk_upper}
	    \frac{d}{dt}\CV_k(\rho_t) \leq -\left(2\lambda-\sigma^2\right) \CV_k(\rho_t) 
        + \sqrt{2}\left(\lambda+\sigma^2\right) \sqrt{\CV_k(\rho_t)} \SN{\left(\conspoint{\rho_t}-\globmin\right)_k}
        + \frac{\sigma^2}{2} \left(\conspoint{\rho_t}-\globmin\right)_k^2
	\end{equation}
    for all $t \in (0,T)$.
\end{lemma}

\begin{proof}
    We note that the function $\phi_k(\dummy) = 1/2\left(\dummy - \globmin\right)_k^2$ is in $\CC_*^2(\RD)$ (with $\CC_*^2(\RD)$ as specified in Definition~\ref{def:fokker_planck_weak_sense}) and recall that $\rho_t$ satisfies the weak solution identity~\eqref{eq:aCBO_mf_weak} for all test functions in $\CC_*^2(\RD)$ and $t \in (0,T)$.
    Hence, by applying \eqref{eq:aCBO_mf_weak} with $\phi_k$ as above,
    we obtain for the evolution of $\CV_k(\rho_t)$ that 
    \begin{align*}
    	\frac{d}{dt} \CV_k(\rho_t)
    	&=  -\lambda\int \sum_{k'_k=1}^D ( x - \conspoint{\rho_t})_{k'_k} \partial_{k'_k} \phi_k(x) \, d\rho_t(x) +  \frac{\sigma^2}{2} \int \sum_{{k'_k}=1}^D \left(x-\conspoint{\rho_t}\right)_{k'_k}^2 \partial^2_{k'_kk'_k} \phi_{k}(x) \,d\rho_t(x)\\
    	&=  \underbrace{-\lambda \int \left(x- \globmin \right)_k \left(x-\conspoint{\rho_t}\right)_k d\rho_t(x)}_{=:T_{1k}} +  \underbrace{\frac{\sigma^2}{2} \int \left(x-\conspoint{\rho_t}\right)_k^2 d\rho_t(x)}_{=:T_{2k}},
    \end{align*}
    where we used $\partial_{k'_k} \phi_k(x) = \left(x- \globmin \right)_k \delta_{kk'_k}$ and $\partial^2_{k'_kk'_k} \phi_k(x) = \delta_{kk'_k}$.
    We can now follow the steps taken in~\cite[Lemma~4.1]{fornasier2024consensus}.
    Subtracting and adding $\left( \globmin\right)_k$ in $T_{1k}$ yields
    \begin{align*}
    	T_{1k} &=-\lambda \int \left(x- \globmin\right)_k \left(x- \globmin\right)_k d\rho_t(x) + \lambda \int \left(x- \globmin\right)_k  \, \left(\conspoint{\rho_t} -  \globmin\right)_k d\rho_t(x)\\
    	&= -2\lambda \CV_k(\rho_t) + \lambda \left(\bbE(\rho_t) -  \globmin\right)_k \left(\conspoint{\rho_t}-\globmin\right)_k\\
        &\le -2\lambda \CV_k(\rho_t) + \lambda \SN{\left(\bbE(\rho_t) -  \globmin\right)_k} \SN{\left(\conspoint{\rho_t}-\globmin\right)_k}.
    \end{align*}
    Similarly, again by subtracting and adding $\left( \globmin\right)_k$ in $T_{2k}$, we obtain
    \begin{equation}
    \begin{split}
        T_{2k} &= \frac{\sigma^2}{2} \int \left(x-\conspoint{\rho_t}\right)_k^2  d\rho_t(x) \\
    	\nonumber
    	&= \frac{\sigma^2}{2} \left(\int \left(x-\globmin\right)_k^2 d\rho_t(x) - 2 \int \left(x-\globmin\right)_k  \, \left(\conspoint{\rho_t}-\globmin\right)_k d\rho_t(x) + \left(\conspoint{\rho_t}-\globmin\right)_k^2\right)\\
    	&=\sigma^2 \left(\CV_k(\rho_t) - 
         \left(\bbE(\rho_t) -  \globmin\right)_k \left(\conspoint{\rho_t}-\globmin\right)_k
        + \frac{1}{2}\left(\conspoint{\rho_t}- \globmin\right)_k^2 \right)\\
        & \le \sigma^2 \left(\CV_k(\rho_t) + 
        \SN{\left(\bbE(\rho_t) -  \globmin\right)_k} \SN{\left(\conspoint{\rho_t}-\globmin\right)_k}
        + \frac{1}{2}\left(\conspoint{\rho_t}- \globmin\right)_k^2 \right).
    \end{split}
    \end{equation}
    The result now follows after noting than
    \begin{align*}
    	\SN{\left(\bbE(\rho_t) -  \globmin\right)_k} \leq \int \SN{\left(x- \globmin\right)_k} d\rho_t(v) \leq \sqrt{\int \left(v- \globmin\right)_k^2 d\rho_t(v)}=\sqrt{2\CV_k(\rho_t)},
    \end{align*}
    as a consequence of Jensen's inequality.
\end{proof}

\begin{lemma}
\label{lem:evolutionVk_lower}
    Under the assumptions of Lemma~\ref{lem:evolutionVk_upper}, for all $k=1,\dots,D$, the functionals $\CV_k(\rho_t)$ satisfy
    	\begin{equation}
            \label{eq:evolutionVk_lower}
    	    \frac{d}{dt}\CV_k(\rho_t) \geq -\left(2\lambda-\sigma^2\right) \CV_k(\rho_t) 
            - \sqrt{2}\left(\lambda+\sigma^2\right) \sqrt{\CV_k(\rho_t)} \SN{\left(\conspoint{\rho_t}-\globmin\right)_k}
    	\end{equation}
        for all $t \in (0,T)$.
\end{lemma}

\begin{proof}
    The lower bound is proven by following the lines of the proof of Lemma~\ref{lem:evolutionVk_upper} and noticing that $\left(\conspoint{\rho_t}-\globmin\right)_k^2 \ge 0$.
\end{proof}

\subsubsection{Evolution of the functional \texorpdfstring{$\CV_\infty$}{Vinf}}
\label{subsec:Vinf_evolution}

Since the functional $\CV_\infty$ is the maximum of finitely many $\CV_k$,
we can use Rademacher's theorem 
(see, e.g., \cite[Section~3.1.2, Theorem~2]{evans2025measure}) to derive evolution inequalities for $\CV_\infty$ that hold almost everywhere.

Throughout this subsection, no additive separability assumption as in Assumption \ref{ass:add_sep} is required on $\CE$.

\begin{lemma} \label{lem:evolutionVinf_upper}
    Let $\CE:\RD\rightarrow\R$, and fix $\alpha,\lambda,\sigma > 0$.
    Moreover, let $T>0$ and let $\rho \in \CC([0,T], \CP_4(\RD))$ be a weak solution to the Fokker-Planck equation~\eqref{eq:aCBO_mf}.
    Then, the functional $\CV_\infty(\rho_t)$ satisfies
    \begin{equation}
    \label{eq:evolutionVinf_upper}
    \begin{split}
        \frac{d}{dt}\CV_\infty(\rho_t) \leq -\left(2\lambda-\sigma^2\right) \CV_\infty(\rho_t) 
        &+ \sqrt{2}\left(\lambda+\sigma^2\right) \sqrt{\CV_\infty(\rho_t)} \max_{k \in \{1,\dots,D\}} \SN{\left(\conspoint{\rho_t}-\globmin\right)_k} \\
        &+ \frac{\sigma^2}{2} \max_{k \in \{1,\dots,D\}} \left(\conspoint{\rho_t}-\globmin\right)_k^2
    \end{split}
    \end{equation}
    for almost every $t \in (0,T)$.
\end{lemma}

\begin{proof}
    For each $k \in \{1, \dots, D\}$,
    as observed in Lemma~\ref{lem:evolutionVk_upper},
    the functional $\CV_k(\rho_t)$ is continuously differentiable for any $t \in (0,T)$. 
    Hence, $\CV_\infty(\dummy) = \max_k \CV_k(\dummy)$ is locally Lipschitz as the maximum of finitely many continuously differentiable functions.
    By Rademacher's theorem, $\CV_\infty(\rho_t)$ is differentiable for almost every $t \in (0,T)$.
   
    Let $t$ be a point of differentiability for $\CV_\infty$ and let $\bar{k} \in I(t) \coloneqq \{ k \in \{1,\ldots,D\}: \CV_{k}(\rho_t) = \CV_\infty(\rho_t) \}$ be an index such that $\CV_{\bar{k}}(\rho_t) = \CV_\infty(\rho_t)$. At such a point $t$, it holds
    \begin{equation*}
        \frac{d}{dt} \CV_\infty(\rho_t) = \frac{d}{dt} \CV_{\bar{k}}(\rho_t).
    \end{equation*}
    % \tm{for every $\bar k\in I(t)$.}
    Invoking the coordinate-wise result from Lemma~\ref{lem:evolutionVk_upper} for the index $\bar{k}$, we have
    \begin{align*}
        \frac{d}{dt}\CV_{\bar{k}}(\rho_t) \leq -\left(2\lambda-\sigma^2\right) \CV_{\bar{k}}(\rho_t) 
        + \sqrt{2}\left(\lambda+\sigma^2\right) \sqrt{\CV_{\bar{k}}(\rho_t)} \SN{\left(\conspoint{\rho_t}-\globmin\right)_{\bar{k}}}
        + \frac{\sigma^2}{2} \left(\conspoint{\rho_t}-\globmin\right)_{\bar{k}}^2.
    \end{align*}
    Using the definition of the index $\bar{k}$, and upper-bounding the error terms $\SN{\left(\conspoint{\rho_t}-\globmin\right)_{\bar{k}}}$ by the maximum over all coordinates $k$, we obtain that \eqref{eq:evolutionVinf_upper}  holds for almost every $t \in (0,T)$.
\end{proof}

\begin{lemma} \label{lem:evolutionVinf_lower}
    Under the assumptions of Lemma~\ref{lem:evolutionVinf_upper}, the functional $\CV_\infty(\rho_t)$ satisfies
    \begin{equation}
    \label{eq:evolutionVinf_lower}
        \frac{d}{dt}\CV_\infty(\rho_t) \geq -\left(2\lambda-\sigma^2\right) \CV_\infty(\rho_t) 
        - \sqrt{2}\left(\lambda+\sigma^2\right) \sqrt{\CV_\infty(\rho_t)} \max_{k=1,\dots,D} \SN{\left(\conspoint{\rho_t}-\globmin\right)_k}
    \end{equation}
    for almost every $t \in (0,T)$.
\end{lemma}

\begin{proof}
    As established in the proof of Lemma~\ref{lem:evolutionVinf_upper}, $\CV_\infty(\rho_t)$ is differentiable for almost every $t \in (0,T)$ and at any point of differentiability $t$, there exists an index $\bar{k} \in I(t) \coloneqq \{ k \in \{1,\ldots,D\}: \CV_{k}(\rho_t) = \CV_\infty(\rho_t) \}$ such that $\CV_{\bar{k}}(\rho_t) = \CV_\infty(\rho_t)$ and $\frac{d}{dt} \CV_\infty(\rho_t) = \frac{d}{dt} \CV_{\bar{k}}(\rho_t)$.
    Applying the coordinate-wise lower bound from Lemma~\ref{lem:evolutionVk_lower} to the index $\bar{k}$, we have
    \begin{equation*}
        \frac{d}{dt}\CV_{\bar{k}}(\rho_t) \geq -\left(2\lambda-\sigma^2\right) \CV_{\bar{k}}(\rho_t) 
        - \sqrt{2}\left(\lambda+\sigma^2\right) \sqrt{\CV_{\bar{k}}(\rho_t)} \SN{\left(\conspoint{\rho_t}-\globmin\right)_{\bar{k}}}.
    \end{equation*}
    The result follows directly from the definition of \(\bar{k}\) and the upper bound on the error terms $\SN{\left(\conspoint{\rho_t}-\globmin\right)_{\bar{k}}}$, as already done in Lemma~\ref{lem:evolutionVinf_upper}.
\end{proof}

%%%%%%%%%%%%%%%%%%%%%%%%%%%%%%%%%%%%%%%%%%%%%%%%%%%%%%%%%%%%%%%%
\subsubsection{Quantitative Laplace principle for \texorpdfstring{$\CE_\ell$,}{El} \texorpdfstring{$\ell \in \{ 1, \ldots, L \}$}{}}
\label{subsec:QLPl}

In this section, we quantitatively characterize the magnitude of $\abs{\conspoint{\rho_t}-\globmin}_k$ by further assuming that $\CE_\ell, \ell \in \{ 1, \ldots, L \},$ satifies Assumption \ref{ass:invcont_QLPl2}.
We remark that the choice of the term to be estimated strictly depends on the evolution inequalities of the energy functionals $\CV_k$ \eqref{def:Vk} derived in Section \ref{subsec:Vk_evolution}.

\begin{proposition}%[$(QLP)_\ell$ for $\CV_k$]
\label{prop:QLPl_Vk}
    Let $\CE:\RD\rightarrow\R$ be additively separable as in Assumption \ref{ass:add_sep}.
    Let $\indivmeasure^{\ell} \in \mathcal{P}(\Rdl)$ for all $\ell \in \{ 1, \ldots, L \}$ and let $\indivmeasure \in \mathcal{P}(\RD)$ be such that $\indivmeasure = \indivmeasure^1 \otimes \cdots \otimes \indivmeasure^L$.
    Fix $\alpha > 0$ and $k \in \{1,\dots,D\}$. 
    Then, there exist $\ell_k \in \{1,\dots,L\}$ such that $k \in \CI_{\ell_k}$ and $k'_k \in \{1,\dots,d_{\ell_k}\}$ such that 
    \begin{equation}
        \label{eq:simplification_diffk}
        \left| \left( \conspoint{\indivmeasure} - \globmin \right)_k \right| = \SN{ \left( x_{\alpha}^{\ell_k}(\indivmeasure^{\ell_k}) \right)_{k'_k} -  \left( \CB_{\ell_k}^\top \globmin \right)_{k'_k}}.
    \end{equation}
    For any $r>0$, define $\CE_{r,\ell_k} \coloneqq \sup_{v \in B^{d_{\ell_k}}_{r}(\CB_{\ell_k}^\top\globmin)} \CE_{\ell_k}(v)$.    
    Then, under the inverse continuity property, Assumption~\ref{ass:invcont_QLPl2}, on $\CE_{\ell_k}$ and assuming w.l.o.g.\@ that $\CEunder_{\ell_k} = 0$,
    for any $\rlk \in (0,R_{0,{\ell_k}}]$ and $\qlk>0$ such that $\qlk + \CE_{\rlk,{\ell_k}} < \CE_{\infty,{\ell_k}}$, we have
    \begin{equation}
    \label{eq:QLPl_Vk}
        \left| \left( \conspoint{\indivmeasure} - \globmin \right)_k \right| 
        \le \frac{(\qlk +\CE_{\rlk,{\ell_k}})^{\nu_{\ell_k}}}{\eta_{\ell_k}} + \frac{\exp(-\alpha \qlk)}{\indivmeasure^{\ell_k}\!\left(\B^{d_{\ell_k}}_{\rlk}(\CB_{\ell_k}^\top\globmin)\right)} \int_{\R^{d_{\ell_k}}} \left| v_{k'_k}- \left( \CB_{\ell_k}^\top \globmin \right)_{k'_k} \right| d\indivmeasure^{\ell_k}(v).
    \end{equation}
\end{proposition}

\begin{proof}
    A simple application of Lemma \ref{lem:sepmeasure_implies_sepcons} along with the definition of the orthoprojector $\CP_\ell = \CB_\ell \CB_\ell^\top$ shows
    \begin{equation}
    \begin{split}
        \left| \left( \conspoint{\indivmeasure} - \globmin \right)_k \right|
        = \SN{\left( \conspoint{\indivmeasure} \right)_k  - \globmin_k}
        &= \left| \left( \sum_{\ell=1}^L \CB_\ell \conspointcoordinate{\indivmeasure^\ell} \right)_k - \left(\sum_{\ell=1}^L  \CB_\ell \left(\CB_\ell^\top \globmin \right) \right)_k \right| \\
        &= \left| \left( x_{\alpha}^{\ell_k}(\indivmeasure^{\ell_k}) \right)_{k'_k} -  \left( \CB_{\ell_k}^\top \globmin \right)_{k'_k}  \right|
    \end{split}
    \end{equation}
    for some index $\ell_k \in \{1,\dots,L\}$ such that $k \in \CI_{\ell_k}$ and $k'_k \in \{1,\dots,d_{\ell_k}\}$.
    This justifies \eqref{eq:simplification_diffk}.
    The bound \eqref{eq:QLPl_Vk} is obtained by adapting the proof of \cite[Proposition~4.5]{fornasier2024consensus}.
    Using the definition of $x_{\alpha}^{\ell_k}(\indivmeasure^{\ell_k})$ and Jensen's inequality,
    we can write
    \begin{equation}
    \label{eq:QLPl_Vk_proof1}
        \left| \left( x_{\alpha}^{\ell_k}(\indivmeasure^{\ell_k}) \right)_{k'_k} -  \left( \CB_{\ell_k}^\top \globmin \right)_{k'_k}  \right| \le 
        \int_{\R^{d_{\ell_k}}} \left| v_{k'_k} - \left( \CB_{\ell_k}^\top \globmin \right)_{k'_k} \right| \frac{\omega_{\alpha}^{\ell_k}(v)}{\Nbig{\omega_{\alpha}^{\ell_k}}_{L^1(\indivmeasure^{\ell_k})}} d \indivmeasure^{\ell_k}(v).
    \end{equation}
    Now let $\tilde{r}_{\ell_k} > \rlk >0$.
    We can decompose the integral on the right-hand side of \eqref{eq:QLPl_Vk_proof1} into the sum of two integrals:
    one over the ball $\B^{d_{\ell_k}}_{\tilde{r}_{\ell_k}}(\CB_{\ell_k}^\top\globmin)$ and the other over its complement.
    For the first of which, we estimate 
    \begin{equation}
    \label{eq:QLPl_Vk_proof2}
        \int_{\B^{d_{\ell_k}}_{\tilde{r}_{\ell_k}}(\CB_{\ell_k}^\top\globmin)} \left| v_{k'_k} - \left( \CB_{\ell_k}^\top \globmin \right)_{k'_k} \right| \frac{\omega_{\alpha}^{\ell_k}(v)}{\Nbig{\omega_{\alpha}^{\ell_k}}_{L^1(\indivmeasure^{\ell_k})}} d \indivmeasure^{\ell_k}(v) \le \tilde{r}_{\ell_k}
    \end{equation}
    since $\big| v_{k'_k}- \left( \CB_{\ell_k}^\top \globmin \right)_{k'_k} \!\big| \le \N{v-\CB_{\ell_k}^\top \globmin}_{\infty,d_\ell}$, which is bounded by $\tilde{r}_{\ell_k}$ on the set $\B^{d_{\ell_k}}_{\tilde{r}_{\ell_k}}(\CB_{\ell_k}^\top\globmin)$.
    For the latter, recalling the definition of $\CE_{r,\ell_k}$ from the statement, it holds 
    \begin{equation}
    \label{eq:QLPl_Vk_proof3}
    \begin{split}
        &\int_{\left( \B^{d_{\ell_k}}_{\tilde{r}_{\ell_k}}(\CB_{\ell_k}^\top\globmin) \right)^c} \left| v_{k'_k} - \left( \CB_{\ell_k}^\top \globmin \right)_{k'_k} \right| \frac{\omega_{\alpha}^{\ell_k}(v)}{\Nbig{\omega_{\alpha}^{\ell_k}}_{L^1(\indivmeasure^{\ell_k})}} d \indivmeasure^{\ell_k}(v)\\
        &\qquad\quad\,\le 
        \frac{\exp \left( -\alpha \left( \inf_{v \in \left( \B^{d_{\ell_k}}_{\tilde{r}_{\ell_k}}(\CB_{\ell_k}^\top\globmin) \right)^c} \CE_{\ell_k}(v) - \CE_{\rlk,\ell_k} \right) \right) }{\indivmeasure^{\ell_k}\left(\B^{d_{\ell_k}}_{\rlk}(\CB_{\ell_k}^\top\globmin)\right)} \int_{\R^{d_{\ell_k}}} \left| v_{k'_k} - \left( \CB_{\ell_k}^\top \globmin \right)_{k'_k} \right| d \indivmeasure^{\ell_k}(v).
    \end{split}
    \end{equation}
    Choosing $\tilde{r}_{\ell_k} = \big(\qlk + \CE_{\rlk,{\ell_k}}\big)^{\nu_{\ell_k}}/\eta_{\ell_k}$, applying the inverse continuity assumption, Assumption \ref{ass:invcont_QLPl2},
    and plugging \eqref{eq:QLPl_Vk_proof2} and \eqref{eq:QLPl_Vk_proof3} into \eqref{eq:QLPl_Vk_proof1} yields the desired bound \eqref{eq:QLPl_Vk}.
\end{proof} 

\begin{remark}
    \label{rem:whynotstdQLP}
    The proof of \cite[Proposition~4.5]{fornasier2024consensus} needs to be adapted as the integral on the right‑hand side of \eqref{eq:QLPl_Vk_proof1} is taken with respect to $\R^{d_{\ell_k}}$ while the integrand itself is the absolute value of a real number (rather than the $\N{\dummy}_{2,\ell_k}$ norm of a $d_{\ell_k}$-dimensional vector).
\end{remark}

%%%%%%%%%%%%%%%%%%%%%%%%%%%%%%%%%%%%%%%%%%%%%%%%%%%%%%%%%%%%%%%%

\subsubsection{Proof of Theorem \ref{thm:conv_dlg1_Vk}}
\label{subsec:proof_Vk_dlg1}

We now present the proof of Theorem \ref{thm:conv_dlg1_Vk}, which is about the convergence of anisotropic CBO for separable functions.
To do this, we analyze the functional $\CV_{\infty}$.
As already mentioned earlier,
our proof follows the framework developed in \cite{fornasier2024consensus,fornasier2021convergence,riedl2024perspective}.
It combines the Lyapunov-type estimates \eqref{eq:evolutionVk_upper} and \eqref{eq:evolutionVk_lower}, with the quantitative Laplace principle developed in Proposition~\ref{prop:QLPl_Vk}.
To be able to apply the latter,
we require lower-bound estimates of the probability mass surrounding the global minimizer \cite[Proposition~2]{fornasier2021convergence}.
To  facilitate cross-referencing, 
we present in what follows such result for a generic dimension $d$.
In the proof thereafter, we will specify the value of $d$.

\begin{proposition}[{\cite[Proposition~2]{fornasier2021convergence}}]
    \label{prop:lower_bound}

    Fix $\alpha,\lambda,\sigma > 0$. Let $r>0$, $T>0$ and let $\mu \in \CC([0,T], \CP(\Rd))$ be a weak solution to the Fokker-Planck equation~\eqref{eq:aCBO_mf_d}
    with initial condition $\mu_0\in\mathcal{P}(\mathbb{R}^d)$ and for $t\in[0,T]$.
    Then, for all $t\in[0,T]$, we have
    \begin{equation}
        \mu_t(B^\infty_r(\globmin))\ge\left(\int\phi_r(v)\,d\mu_0(v)\right)\exp(-pt)
    \end{equation}
    with $\phi_r:\bbR^{d}\rightarrow\bbR$ a suitable mollifier (as provided in \cite[Lemma~2]{fornasier2021convergence}) and 
    \begin{equation}
        p\coloneqq2d\max \left\{\frac{\lambda(cr+\beta\sqrt{c})}{(1-c)^2r}+\frac{\sigma^2(cr^2+\beta^2)(2c+1)}{(1-c)^4r^2},\frac{2\lambda^2}{(2c-1)\sigma^2}\right\}
    \end{equation}
    for any $\beta<\infty$ with $\sup_{t\in[0,T]} \N{\conspoint{\rho_t}-\globmin}_{\infty,d}\le \beta$ and for any $c\in(1/2,1)$ satisfying $(1-c)^2\le (2c-1)c$.
\end{proposition}

With this, we now have all necessary tools at hand to present a detailed proof of Theorem~\ref{thm:conv_dlg1_Vk}.

\begin{proof}[Proof of Theorem~\ref{thm:conv_dlg1_Vk}]
    Without loss of generality, we may assume that $\CEunder_\ell = 0$ for all $\ell\in\{1,\dots,L\}$.
    Let $\varepsilon\in(0,\CV_\infty(\rho_{0}))$ be given.
    We define for all $\ell \in \{ 1, \ldots, L \}$
    \begin{equation}
         \label{def:alpha0l}
         \alphazerol
         \coloneqq \frac{1}{\qle} \left( \log\left( \frac{2^{d_{\ell}+1} \sqrt{2\CV_{\infty}(\rho_0)}}{c(\vartheta,\lambda,\sigma)\sqrt{\varepsilon}} \right) + \frac{\ple}{(1-\vartheta)\left(2\lambda-\sigma^2\right)} \log\left( \frac{\CV_{\infty}(\rho_0)}{\varepsilon} \right) - \log \left( \rho_0^{\ell} \left( \B^{d_{\ell}}_{\rle/2}(\CB^\top_{\ell} \globmin) \right) \right) \right) 
    \end{equation}
    where (here, $\eta_\ell, \nu_\ell, \CE_{\infty,\ell}, R_{0,\ell}$ are the parameters from Assumption~\ref{ass:invcont_QLPl2})
    \begin{equation}
        \qle\coloneqq\frac{1}{2}\min\left\{ \left(  \eta_{\ell}\frac{c(\vartheta,\lambda,\sigma)\sqrt{\varepsilon}}{2}\right)^{1/\nu_{\ell}},\CE_{\infty,\ell}\right\}
        \quad \text{ and } \quad
        \rle \coloneqq \max_{s\in[0,R_{0,\ell}]}\left\{ \max_{v \in \B^{d_{\ell}}_{s}(\CB_{\ell}^\top \globmin)} \CE_{\ell}(v)\leq \qle\right\}
    \end{equation}
    with
    \begin{equation}
        c(\vartheta,\lambda,\sigma)\coloneqq\min \left\{ \frac{\vartheta}{2}\frac{\left(2\lambda-\sigma^2\right)}{\sqrt{2}(\lambda+\sigma^2)},\sqrt{\vartheta \frac{\left(2\lambda-\sigma^2\right)}{\sigma^2}}\right\}
    \end{equation}
    and where $\ple$ is as defined in Proposition~\ref{prop:lower_bound} with $d=d_\ell$, $\beta = c(\vartheta,\lambda,\sigma) \sqrt{\CV_{\infty}(\rho_0)}$ and $r = \rle$.
    We note that 
    $\qle >0$, $\rle \in (0,R_{0,\ell})$ and $\qle + \CE_{\rle} \le \CE_{\infty,\ell}$ by construction.
    
    We now select the parameter $\alpha$ to satisfy 
    \begin{equation}
    \label{def:alpha}
        \alpha
        > \alpha_0
        = \max_{\ell \in \{ 1, \ldots, L \}} \alphazerol.
    \end{equation} 
    Moreover, we define the time horizon $T_{\alpha} \ge 0$ according to
     \begin{equation}
        \label{def:Talpha}
            T_\alpha\coloneqq\sup \left\{t\geq0: \CV_\infty(\rho_{t'})>\varepsilon , \quad \SN{(\conspoint{\rho_{t'}}-\globmin)_{k}}<C_\infty(t') \quad \text{for all $t'\in[0,t]$ and $k \in \{1,\dots,D\}$} \right\}
        \end{equation}
    with $C_\infty(t)\coloneqq c(\vartheta,\lambda,\sigma) \sqrt{\CV_\infty(\rho_t)}$.
    Our goal is to prove that $ \CV_{\infty}(\rho_{T_{\alpha}})=\varepsilon$ with $T_{\alpha} \in \big[\frac{1-\vartheta}{1+\vartheta/2}T^*,T^*\big]$ and that the functional $\CV_{\infty}(\rho_t)$ decays at least exponentially until the accuracy $\varepsilon$ is reached at time $T_{\alpha}$.
    
    The first step involves demonstrating that $T_{\alpha}>0$. The regularity $\rho\in \mathcal{C}([0,T],\mathcal{P}_4(\RD))$ implies the continuity of the mappings $t\mapsto\CV_k(\rho_t)$ and $t\mapsto \SN{(\conspoint{\rho_t}-\globmin)_{k}}$ for all $k \in \{1,\dots,D\}$, where we additionally used \cite[Lemma~3.2]{carrillo2018analytical} for the latter.
    In particular, $t\mapsto\CV_\infty(\rho_t)$ is continuous as the maximum of continuous functions. Since $\CV_{\infty}(\rho_0)>\varepsilon$,
    in order to conclude that $T_{\alpha}>0$, it suffices to prove that   $\SN{(\conspoint{\rho_0}-\globmin)_{k}}< C_{\infty}(0)$ for all $k \in \{1,\dots,D\}$.
    To show this, fix $k \in \{1,\dots,D\}$. 
    Then, there exist $\ell_k \in \{1,\dots,L\}$ such that $k \in \CI_{\ell_k}$ and $k'_k  \in \{1,\dots, d_{\ell_k}\}$ such that \eqref{eq:simplification_diffk} holds for $\indivmeasure = \rho_0$. 
    An application of Proposition \ref{prop:QLPl_Vk} yields
    \begin{equation*}
    \begin{split}
        \SN{\left( \conspoint{\rho_0} - \globmin \right)_k} 
        &\le \frac{(q_{\ell_k,\varepsilon} +\CE_{r_{\ell_k,\varepsilon},{\ell_k}})^{\nu_{\ell_k}}}{\eta_{\ell_k}} + \frac{\exp\left(-\alpha q_{\ell_k,\varepsilon}\right)}{\rho_{0}^{\ell_k}\left(\B^{d_{\ell_k}}_{r_{\ell_k,\varepsilon}}(\CB_{\ell_k}^\top\globmin)\right)} \int_{\R^{d_{\ell_k}}} \left| v_{k'_k}- \left( \CB_{\ell_k}^\top \globmin \right)_{k'_k} \right| d\rho_{0}^{\ell_k}(v)\\
        &= \frac{(q_{\ell_k,\varepsilon} +\CE_{r_{\ell_k,\varepsilon},{\ell_k}})^{\nu_{\ell_k}}}{\eta_{\ell_k}} + \frac{\exp\left(-\alpha q_{\ell_k,\varepsilon}\right)}{\rho_{0}^{\ell_k}\left(\B^{d_{\ell_k}}_{r_{\ell_k,\varepsilon}}(\CB_{\ell_k}^\top\globmin)\right)} \int_{\RD} \left| x_{k}- \globmin \right| d\rho_{0}(x).
    \end{split}
    \end{equation*}
    The first summand of the right-hand side can be upper bounded by $c(\vartheta,\lambda,\sigma)\sqrt{\varepsilon}/2$ by using the definitions of $q_{\ell_k,\varepsilon}$ and $r_{\ell_k,\varepsilon}$.
    More precisely, it holds 
    \begin{equation*}
        q_{\ell_k,\varepsilon}+\CE_{r_{\ell_k,\varepsilon},\ell_k} \leq  2q_{\ell_k,\varepsilon} \leq \left(\eta_{\ell_k} \frac{c(\vartheta,\lambda,\sigma)\sqrt{\varepsilon}}{2} \right)^{\frac{1}{\nu_{\ell_k}}}.
    \end{equation*}
    The integral in the second summand of the right-hand side can be upper bounded by $\sqrt{2\CV_{k}(\rho_0)}$ thanks to Jensen's inequality.
    Using that $\alpha > \alphazerol$ with $\alphazerol$ as defined in \eqref{def:alpha0l}, it further holds that
    \begin{equation*}
        \exp\left(-\alpha q_{\ell_k,\varepsilon}\right)
        \le \frac{c(\vartheta,\lambda,\sigma)\sqrt{\varepsilon}}{2^{d_{\ell_k}+1}\sqrt{2\CV_{\infty}(\rho_0)}} \rho_0^{\ell_k} \left(\B^{d_{\ell_k}}_{r_{\ell_k,\varepsilon}/2}(\CB_{\ell_k}^\top \globmin)\right)
        \le \frac{c(\vartheta,\lambda,\sigma)\sqrt{\varepsilon}}{2\sqrt{2\CV_{\infty}(\rho_0)}} \rho_0^{\ell_k} \left(\B^{d_{\ell_k}}_{r_{\ell_k,\varepsilon}}(\CB_{\ell_k}^\top \globmin)\right),
    \end{equation*}
    where the last inequality is immediate.
    Using these observations and recalling that $\CV_k(\dummy) \le \CV_{\infty}(\dummy)$ by definition,
    we can conclude that
    \begin{equation*}
    \begin{split}
        \SN{\left( \conspoint{\rho_0} - \globmin \right)_k}
        &\le \frac{c(\vartheta,\lambda,\sigma)\sqrt{\varepsilon}}{2} + \frac{c(\vartheta,\lambda,\sigma)\sqrt{\varepsilon}}{2\sqrt{2\CV_{\infty}(\rho_0)}} \sqrt{2\CV_{k}(\rho_0)} \\
        & \leq c(\vartheta,\lambda,\sigma)\sqrt{\varepsilon} <  c(\vartheta,\lambda,\sigma) \sqrt{\CV_\infty(\rho_0)} = C_\infty(0)
    \end{split}
    \end{equation*}
    as desired.
    
    Next, we demonstrate that the functional $\CV_{\infty}$ shows at least exponential decay over the time interval $[0,T_\alpha]$. By employing Lemma \ref{lem:evolutionVinf_upper}
    and the definition of $T_\alpha$ in \eqref{def:Talpha},
    we get
    \begin{equation}
    \begin{split}
        \frac{d}{dt}\CV_\infty(\rho_t) &\leq -\left(2\lambda-\sigma^2\right) \CV_\infty(\rho_t) 
        + \sqrt{2}\left(\lambda+\sigma^2\right) \sqrt{\CV_\infty(\rho_t)} \max_{k \in \{1,\dots,D\}} \SN{\left(\conspoint{\rho_t}-\globmin\right)_k} \\
        & \phantom{-\left(2\lambda-\sigma^2\right) \CV_\infty(\rho_t)}\quad\,\,
        +\frac{\sigma^2}{2} \max_{k \in \{1,\dots,D\}} \left(\conspoint{\rho_t}-\globmin\right)_k^2\\
        &\leq -\left(2\lambda-\sigma^2\right) \CV_\infty(\rho_t) 
        + \sqrt{2}\left(\lambda+\sigma^2\right) c(\vartheta,\lambda,\sigma) \CV_\infty(\rho_t) + \frac{\sigma^2}{2} c^2(\vartheta,\lambda,\sigma) \CV_\infty(\rho_t)\\
        &\le -(1-\vartheta)\left(2\lambda-\sigma^2\right) \CV_\infty(\rho_t),
    \end{split}
    \end{equation}
    for almost every $t \in (0,T_\alpha)$. 
    Applying Gr\"onwall’s inequality and using the continuity of $t\mapsto \CV_{\infty}(\rho_t)$, we obtain
    
    \begin{equation}
    \label{eq:upperbound_functional}
        \CV_{\infty}(\rho_t)
        \leq \CV_{\infty}(\rho_0)\exp\left(-(1-\vartheta)\left(2\lambda-\sigma^2\right)t\right) \quad  \text{ for all } t\in[0,T_{\alpha}].
    \end{equation}
    At the same time, by utilizing Lemma~\ref{lem:evolutionVinf_lower} and following analogous computations,
    we can derive that the functional $\CV_\infty$ exhibits at most exponential decay, i.e.,
    \begin{equation}
        \label{eq:lowerbound_functional}
        \CV_{\infty}(\rho_t)\geq \CV_{\infty}(\rho_0)\exp\left(-\left(1+\vartheta/2\right)\left(2\lambda-\sigma^2\right)t\right) \quad \text{ for all } t\in[0,T_{\alpha}].
    \end{equation}
    To conclude, it remains to show that $\CV_{\infty}(\rho_{T_\alpha}) = \varepsilon$ for $T_\alpha \in \big[\frac{1 - \vartheta}{1 + \vartheta/2} T^*, T^*\big]$.
    We will do this by considering three distinct cases.
    \begin{itemize}
            \item \textbf{Case $\bm{T_{\alpha}\geq T^*}$}:
            Evaluating the inequality \eqref{eq:upperbound_functional} at time $T^*$ and exploiting the definition of $T^*$ in \eqref{def:Tstar}, we obtain that $\CV_{\infty}(\rho_{T^*}) \le \varepsilon$. 
            Consequently, by the definition of $T_{\alpha}$ in \eqref{def:Talpha} and the continuity of $t\mapsto \CV_{\infty}(\rho_t)$, we have that $\CV_{\infty}(\rho_{T_\alpha})=\varepsilon$ with $T_{\alpha}=T^*$.
            \item \textbf{Case $\bm{T_{\alpha}< T^*}$ and $\bm{\CV_{\infty}(\rho_{T_{\alpha}})\leq\varepsilon}$}:
            From the definition of $T_{\alpha}$ in \eqref{def:Talpha} and the continuity of the functional $t\mapsto \CV_{\infty}(\rho_t)$, it is possible to conclude that $\CV_{\infty}(\rho_{T_{\alpha}})=\varepsilon$. From the lower bound of the functional \eqref{eq:lowerbound_functional}, we hence obtain
            $\varepsilon=\CV_{\infty}(\rho_{T_{\alpha}})\geq \CV_{\infty}(\rho_0) \exp{\left(-(1+\vartheta/2)\left(2\lambda-\sigma^2\right)T_{\alpha}\right)}$,
            which is equivalent to
            \begin{equation}
                \frac{1}{(1+\vartheta/2)\left(2\lambda-\sigma^2\right)}\log\left(\frac{\CV_{\infty}(\rho_0)}{\varepsilon} \right)\leq T_{\alpha}.
            \end{equation}
            By definition of $T^*$ in \eqref{def:Tstar}, it holds $\frac{1-\vartheta}{1+\vartheta/2}T^*= \frac{1}{(1+\vartheta/2)\left(2\lambda-\sigma^2\right)}\log\left({\CV_{\infty}(\rho_0)/\varepsilon} \right)$, so we can conclude
            \begin{equation}
                \frac{1-\vartheta}{1+\vartheta/2}T^*\leq T_{\alpha}<T^*.
            \end{equation}
            \item \textbf{Case $\bm{T_{\alpha}< T^*}$ and $\bm{\CV_{\infty}(\rho_{T_{\alpha}})>\varepsilon}$}:
            We show that this case can never occur by proving that $\SN{(\conspoint{\rho_{T_{\alpha}}}-\globmin)_{k}}< C_{\infty}(T_{\alpha})$ for all $k \in \{1,\dots,D\}$ due to the previously made choice of $\alpha$.
            This would contradict the definition of $T_\alpha$ in \eqref{def:Talpha}.
            To show this, fix $k \in \{1,\dots,D\}$. 
            Then, there exist $\ell_k \in \{1,\dots,L\}$ such that $k \in \CI_{\ell_k}$ and $k'_k  \in \{1,\dots, d_{\ell_k}\}$ such that \eqref{eq:simplification_diffk} holds for $\indivmeasure = \rho_{T_\alpha}$.
            The key steps are similar to those taken earlier in the proof at time $t=0$. 
            Specifically, an application of Proposition \ref{prop:QLPl_Vk} yields
            \begin{equation}
            \begin{split}
                \SN{(\conspoint{\rho_{T_{\alpha}}}-\globmin)_{k}}
                & \le \frac{(q_{\ell_k,\varepsilon} +\CE_{r_{\ell_k,\varepsilon},{\ell_k}})^{\nu_{\ell_k}}}{\eta_{\ell_k}} + \frac{\exp\left(-\alpha q_{\ell_k,\varepsilon}\right)}{\rho_{T_{\alpha}}^{\ell_k}\left(\B^{d_{\ell_k}}_{r_{\ell_k,\varepsilon}}(\CB_{\ell_k}^\top\globmin)\right)} \int_{\R^{d_{\ell_k}}} \left| v_{k'_k}- \left( \CB_{\ell_k}^\top \globmin \right)_{k'_k} \right| d\rho_{T_{\alpha}}^{\ell_k}(v)\\
                & = \frac{(q_{\ell_k,\varepsilon} +\CE_{r_{\ell_k,\varepsilon},{\ell_k}})^{\nu_{\ell_k}}}{\eta_{\ell_k}} + \frac{\exp\left(-\alpha q_{\ell_k,\varepsilon}\right)}{\rho_{T_{\alpha}}^{\ell_k}\left(\B^{d_{\ell_k}}_{r_{\ell_k,\varepsilon}}(\CB_{\ell_k}^\top\globmin)\right)} \int_{\RD} \left| x_k- \globmin_{k} \right| d\rho_{T_{\alpha}}(x)\\
                &\leq \frac{(q_{\ell_k,\varepsilon} +\CE_{r_{\ell_k,\varepsilon},{\ell_k}})^{\nu_{\ell_k}}}{\eta_{\ell_k}} + \frac{\exp\left(-\alpha q_{\ell_k,\varepsilon}\right)}{\rho_{T_{\alpha}}^{\ell_k}\left(\B^{d_{\ell_k}}_{r_{\ell_k,\varepsilon}}(\CB_{\ell_k}^\top\globmin)\right)} \sqrt{\CV_{k}(\rho_{T_\alpha})} \\
                &\leq \frac{(q_{\ell_k,\varepsilon} +\CE_{r_{\ell_k,\varepsilon},{\ell_k}})^{\nu_{\ell_k}}}{\eta_{\ell_k}} + \frac{\exp\left(-\alpha q_{\ell_k,\varepsilon}\right)}{\rho_{T_{\alpha}}^{\ell_k}\left(\B^{d_{\ell_k}}_{r_{\ell_k,\varepsilon}}(\CB_{\ell_k}^\top\globmin)\right)} \sqrt{\CV_{\infty}(\rho_{T_\alpha})},
            \end{split}
            \end{equation}
            where the inequality in the next-to-last line is due to Jensen's inequality and
            the last step due to $\CV_{k}(\rho_{T_\alpha})\leq \CV_{\infty}(\rho_{T_\alpha})$.
            Thanks to Proposition \ref{prop:lower_bound} it further holds that
             \begin{equation}
            \begin{split}
                \rho_{T_{\alpha}}^{\ell_k}\left(\B^{d_{\ell_k}}_{r_{\ell_k,\varepsilon}}(\CB_{\ell_k}^\top\globmin)\right)
                &\geq \left( \int_{\R^{d_{\ell_k}}} \phi_{r_{\ell_k,\varepsilon}}(v) \,d\rho_0^{\ell_k}(v)\right)\exp(-p_{\ell_k,\varepsilon} T_{\alpha})\\
                 &\geq \frac{1}{2^{d_{\ell_k}}} \rho_0^{\ell_k}\left(\B^{d_{\ell_k}}_{r_{\ell_k,\varepsilon}/2}(\CB_{\ell_k}^\top\globmin)\right)\exp{\left(-p_{\ell_k,\varepsilon} T^*\right)}>0,
            \end{split}
            \end{equation}
            allowing us to continue the former as
            \begin{equation}
            \begin{split}
                \SN{(\conspoint{\rho_{T_{\alpha}}}-\globmin)_{k}}
                &\leq \frac{(q_{\ell_k,\varepsilon} +\CE_{r_{\ell_k,\varepsilon},{\ell_k}})^{\nu_{\ell_k}}}{\eta_{\ell_k}} + \frac{2^{d_{\ell_k}}\exp\left(-\alpha q_{\ell_k,\varepsilon}\right)\exp{\left(-p_{\ell_k,\varepsilon} T^*\right)}}{\rho_{0}^{\ell_k}\left(\B^{d_{\ell_k}}_{r_{\ell_k,\varepsilon}/2}(\CB_{\ell_k}^\top\globmin)\right)} \sqrt{\CV_{\infty}(\rho_{T_\alpha})}.
            \end{split}
            \end{equation}
            The first summand of the right-hand side can be upper bounded as in the computations at time $t=0$ by $c(\vartheta,\lambda,\sigma)\sqrt{\varepsilon}/2$ by using the definitions of $q_{\ell_k,\varepsilon}$ and $r_{\ell_k,\varepsilon}$.
            Using that $\alpha > \alphazerol$ with $\alphazerol$ as defined in \eqref{def:alpha0l}, it further holds that
            \begin{equation*}
                \exp\left(-\alpha q_{\ell_k,\varepsilon}\right)
                \le \frac{c(\vartheta,\lambda,\sigma)\sqrt{\varepsilon}}{2^{d_{\ell_k}+1}\sqrt{2\CV_{\infty}(\rho_0)}} \exp{\left(p_{\ell_k,\varepsilon} T^*\right)} \rho_0^{\ell_k} \left(\B^{d_{\ell_k}}_{r_{\ell_k,\varepsilon}/2}(\CB_{\ell_k}^\top \globmin)\right).
            \end{equation*}
            Using these observations,
            we can conclude that
            \begin{equation}
            \begin{split}
                \SN{(\conspoint{\rho_{T_{\alpha}}}-\globmin)_{k}}
                &\leq \frac{c(\vartheta,\lambda,\sigma)\sqrt{\varepsilon}}{2} +  \frac{c(\vartheta,\lambda,\sigma)\sqrt{\varepsilon}}{2\sqrt{2\CV_{\infty}(\rho_0)}} \sqrt{2\CV_{\infty}(\rho_{T_\alpha})}\\
                &< \frac{c(\vartheta,\lambda,\sigma)\sqrt{\varepsilon}}{2} +  \frac{c(\vartheta,\lambda,\sigma)}{2} \sqrt{\CV_{\infty}(\rho_{T_\alpha})}\\
                & < c(\vartheta,\lambda,\sigma)\sqrt{\CV_\infty(\rho_{T_{\alpha}})}
                = C_\infty(T_{\alpha}),
            \end{split}
            \end{equation}
            where we used in the second line that 
            $\CV_{\infty}(\rho_0)>\varepsilon$
            and in the last step that $\CV_{\infty}(\rho_{T_{\alpha}})>\varepsilon$ by assumption of this case, hence the desired contradiction.
    \end{itemize} 
    This concludes the proof.
\end{proof}

\begin{remark}
\label{rem:proof_Vk_dlg1_specialcases}
As already mentioned in Remark \ref{rem:notablechoicesL}, the expression \eqref{eq:add_sep} of $\CE$ simplifies for the special choices $L=1$ ($\intrd = D$) and $\intrd = 1$ ($L=D$).
In this remark, we comment on the convergence of the anisotropic CBO method \eqref{eq:aCBO_McKean} in these two special scenarios for the functional $\CV_{\infty}$.

(i) If $L=1$, $\CE$ is non-separable, which corresponds to the case of general objectives already treated in~\cite[Theorem~3.6]{riedl2024perspective}. There, a mean‑field convergence result was established for the functional $\CV(\indivmeasure) = \frac{1}{2} \int_{\RD} \N{x-\globmin}_{2,D}^2 d\indivmeasure(x)$ for $\indivmeasure \in \CP_2(\RD)$, with $\alpha$ being chosen greater than some \(\alpha_{0}\) scaling as \(\CO(D)\). We recover this result by using the functional \(\CV_{\infty}\) under the identical Assumption~\ref{ass:invcont_QLP_glob}.

(ii) If $L=D$ and $\intrd = 1$, $\CE$ is fully separable and
Theorem \ref{thm:conv_dlg1_Vk} simplifies.
In this case, condition \eqref{eq:simplification_diffk} appearing in the new quantitative Laplace principle (Proposition \ref{prop:QLPl_Vk}) reads $\left| \left( \conspoint{\rho_t} - \globmin \right)_k \right| = \left|  x_{\alpha}^{k}(\rho_t^{k})-  \globmin_{k}  \right|$, $k\in\{1,\ldots,D\}$
and the right-hand side of \eqref{eq:QLPl_Vk} reduces to an integral over $\R$
(cf.\@ Remark \ref{rem:whynotstdQLP}),
allowing us to directly apply~\cite[Proposition~1]{fornasier2021convergence} in the $1$-dimensional case to 
eventually derive that $\alpha$ needs to be chosen larger than $\widetilde\alpha_{0}^{L} \coloneqq \max_{k \in \{1,\dots,D\}} \widetilde\alpha_{0,k}^{L}$ with
\begin{equation}
     \begin{split}
     \widetilde\alpha_{0,k}^{L}\coloneqq \frac{1}{q_{k,\varepsilon}} \left( \log\left( \frac{2^{2} \sqrt{2\CV_{\infty}(\rho_0)}}{c(\vartheta,\lambda,\sigma)\sqrt{\varepsilon}} \right) + \frac{p_{k,\varepsilon}}{(1-\vartheta)\left(2\lambda-\sigma^2\right)} \log\left( \frac{\CV_{\infty}(\rho_0)}{\varepsilon} \right) - \log \left( \rho_0^{k} ( \B^{1}_{r_{k,\varepsilon}/2}(\globmin_k) ) \right) \right),
     \end{split}
 \end{equation}
where $q_{k,\varepsilon}, r_{k,\varepsilon}$ and $p_{k,\varepsilon}$ are defined as in the proof of Theorem \ref{thm:conv_dlg1_Vk}.
\end{remark}

%%%%%%%%%%%%%%%%%%%%%%%%%%%%%%%%%%%%%%%%%%%%%%%%%%
%%%%%%%%%% Section %%%%%%%%%%%%%%%%%%%%%%%%%%%%%%%
%%%%%%%%%%%%%%%%%%%%%%%%%%%%%%%%%%%%%%%%%%%%%%%%%%
\section{Numerical experiments} 
\label{sec:numerics}
To analyze how different levels of complexity (in particular, different levels of separability and complexity within the separable components themselves) of the objective function
affect the performance of anisotropic CBO,
we compare the success probabilities of the anisotropic CBO method when minimizing two types of objective functions:
a standard Rastrigin-type function of the form \eqref{eq:rastrigin}, which is by design fully separable, i.e., $\intrd = 1$,
and rotated versions of it, which, by design, are not fully separable and can have varying intrinsic dimensionalities~$\intrd$ that we can specify explicitly, see Section~\ref{sec:numerics_intrd}.
This allows us to investigate the influence of the level of separability of the objective function.
Beyond that, we compare the performance of the anisotropic CBO method for different magnitudes of oscillation and hence non-convexity within the components of the rotated Rastrigin function, see Section~\ref{sec:numerics_cD}.
This permits us to study the influence of the complexity of the objective function within the separable components.

By doing so, we numerically verify our theoretical findings which are concisely summarized in Remark~\ref{rem:crucial_remark}.
Firstly, we demonstrate that the performance of anisotropic CBO~\eqref{eq:aCBO_micro_EM} deteriorates as the function becomes less separable, verifying the scaling of $\alpha_0$ on the intrinsic dimension $\intrd$. 
On a side note, we show that isotropic CBO, which is as \eqref{eq:aCBO_micro_EM} once the diffusion term is replaced by $\sigma \Nbig{\Xihat{h\dt} - \conspoint{\empmeasure{h\dt}}}_{2,D} B^{(i)}_{h\dt}$,
is unaffected by such additive structure.
Secondly, we show that the complexity depends on the properties of the individual component functions $\CE_\ell$.

\subsection{Experimental setting} 
\label{sec:numerics_setting}

The standard Rastrigin function $\CE : \mathbb{R}^D \to \mathbb{R}$ is defined as
\begin{equation}
\label{eq:rastrigin}
    \CE(x) = \sum_{k=1}^{D} \left( x_k^2 - (c_D)_k \cos(2\pi x_k) + (c_D)_k \right),
\end{equation}
which is a non-convex, axis-aligned function with a unique global minimizer at $\globmin = (0,\ldots,0) \in \mathbb{R}^D$. The constants $(c_D)_k$ quantifies the magnitude of the oscillatory component of the function, describing the level of non-convexity and hence complexity.

By varying the constant $(c_D)_k$, we can increase or decrease the non-convexity in a certain coordinate direction by making $(c_D)_k$ larger or smaller, respectively.
This is the focus of Section~\ref{sec:numerics_cD}.

In Section~\ref{sec:numerics_intrd}, we focus on the level of separability.
It is straightforward to see that $\CE$  in \eqref{eq:rastrigin} is fully separable, i.e., $\intrd=1$.
However, as we begin to gradually rotate the function across more and more dimensions, we weaken its additive‑separability property.
More precisely, we consider multiple rotated Rastrigin functions, defined according to
\begin{equation}
\label{eq:rotated_rastrigin}
\CE_R(x) = \CE(Rx),
\end{equation}
where $R \in O(D)\subset\mathbb{R}^{D \times D}$ is an orthogonal matrix.

The effects of both such rotation on the Rastrigin function and the altering of the non-convexity parameters~$(c_D)_k$ is illustrated in Figure~\ref{fig:standrotRastrigin} in two dimensions.
\begin{figure}[ht!]
    \centering
    \hphantom{--}
    % Subfigure 1
    \begin{subfigure}[b]{0.3\textwidth}
        \centering
        % Include your first heatmap image here
        \includegraphics[width=1\textwidth]{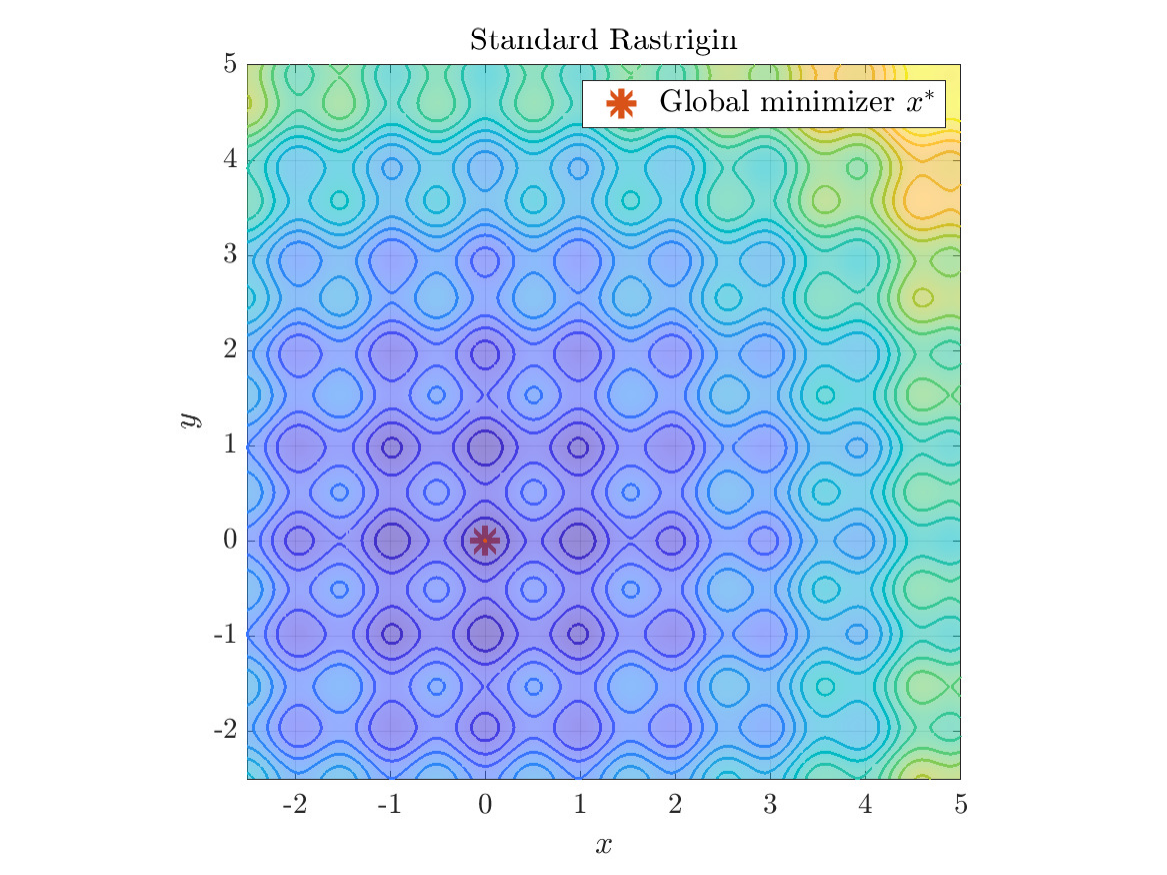}
        \caption{Standard Rastrigin function~\eqref{eq:rastrigin} with $(c_D)_1=(c_D)_2={10}$}
        \label{fig:standardRastrigin}
    \end{subfigure}
    \hfill
    % Subfigure 2
       \begin{subfigure}[b]{0.3\textwidth}
        \centering
        % Include your third heatmap image here
        \includegraphics[width=1\textwidth]{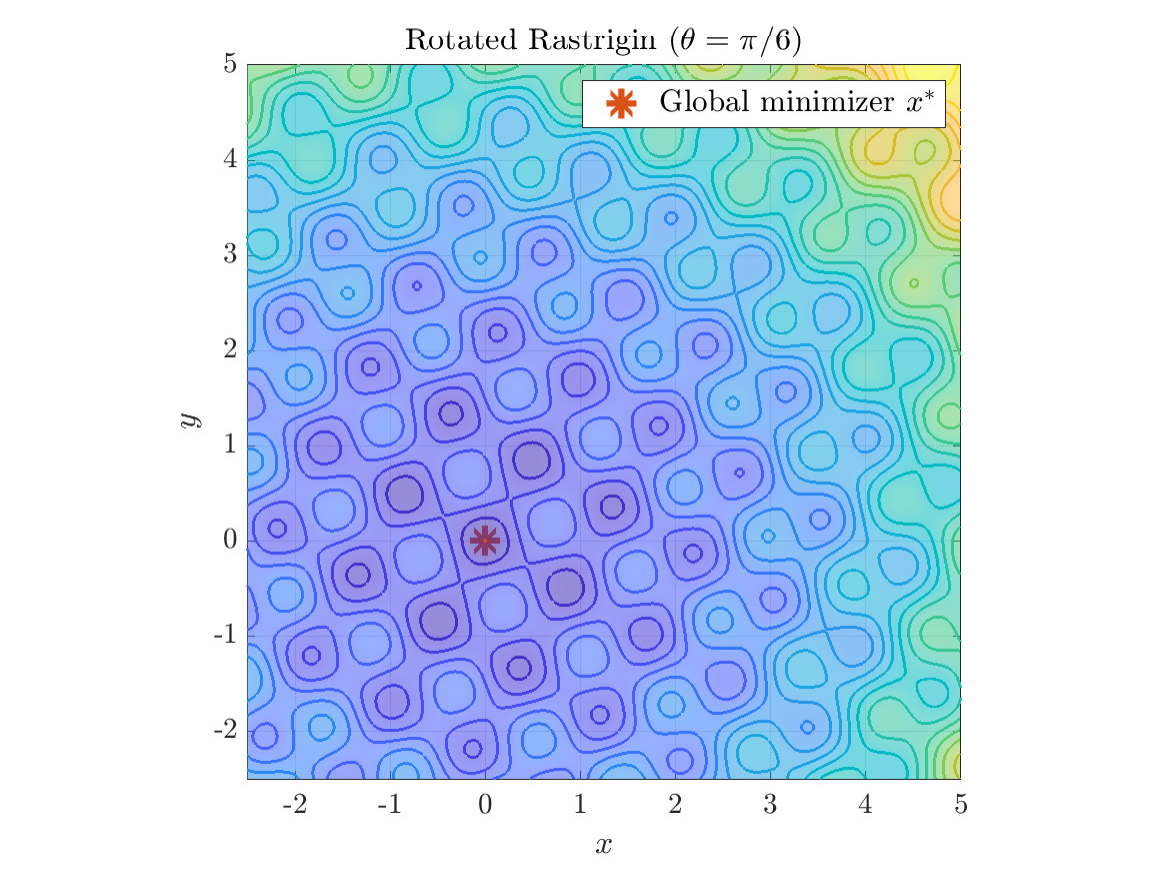}
        \caption{Rotated Rastrigin function~\eqref{eq:rotated_rastrigin} with $(c_D)_1={10}$, $(c_D)_2={10}$}
        \label{fig:rotatedRastrigin}
    \end{subfigure}
    \hfill
    % Subfigure 3
       \begin{subfigure}[b]{0.3\textwidth}
        \centering
        % Include your third heatmap image here
        \includegraphics[width=1\textwidth]{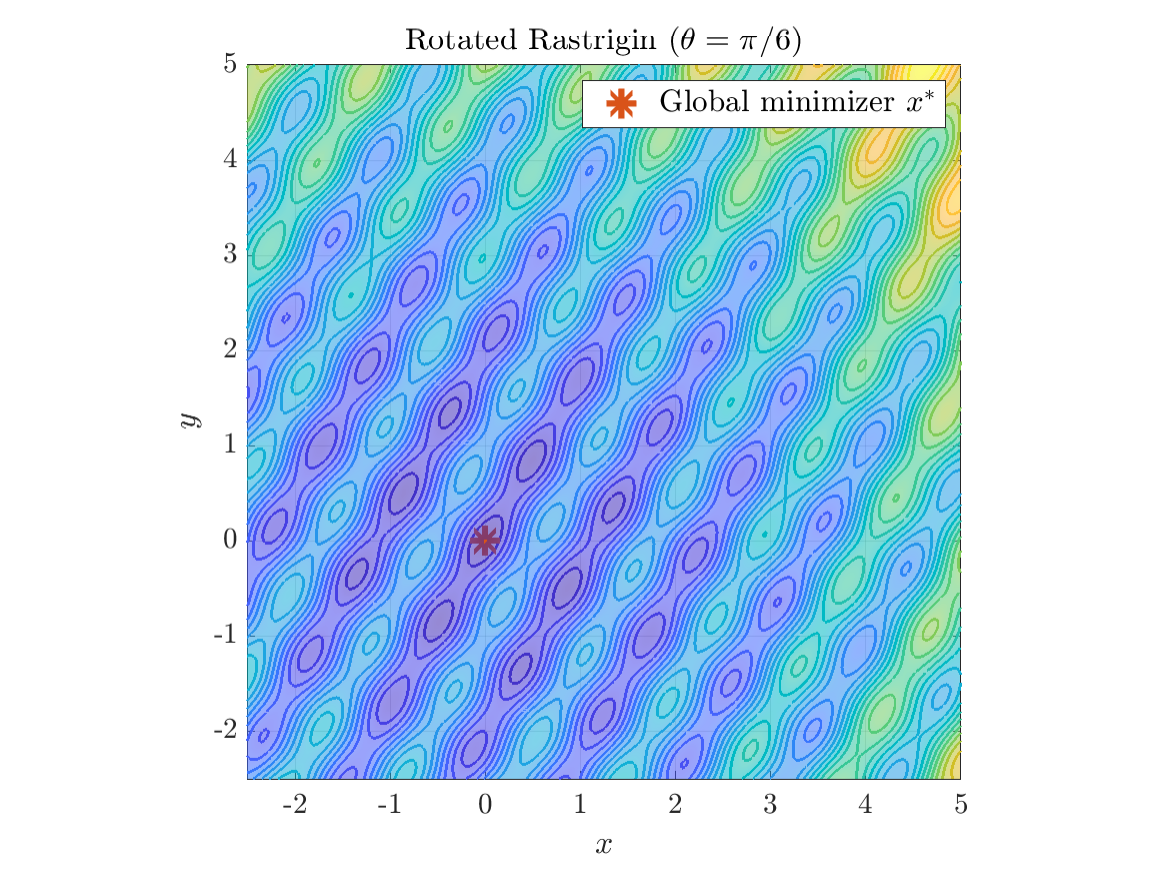}
        \caption{Rotated Rastrigin function~\eqref{eq:rotated_rastrigin} with $(c_D)_1={10}$, $(c_D)_2={2.5}$}
        \label{fig:rotatedRastrigin_cD}
    \end{subfigure}
    \hphantom{--}
    \caption{Comparison between the standard Rastrigin function (fully separable) and a rotated version of the function by $\pi/6$ (non-separable) with different levels of non-convexity. The rotation mixes the variables, thereby preventing additive separability.}
    \label{fig:standrotRastrigin}
\end{figure}

In order to generate objective functions with different levels of additive separability,
we design, as depicted in Figure~\ref{fig:blocks}, block-diagonal rotation matrices $R$ of the form
\begin{equation*}
    R = \mathrm{diag}(Q_1, \dots, Q_L) \in O(D)
\end{equation*}
with $Q_\ell\in O(d_\ell)$ for all $\ell \in \{1,\dots,L\}$.
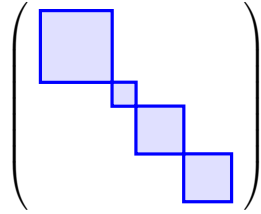
\begin{figure}
\begin{center}
\usetikzlibrary{matrix,fit,positioning,calc}

\begin{tikzpicture}[baseline=(m.center),scale=01]
    \matrix (m) [matrix of math nodes,
             left delimiter=(, right delimiter=),
             nodes={ text width=3mm, text height=1.5mm, text depth=1.5mm, text=white, align=center, inner sep=0pt},
             column sep=0.15mm, row sep=0.15mm
              ]{      
        a_{11} & a_{12} & a_{13} &      &      &      &      &      \\
        a_{21} & a_{22} & a_{23} &      &      &      &      &      \\
        a_{31} & a_{32} & a_{33} &      &      &      &      &      \\
        &       &       & b_{44}&      &      &      &      \\
        &       &       &      & c_{55} & c_{56}&      &      \\
        &       &       &      & c_{65} & c_{66}&      &      \\
        &       &       &      &      &      & d_{77} & d_{78} \\
        &       &       &      &      &      & d_{87} & d_{88} \\
  };
  % draw colored boxes using fit (use the node names m-row-col)
  \node[draw=blue, very thick, fill=blue!12, inner sep=0.04pt, fit=(m-1-1) (m-3-3)] (B1) {};
  \node[draw=blue, very thick, fill=blue!12, inner sep=0.04pt, fit=(m-4-4) (m-4-4)] (B2) {};
  \node[draw=blue, very thick, fill=blue!12, inner sep=0.04pt, fit=(m-5-5) (m-6-6)] (B3) {};
  \node[draw=blue, very thick, fill=blue!12, inner sep=0.04pt, fit=(m-7-7) (m-8-8)] (B4) {};
\end{tikzpicture}
\caption{Conceptual illustration of the block-diagonal structure of the rotation matrix $R$. Each colored block represents a full rotation matrix $Q_\ell\in O(d_\ell)$.}
\label{fig:blocks}
\end{center}
\end{figure}
Each block corresponds to a subset of $d_\ell$ variables that are linearly mixed together through the orthogonal transformation~$Q_{\ell}$.
Then, \eqref{eq:rotated_rastrigin} yields a rotated Rastrigin function $\CE_R$
which has precisely $L$ separable components with individual dimensions $d_\ell$.

From a technical stance, in order to obtain the orthogonal matrices~$Q_\ell$, we generate a random matrix
$A_\ell \in \mathbb{R}^{d_\ell \times d_\ell}$ with $(A_\ell)_{ij} \sim \mathcal{N}_1(0,1)$,
and compute its QR decomposition $A_\ell = Q_\ell R_\ell$, using the orthogonal factor $Q_\ell$ to define our rotation.
By varying the block sizes $d_\ell$ and the number of blocks $L$,
we generate a family of rotated functions with different degrees of separability: smaller blocks (larger $L$) correspond to rotations that preserve most of the original separability, while larger blocks (smaller $L$) progressively introduce stronger coupling among variables. The limiting cases are when $d_\ell = 1$ for all $\ell$, $R$ is a signed permutation matrix with entries $\pm 1$, and the problem remains fully separable, and when $L = 1$, $R$ is a dense orthogonal matrix of size $D \times D$, coupling all variables.

In order to vary the complexity of the objective function within the separable components, we investigate the influence of the magnitude of the oscillation~$(c_D)_k$.

\subsection{Influence of the intrinsic dimensionality for dimension \texorpdfstring{$\intrd$}{} and the level of separability}
\label{sec:numerics_intrd}

In this section, we investigate the influence of the intrinsic dimensionality $\intrd$ and the level of separability of the objective function on the performance of the anisotropic CBO algorithm.
To do so, we compare success rates for the standard (fully separable) Rastrigin function and its rotated versions, which, by design, have varying intrinsic dimensionalities~$\intrd$ and component structures.
We showcase experiments in $D=4$ and $D=8$.

\subsubsection{Experiments for dimension \texorpdfstring{$D=4$}{4}}
\label{sec:numerics_intrd_D4}

The experiments in dimension $D = 4$ are performed using the parameters
\begin{equation*}
    \lambda = 1, \sigma = \sqrt{1.6}, \quad T =5, \Delta t = 10^{-3}, \quad \rho_0 \sim \CN_D((3,\ldots,3)^T,20I_D).
\end{equation*}  
In Figure~\ref{fig:heatmaps_cbo_ANISO4D}, we report the success rates of the CBO method with anisotropic noise for the standard Rastrigin function and for several rotated variants with different intrinsic dimensions.
Throughout this section, $(c_D)_k=10$ for all $k \in \{1,\dots,D\}$.

A run is considered successful if
\begin{equation*}
\label{eq: success_rate}
|\CE(\conspoint{\rho_t}) - \CE(\globmin)| < 0.25,
\end{equation*}
which ensures that the consensus point lies within the basin of attraction of the global minimizer, i.e., in the correct valley of the objective landscape.

We begin in Figure~\ref{fig:heatmap_rastr4D_anisod1} with the results for the standard Rastrigin function~\eqref{eq:rastrigin}, whose level sets are aligned with the Cartesian axes.
The corresponding heatmap shows high success rates over a wide range of values of the hyperparameter $\alpha$ and the number of particles $N$.
We then display the results for rotated versions of the objective with different rotation matrices $R$~\eqref{eq:rotated_rastrigin}, constructed such that the intrinsic dimension increases $\intrd = 2,3,4$, corresponding to a progressively stronger coupling between the variables.
As the intrinsic dimension increases, the level sets become increasingly misaligned with the coordinate axes. This results in a degradation of the success rates. Already for $\intrd = 2$ a noticeable drop is observed (Figure~\ref{fig:heatmap_rastr4D_anisod2}), which becomes more pronounced for $\intrd = 3$ (Figure~\ref{fig:heatmap_rastr4D_anisod3}), and is most significant in the fully non-separable case $\intrd = 4$ (Figure~\ref{fig:heatmap_rastr4D_anisod4}).

\begin{figure}[ht!]
    \centering
    % Subfigure 1
    \begin{subfigure}[b]{0.24\textwidth}
        \centering
        \includegraphics[width=\textwidth]{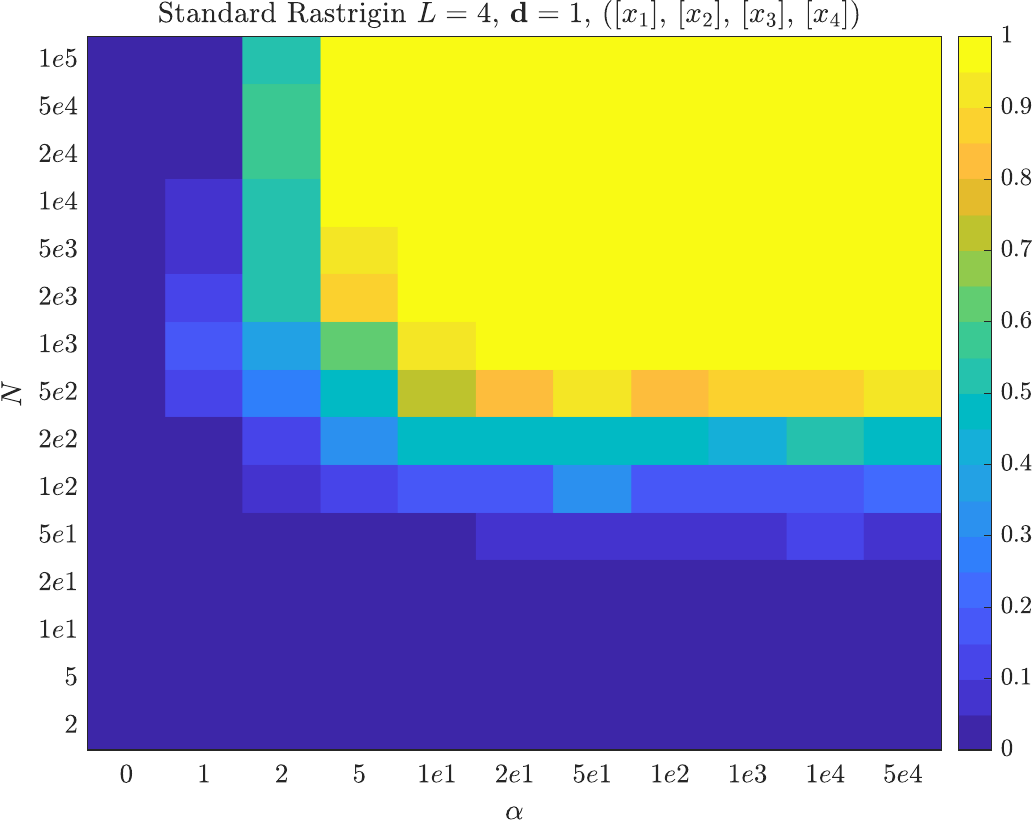}
        \caption{$\mathbf{d}=1$}
        \label{fig:heatmap_rastr4D_anisod1}
    \end{subfigure}
    \hfill
    % Subfigure 2
       \begin{subfigure}[b]{0.24\textwidth}
        \centering
        \includegraphics[width=\textwidth]{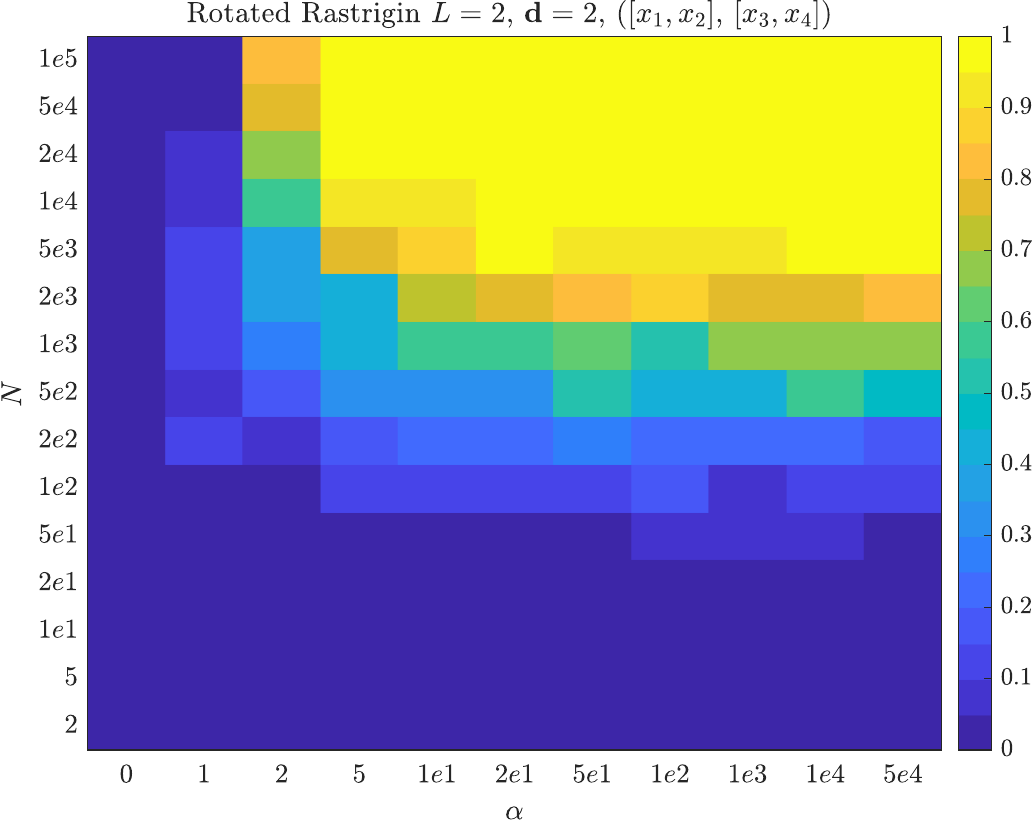}
        \caption{$\mathbf{d}=2$}
        \label{fig:heatmap_rastr4D_anisod2}
    \end{subfigure}
    \hfill
    % Subfigure 3
    \begin{subfigure}[b]{0.24\textwidth}
        \centering
        \includegraphics[width=\textwidth]{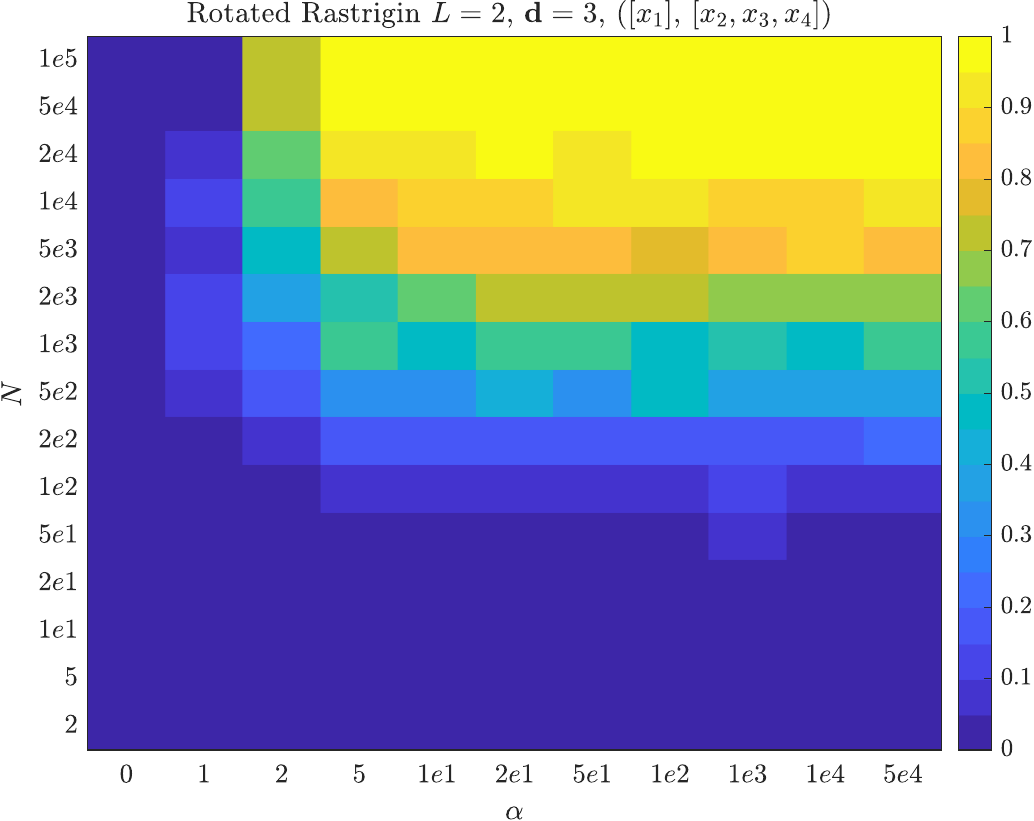}
        \caption{$\mathbf{d}=3$.}
        \label{fig:heatmap_rastr4D_anisod3}
    \end{subfigure}
    \hfill
    % Subfigure 4
    \begin{subfigure}[b]{0.24\textwidth}
        \centering
        \includegraphics[width=\textwidth]{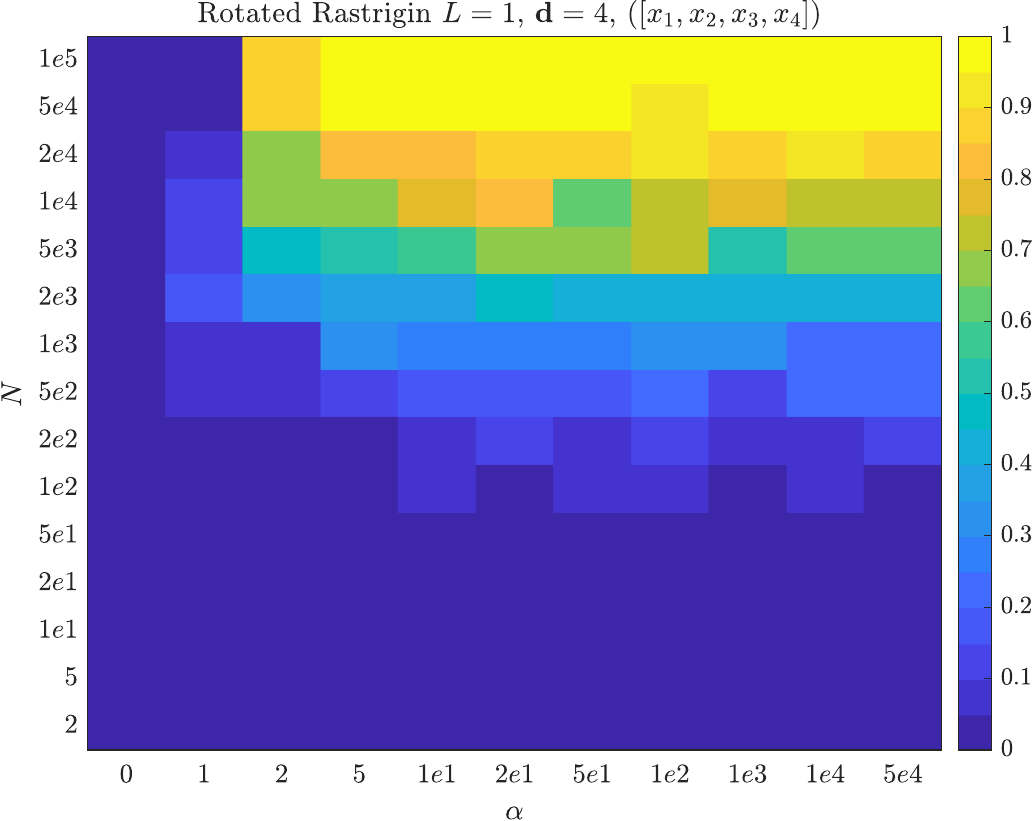}
        \caption{$\mathbf{d}= D= 4$.}
        \label{fig:heatmap_rastr4D_anisod4}
    \end{subfigure}

    \caption{
    Success rate of the CBO method with anisotropic noise applied to different versions of the Rastrigin function in dimension $D=4$. Each heatmap shows the success rate as a function of $\alpha$ and $N$. Results are averaged over 100 runs. The progressive deterioration of the success rate can be observed when comparing the images from Figure \ref{fig:heatmap_rastr4D_anisod1} to Figure \ref{fig:heatmap_rastr4D_anisod4}.}
    \label{fig:heatmaps_cbo_ANISO4D}
\end{figure}

\paragraph{Comparison for isotropic CBO.}
Next, let us, as a side remark, examine how the isotropic CBO algorithm \cite{pinnau2017consensus,carrillo2018analytical,fornasier2024consensus} is affected by the level of separability of the objective function.
Unlike the anisotropic case, in which the noise has a directional structure, in the isotropic case the noise is uniform in all spatial directions.
One therefore does not expect that isotropic CBO can benefit from any additively separable structure.
Our goal is not to conduct a theoretical study of the isotropic dynamics, but rather to evaluate at the numerical level the method’s dependence on the separability of the objective function.

The experiments are performed for parameters
\begin{equation*}
    \lambda = 1, \sigma = \frac{\sqrt{1.6}}{2}, \quad T =5, \Delta t = 10^{-3}, \quad \rho_0 \sim \CN_D((3,\ldots,3)^T,20I_D),
\end{equation*} 
and the remaining setting is as before.

Figure~\ref{fig:heatmaps_cbo_ISO4D} shows that, in the isotropic case, the heatmaps of the standard Rastrigin function, with constant $c_D=10$, and its rotated version are essentially indistinguishable, confirming our expectation.
This suggests that the degree of separability of the objective function does not affect the performance of the isotropic CBO.
The reason is that the noise has a global and uniform structure, so it acts the same way in all directions: the alignment of the function with respect to the axes does not change its effect.

\begin{figure}[ht!]
    \centering
    % Empty slot 1
    \begin{subfigure}[b]{0.24\textwidth}
        \centering
        \phantom{\includegraphics[width=\textwidth]{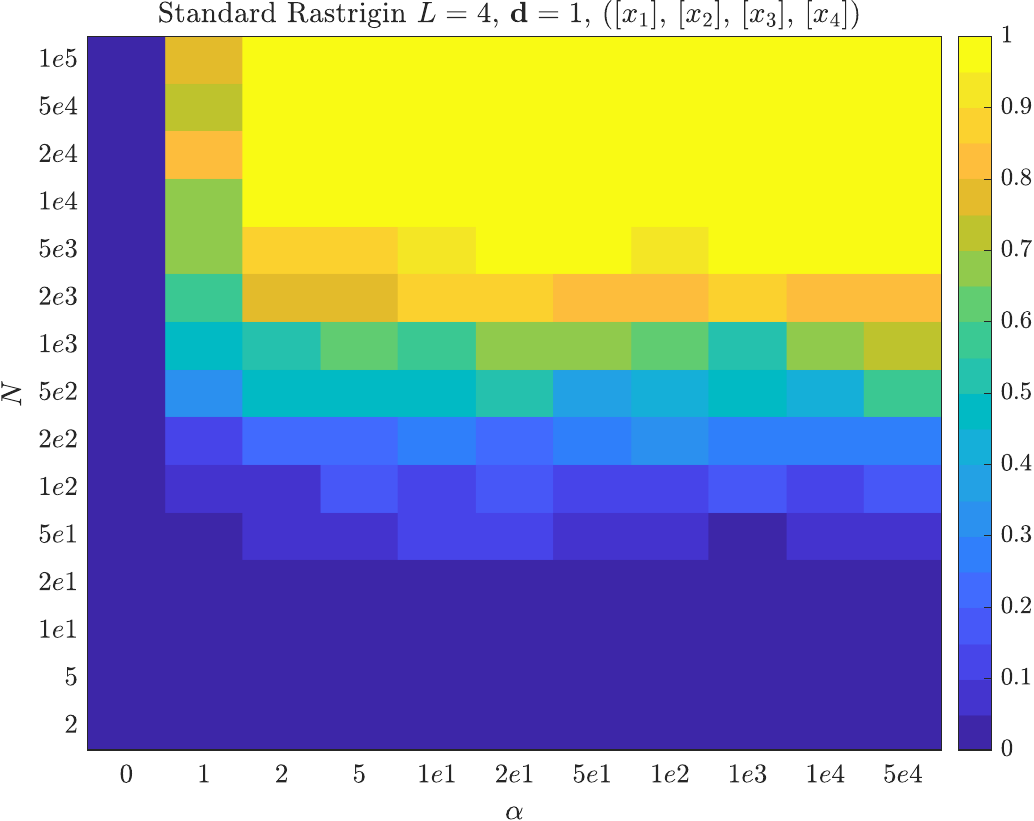}}
    \end{subfigure}
    \hfill
    % Subfigure 1
    \begin{subfigure}[b]{0.24\textwidth}
        \centering
        % Include your first heatmap image here
        \includegraphics[width=\textwidth]{Figures/ISOstandardD4d4.pdf}
        \caption{ $\mathbf{d}=1$.}
        \label{fig:heatmap_rastr_isod1}
    \end{subfigure}
    \hfill
    % Subfigure 2
       \begin{subfigure}[b]{0.24\textwidth}
        \centering
        % Include your third heatmap image here
        \includegraphics[width=\textwidth]{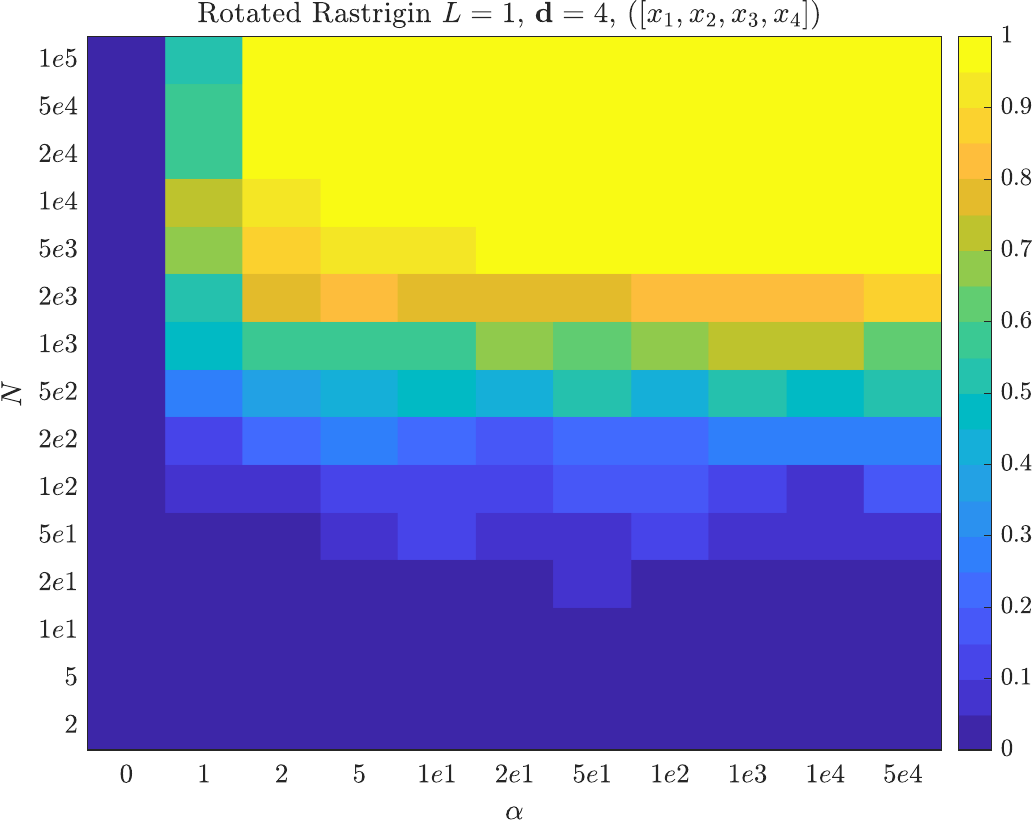}
        \caption{$\mathbf{d}=4$.}
        \label{fig:heatmap_rastr_isod4}
    \end{subfigure}
    \hfill
    % Empty slot 4
    \begin{subfigure}[b]{0.24\textwidth}
        \centering
        \phantom{\includegraphics[width=\textwidth]{Figures/ISOstandardD4d4.pdf}}
    \end{subfigure}

    \caption{
    Success rate of the CBO method with isotropic noise applied to different versions of the Rastrigin function in dimension $D=4$. Each heatmap shows the success rate as a function of $\alpha$ and $N$. Results are averaged over 100 runs.}
    \label{fig:heatmaps_cbo_ISO4D}
\end{figure}
It should be noted, however, that in the case of a separable function, the heatmap obtained with anisotropic CBO in Figure~\ref{fig:heatmap_rastr4D_anisod1} is better than that of its isotropic counterpart in Figure~\ref{fig:heatmap_rastr_isod4}. This is because the anisotropic formulation is able to exploit the structure of the function, adapting to the different directions. The isotropic formulation, on the other hand, precisely because of how it is constructed, fails to capture this characteristic and is therefore less effective in this case. This comparison is obtained after a proper tuning of the parameters for both formulations.

\subsubsection{Experiments for dimension \texorpdfstring{$D=8$}{8}}
\label{sec:numerics_intrd_D8}
We now repeat the experiments of the previous section in the higher dimension $D = 8$ using parameters
\begin{equation*}
    \lambda = 1, \sigma = 3, \quad T =10, \Delta t = 10^{-3}, \quad \rho_0 \sim \CN_D((0.5,\ldots,0.5)^T,4I_D).
\end{equation*} 
As before, we minimize the Rastrigin function \eqref{eq:rastrigin} and several of its rotated versions \eqref{eq:rotated_rastrigin} with anisotropic CBO.
Throughout this section, $(c_D)_k=2.5$ for all $k \in \{1,\dots,D\}$.

The results obtained in Figure~\ref{fig:heatmaps_cbo_ANISO8D} confirm the qualitative behavior already observed in lower dimension. In particular, the success rates are strongly influenced by the component structure of the objective function.
Indeed, for the standard Rastrigin function, which is fully additively separable, and for versions with small intrinsic dimension, there are relatively larger regions of success. In contrast, an increase in intrinsic dimensionality leads to a progressive deterioration in performance.

\begin{figure}[ht!]
    \centering
    % Subfigure 1
    \begin{subfigure}[b]{0.24\textwidth}
        \centering
        \includegraphics[width=\textwidth]{Figures/standardD8d4.pdf}
        \caption{$\mathbf{d}=1$}
        \label{fig:heatmap_rastr8D_anisod1}
    \end{subfigure}
    \hfill
    % Subfigure 2
       \begin{subfigure}[b]{0.24\textwidth}
        \centering
        \includegraphics[width=\textwidth]{Figures/rotatedD8d2.pdf}
        \caption{$\mathbf{d}=2$}
        \label{fig:heatmap_rastr8D_anisod2}
    \end{subfigure}
    \hfill
    % Subfigure 3
    \begin{subfigure}[b]{0.24\textwidth}
        \centering
        \includegraphics[width=\textwidth]{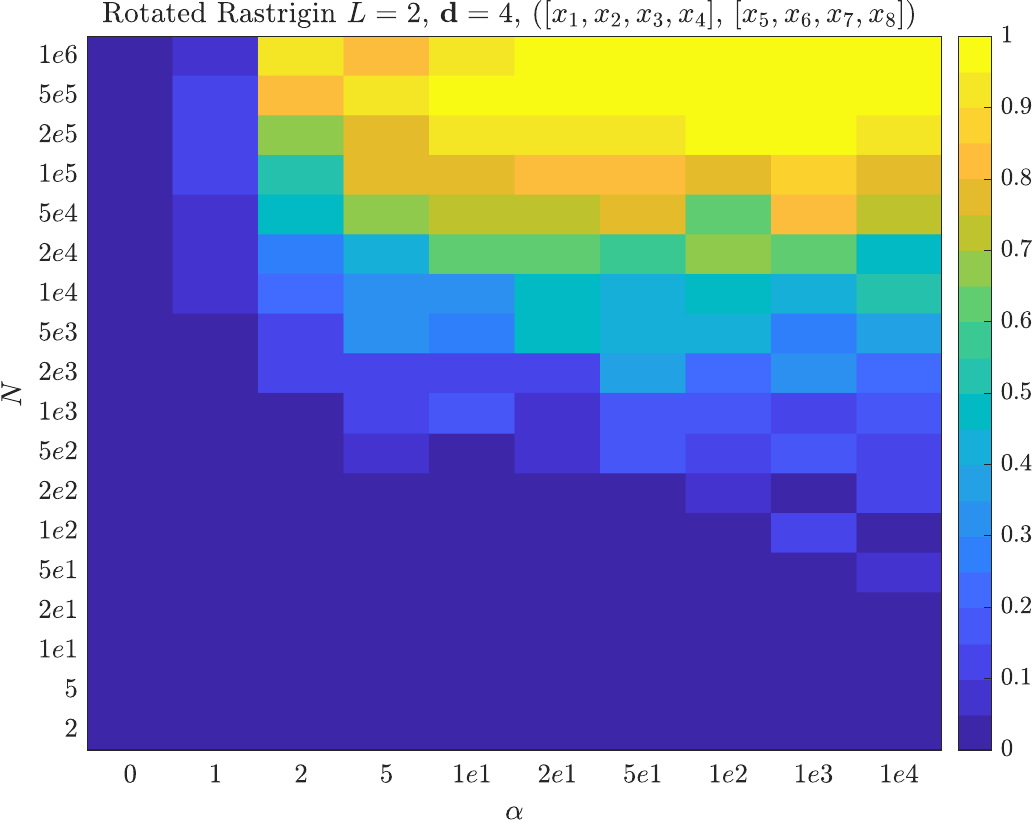}
        \caption{$\mathbf{d}=4$.}
        \label{fig:heatmap_rastr8D_anisod4}
    \end{subfigure}
    \hfill
    % Subfigure 4
    \begin{subfigure}[b]{0.24\textwidth}
        \centering
        \includegraphics[width=\textwidth]{Figures/rotatedD8d7.pdf}
        \caption{$\mathbf{d}= 7$.}
        \label{fig:heatmap_rastr8D_anisod7}
    \end{subfigure}

    \caption{
    Success rate of the CBO method with anisotropic noise applied to different versions of the Rastrigin function in dimension $D=8$. Each heatmap shows the success rate as a function of $\alpha$ and $N$. Results are averaged over 50 runs. The progressive deterioration of the success rate can be observed when comparing the images from Figure \ref{fig:heatmap_rastr8D_anisod1} to Figure \ref{fig:heatmap_rastr8D_anisod7}.}
    \label{fig:heatmaps_cbo_ANISO8D}
\end{figure}

\paragraph{Comparison for isotropic CBO.}
We now present again also the numerical results obtained in dimension $D=8$ for isotropic CBO.
The parameters are set to 
\begin{equation*}
    \lambda = 1, \sigma = 1.5, \quad T =10, \Delta t = 10^{-3}, \quad \rho_0 \sim \CN_D((0.5,\ldots,0.5)^T,4I_D).
\end{equation*} 
We minimize the Rastrigin function~\eqref{eq:rastrigin} and its rotations. Figure~\ref{fig:heatmaps_cbo_ISO8D} shows that the heatmaps for the standard and rotated Rastrigin functions remain indistinguishable, confirming that the performance of the isotropic CBO is not influenced by the separability properties of the objective function. Thus, for the isotropic CBO, success rates depend primarily on the choice of algorithmic parameters rather than on the structure of the objective function. 

\begin{figure}[ht!]
    \centering
    % Empty slot 1
    \begin{subfigure}[b]{0.24\textwidth}
        \centering
        \phantom{\includegraphics[width=\textwidth]{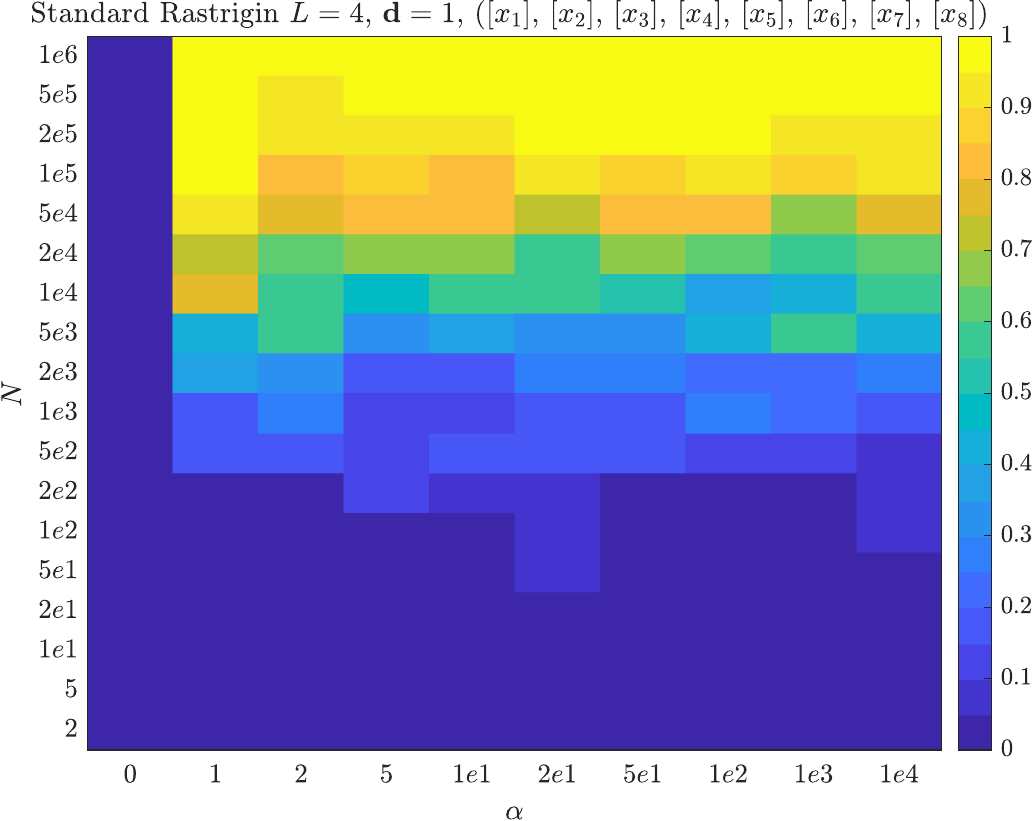}}
    \end{subfigure}
    \hfill
    % Subfigure 1
    \begin{subfigure}[b]{0.24\textwidth}
        \centering
        % Include your first heatmap image here
        \includegraphics[width=\textwidth]{Figures/ISOstandardD8d2.pdf}
        \caption{ $\mathbf{d}=1$.}
        \label{fig:heatmap_rastr_isod1}
    \end{subfigure}
    \hfill
    % Subfigure 2
       \begin{subfigure}[b]{0.24\textwidth}
        \centering
        % Include your third heatmap image here
        \includegraphics[width=\textwidth]{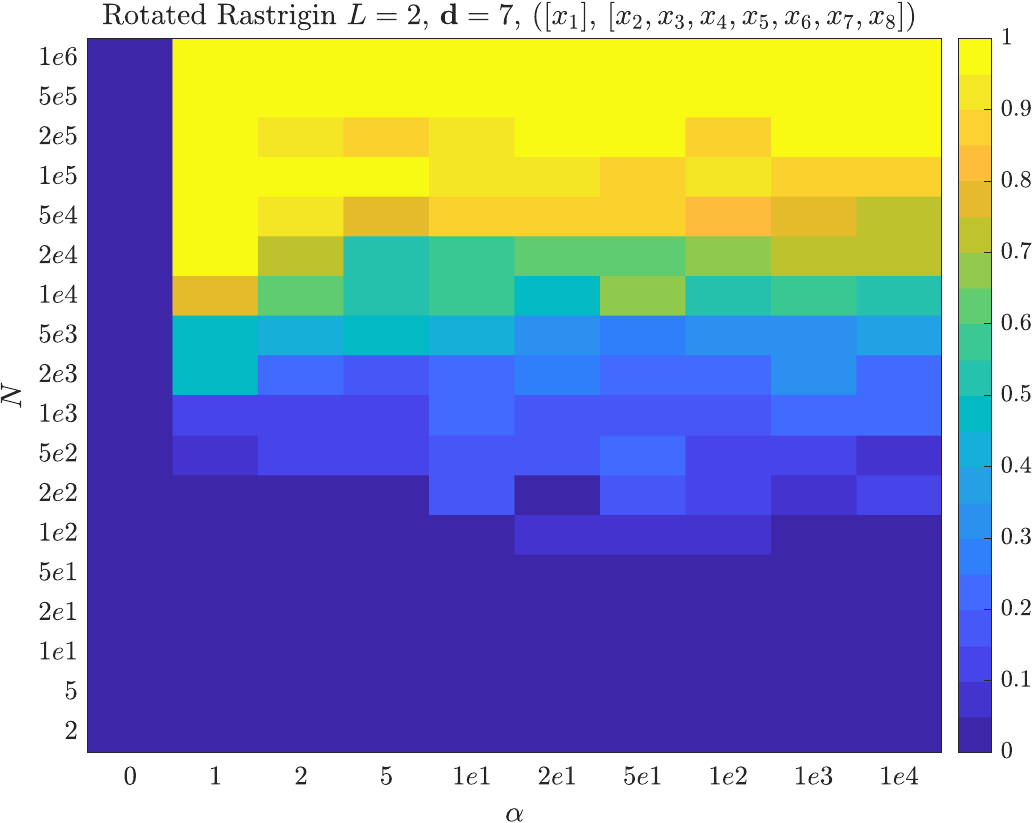}
        \caption{$\mathbf{d}=7$.}
        \label{fig:heatmap_rastr_isod2}
    \end{subfigure}
    \hfill
    % Empty slot 4
    \begin{subfigure}[b]{0.24\textwidth}
        \centering
        \phantom{\includegraphics[width=\textwidth]{Figures/ISOstandardD8d2.pdf}}
    \end{subfigure}

    \caption{
    Success rate of the CBO method with isotropic noise applied to different versions of the Rastrigin function in dimension $D=8$. Each heatmap shows the success rate as a function of $\alpha$ and $N$. Results are averaged over 50 runs.}
    \label{fig:heatmaps_cbo_ISO8D}
\end{figure}

\subsection{Influence of the complexity of the objective function within the separable components}
\label{sec:numerics_cD}

In this section, we investigate the influence of the complexity and the level of non-convexity of the objective function within the separable components on the performance of the anisotropic CBO algorithm.
To do so, we compare success rates for a rotated version of the Rastrigin function for different values of the parameters~$(c_D)_k$. We recall that the coefficients $(c_D)_k$ in \eqref{eq:rastrigin} regulate the magnitude of the oscillatory component of the energy landscape. We showcase experiments in $D=4$ and $D=8$.

\subsubsection{Experiments for dimension \texorpdfstring{$D=4$}{4}}
\label{sec:numerics_cD_D4}

The experiments in dimension $D=4$ are performed using the parameters
\begin{equation*}
    \lambda = 1, \sigma = \sqrt{1.6}, \quad T =5, \Delta t = 10^{-3}, \quad \rho_0 \sim \CN_D((3,\ldots,3)^T,20I_D).
\end{equation*}
In Figure~\ref{fig:heatmaps_cbo_ANISO4D_cD}, we report the success rates of the CBO method with anisotropic noise for the rotated Rastrigin function~\eqref{eq:rotated_rastrigin} with fixed intrinsic dimension but several varying coefficients $(c_D)_k$.

The success criterion is the same as in Section~\ref{sec:numerics_intrd_D4}, namely
\begin{equation*}
    |\CE(\conspoint{\rho_t})-\CE(\globmin)|<0.25.
\end{equation*}

The results show that an increase in the oscillation amplitudes $(c_D)_k$ generally leads to a decrease in the success rates of the anisotropic CBO algorithm. In particular, Figure~\ref{fig:heatmap_rastr4D_cD_anisod1} exhibits the best performance, while Figures~\ref{fig:heatmap_rastr4D_cD_anisod2} and~\ref{fig:heatmap_rastr4D_cD_anisod3} show a gradual reduction of the convergence region as the oscillatory amplitudes increase. As explained in Remark~\ref{rem:crucial_remark}, larger oscillations produce a more complex local geometry in the separable components, which affects the constants entering the expression for $\alphazerol$ in~\eqref{eq:alphazerol_simplified}.
Interestingly, Figure~\ref{fig:heatmap_rastr4D_cD_anisod4}, in which the coefficient $(c_D)_1=20$ (hence, substantially worse than in Figure~\ref{fig:heatmap_rastr4D_cD_anisod3}), while the remaining coefficients are $(c_D)_2=\cdots= (c_D)_4=10$, produces a heatmap comparable (in fact, visually indistinguishable) to that of Figure~\ref{fig:heatmap_rastr4D_cD_anisod3}. This, however, is in line with our theoretical observations.
The complexity of the objective function is determined by an interplay between the dimensionality of the component and its local complexity.
Here, despite the worsening of the complexity in the first component (by increasing the value of $(c_D)_1$ from $10$ to $20$), this does not affect the overall success rate, since the critical component is the second one (due to it having a significantly bigger dimensionality $d_2=3>1=d_1$).

Thus, the experiments show that the difficulty of the optimization problem does not depend solely on the intrinsic dimension~$\intrd$, but also on the geometric properties of the individual components~$\CE_\ell$. It follows that two objective functions with identical separability structure may exhibit different convergence behaviors, depending on the local complexity of their respective separable components.

\begin{figure}[ht!]
    \centering
    % Subfigure 1
    \begin{subfigure}[b]{0.24\textwidth}
        \centering
        \includegraphics[width=\textwidth]{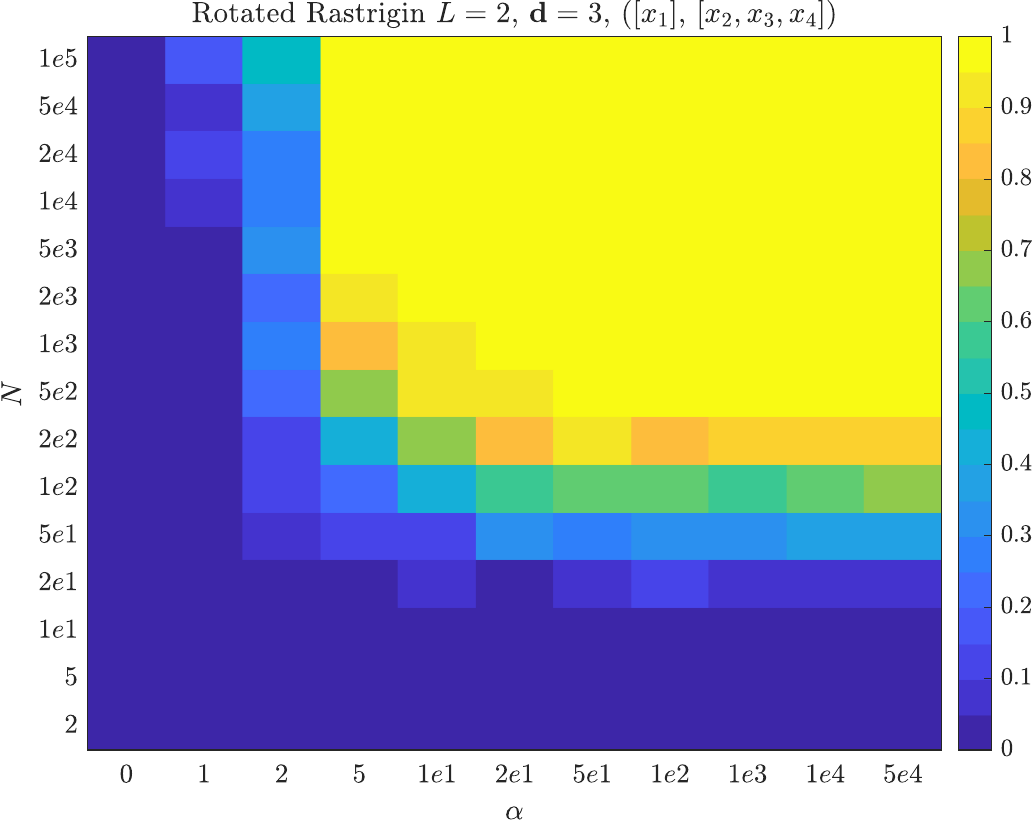}
        \caption{$\mathbf{d}=3$ with $(c_D)_1=10$ and $(c_D)_2=(c_D)_3=(c_D)_4=0.5$.}
        \label{fig:heatmap_rastr4D_cD_anisod1}
    \end{subfigure}
    \hfill
    % Subfigure 2
       \begin{subfigure}[b]{0.24\textwidth}
        \centering
        \includegraphics[width=\textwidth]{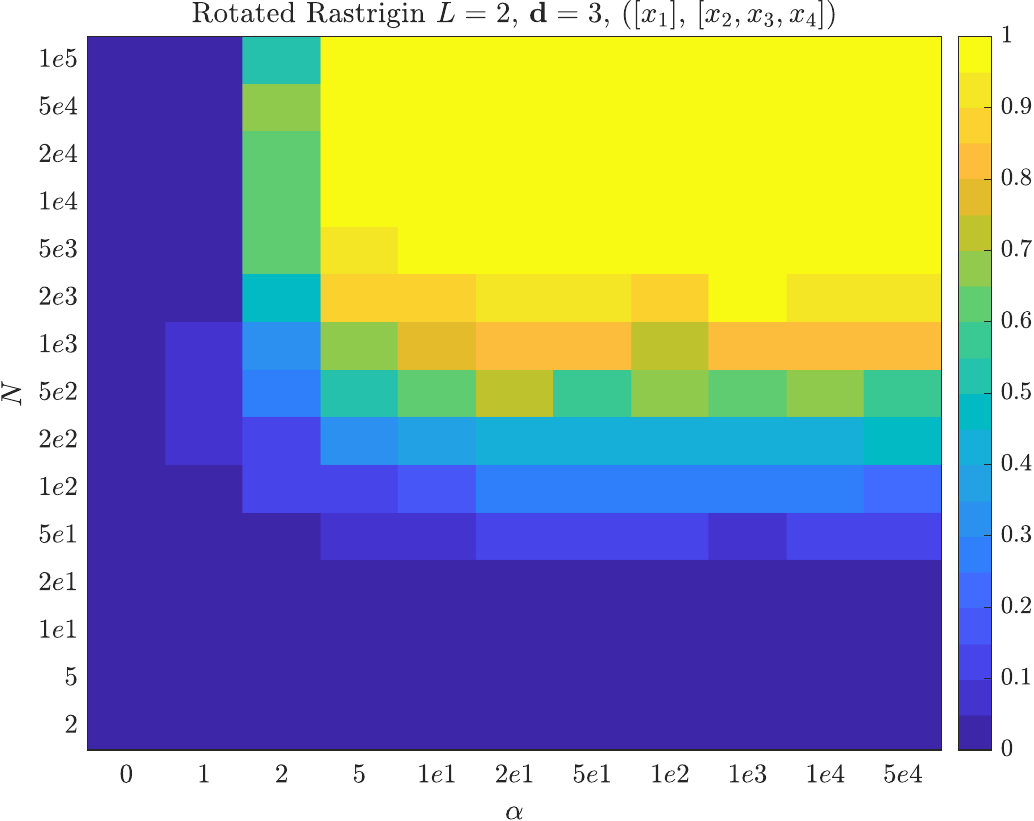}
        \caption{$\mathbf{d}=3$ with $(c_D)_1=10$ and $(c_D)_2=(c_D)_3=(c_D)_4=2.5$.}
        \label{fig:heatmap_rastr4D_cD_anisod2}
    \end{subfigure}
    \hfill
    % Subfigure 3
    \begin{subfigure}[b]{0.24\textwidth}
        \centering
        \includegraphics[width=\textwidth]{Figures/rotatedD4d3.pdf}
        \caption{$\mathbf{d}=3$ with $(c_D)_1=10$ and $(c_D)_2=(c_D)_3=(c_D)_4=10$.}
        \label{fig:heatmap_rastr4D_cD_anisod3}
    \end{subfigure}
    \hfill
    % Subfigure 4
    \begin{subfigure}[b]{0.24\textwidth}
        \centering
        \includegraphics[width=\textwidth]{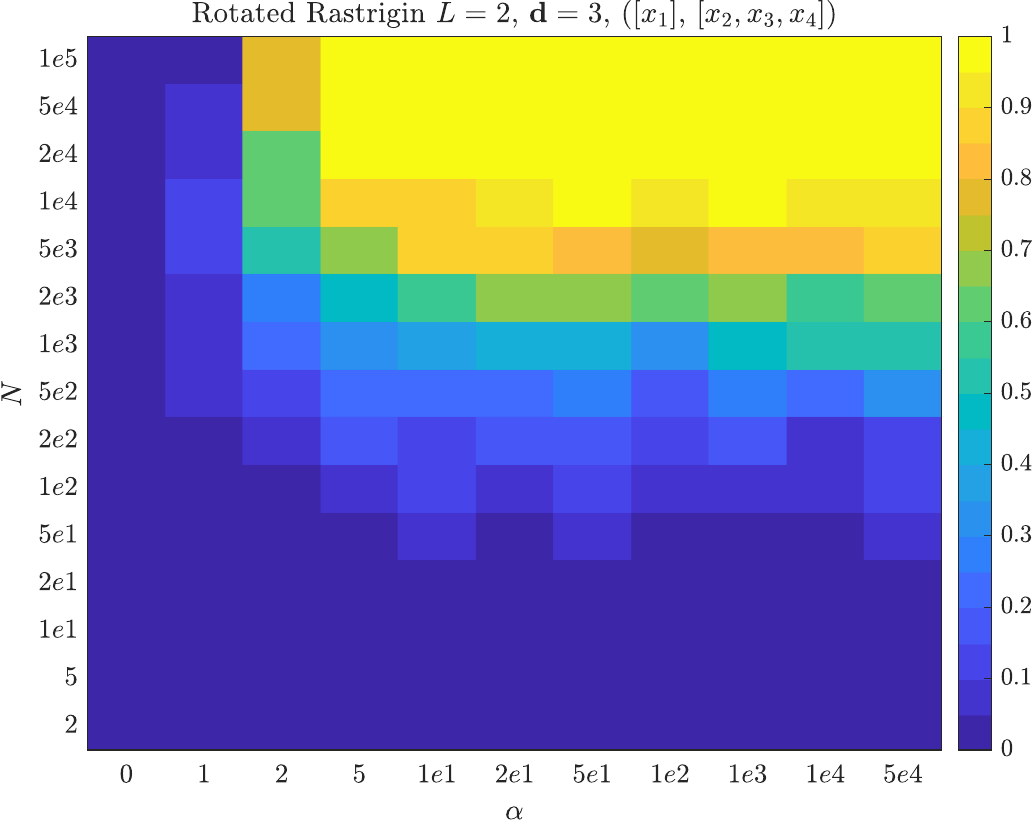}
        \caption{$\mathbf{d}=3$ with $(c_D)_1=20$ and $(c_D)_2=(c_D)_3=(c_D)_4=10$.}
        \label{fig:heatmap_rastr4D_cD_anisod4}
    \end{subfigure}

    \caption{
    Success rate of the CBO method with anisotropic noise applied to different versions of the Rastrigin function in dimension $D=4$. Each heatmap shows the success rate as a function of $\alpha$ and $N$. Results are averaged over 100 runs. The figures highlight a progressive reduction of the convergence region as the oscillatory amplitudes $(c_D)_k$, and thereby the complexity and level of non-convexity of the objective function, increase.}
    \label{fig:heatmaps_cbo_ANISO4D_cD}
\end{figure}

\subsubsection{Experiments for dimension \texorpdfstring{$D=8$}{8}}
\label{sec:numerics_cD_D8}
We now repeat the experiments of the previous section in the higher dimension $D=8$ using parameters
\begin{equation*}
    \lambda = 1, \sigma = 3, \quad T =10, \Delta t = 10^{-3}, \quad \rho_0 \sim \CN_D((0.5,\ldots,0.5)^T,4I_D).
\end{equation*}

As before, we consider rotated Rastrigin functions of the form~\eqref{eq:rotated_rastrigin} with fixed intrinsic dimension $\intrd=7$, while varying the oscillation amplitudes $(c_D)_k$.

The numerical results are reported in Figure~\ref{fig:heatmaps_cbo_ANISO8D_cD}.
As in the previous setting, the experiments show that the local complexity of the separable components has a strong influence on the success probability of anisotropic CBO. In particular, functions with smaller oscillatory parameters exhibit larger regions of successful convergence in the parameter space $(\alpha,N)$, whereas higher values of $(c_D)_k$ result in smaller regions of success. The behavior observed in Figure~\ref{fig:heatmap_rastr8D_anisod7_cD_setting_d} is consistent with the corresponding observations in dimension $D=4$ of Figure~\ref{fig:heatmap_rastr4D_cD_anisod4}.

\begin{figure}[ht!]
    \centering
    % Subfigure 1
    \begin{subfigure}[b]{0.24\textwidth}
        \centering
        \includegraphics[width=\textwidth]{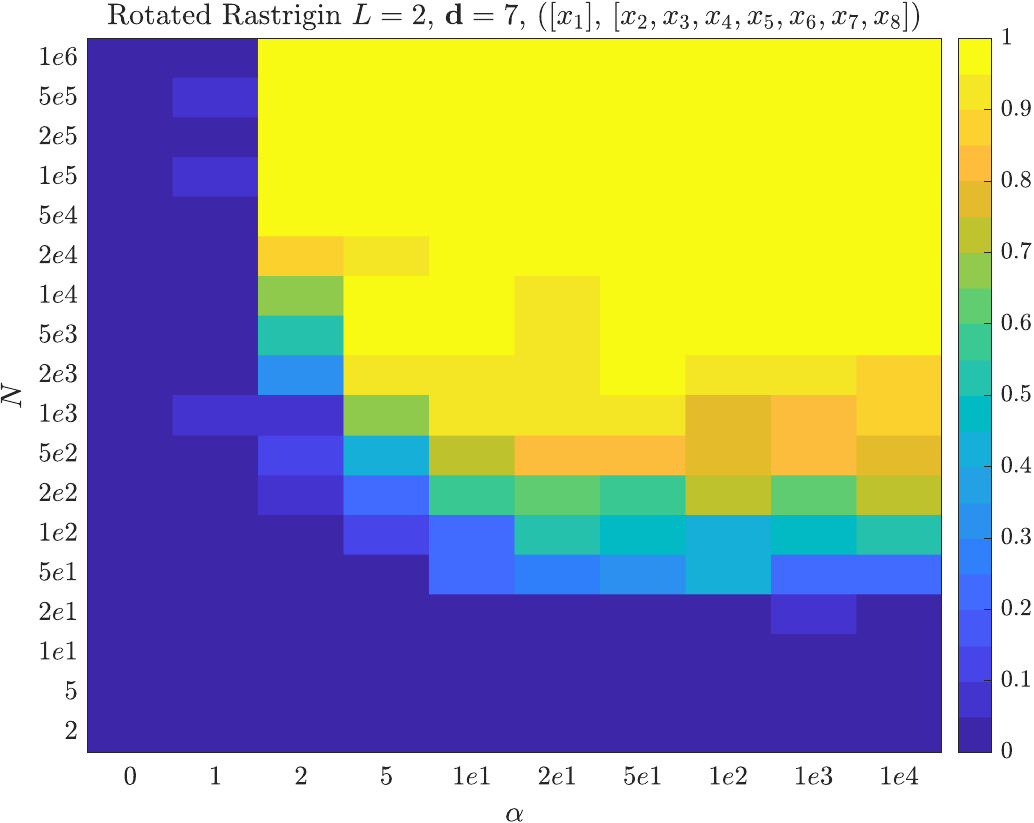}
        \caption{$\mathbf{d}=7$ with $(c_D)_1=2.5$ and $(c_D)_2=\cdots=(c_D)_8=0.5$.}
        \label{fig:heatmap_rastr8D_anisod7_cD_setting_a}
    \end{subfigure}
    \hfill
    % Subfigure 2
       \begin{subfigure}[b]{0.24\textwidth}
        \centering
        \includegraphics[width=\textwidth]{Figures/rotatedD8d7_cD_easier.pdf}
        \caption{$\mathbf{d}=7$ with $(c_D)_1=2.5$ and $(c_D)_2=\cdots=(c_D)_8=1.5$.}
        \label{fig:heatmap_rastr8D_anisod7_cD_setting_b}
    \end{subfigure}
    \hfill
    % Subfigure 3
    \begin{subfigure}[b]{0.24\textwidth}
        \centering
        \includegraphics[width=\textwidth]{Figures/rotatedD8d7.pdf}
        \caption{$\mathbf{d}=7$ with $(c_D)_1=2.5$ and $(c_D)_2=\cdots=(c_D)_8=2.5$.}
        \label{fig:heatmap_rastr8D_anisod7_cD_setting_c}
    \end{subfigure}
    \hfill
    % Subfigure 4
    \begin{subfigure}[b]{0.24\textwidth}
        \centering
        \includegraphics[width=\textwidth]{Figures/rotatedD8d7_cD_harder.pdf}
        \caption{$\mathbf{d}=7$ with $(c_D)_1=10$ and $(c_D)_2=\cdots=(c_D)_8=2.5$.}
        \label{fig:heatmap_rastr8D_anisod7_cD_setting_d}
    \end{subfigure}

    \caption{Success rate of the CBO method with anisotropic noise applied to different versions of the Rastrigin function in dimension $D=8$. Each heatmap shows the success rate as a function of $\alpha$ and $N$. Results are averaged over 50 runs. The figures highlight a progressive reduction of the convergence region as the oscillatory amplitudes $(c_D)_k$, and thereby the complexity and level of non-convexity of the objective function, increase.}
    \label{fig:heatmaps_cbo_ANISO8D_cD}
\end{figure}

%%%%%%%%%%%%%%%%%%%%%%%%%%%%%%%%%%%%%%%%%%%%%%%%%%
%%%%%%%%%% Section %%%%%%%%%%%%%%%%%%%%%%%%%%%%%%%
%%%%%%%%%%%%%%%%%%%%%%%%%%%%%%%%%%%%%%%%%%%%%%%%%%

\section{Conclusions}
\label{sec:conclusion}

In this work,
we have investigated, both theoretically and experimentally,
the performance of the anisotropic consensus-based optimization (CBO) method
in the scenario where the objective function possesses additional structure.
Specifically, we demonstrated that
anisotropic CBO automatically detects and exploits the additive separable structure of the objective
as described by Assumption~\ref{ass:add_sep},
i.e., when the objective can be decomposed additively into lower-dimensional components. 
Our analysis shows that this enables anisotropic CBO to provably mitigate the curse of dimensionality, which, for instance, isotropic CBO would have incurred into.
This is reflected by our main contributions through two aspects.
First, we demonstrate that the computational complexity depends only exponentially on the intrinsic dimension $\intrd$ of the objective function, rather than on the ambient dimension $D\gg \intrd$.
Second, we show that, rather than depending on a tractability condition of the full energy landscape, it depends only on tractability conditions on the lower-dimensional components, 
which allows for a more refined description of the objective function, as it captures the objective function landscape directly on the level of the individual components.
This is confirmed by our numerical experiments.

Our results highlight the effectiveness of anisotropic CBO for additively separable objective functions provided that there is sufficient alignment between the structure of the anisotropic noise in the algorithm and the separability structure of the objective itself.
This raises a crucial and interesting question for future research regarding an enhanced algorithm design of CBO:
How can we adapt the method such that it learns during the optimization how to most effectively explore the loss landscape by aligning on the fly the noise with the structure of the objective function.

%%%%%%%%%%%%%%%%%%%%%%%%%%%%%%%%%%%%%%%%%%%%%%%%%%
%%%%%%%%%% Acknowledgements %%%%%%%%%%%%%%%%%%%%%%
%%%%%%%%%%%%%%%%%%%%%%%%%%%%%%%%%%%%%%%%%%%%%%%%%%
\section*{Acknowledgements}
The authors, and in particular KR, would like to sincerely thank Timo Klock for discussions motivating and leading to this work. Moreover, the authors would like to thank Michael Herty, Lorenzo Pareschi, and Giacomo Borghi for valuable discussions during the preparation of the manuscript.

The work of SB is funded by the Deutsche Forschungsgemeinschaft (DFG, German Research Foundation) – 320021702/GRK2326 – Energy, Entropy, and Dissipative Dynamics (EDDy).
SB is a member of the INdAM Research National Group of Mathematical Physics (INdAM-GNFM).
KR acknowledges the financial support from the Technical University of Munich and the Munich Center for Machine Learning, where most of this work was done.
His work there has been funded by the German Federal Ministry of Education and Research and the Bavarian State Ministry for Science and the Arts.
The work of SV is supported by the European Union’s Horizon Europe research and innovation program under the Marie Sklodowska-Curie Doctoral Network Datahyking (Grant No. 101072546).
SV is a member of the INdAM Research National Group of Scientific Computing (INdAM-GNCS).

For the purpose of Open Access, the authors have applied a CC BY public copyright license to any Author Accepted Manuscript (AAM) version arising from this submission.

%%%%%%%%%%%%%%%%%%%%%%%%%%%%%%%%%%%%%%%%%%%%%%%%%%
%%%%%%%%%% Bibliography %%%%%%%%%%%%%%%%%%%%%%%%%%
%%%%%%%%%%%%%%%%%%%%%%%%%%%%%%%%%%%%%%%%%%%%%%%%%%
\bibliographystyle{abbrv}
\bibliography{bibliography.bib}

\end{document}